\documentclass[reqno]{amsart}
\usepackage[top=1.1in, bottom=1in, left=1in, right=1in]{geometry}
\usepackage{amsfonts}
\usepackage{amssymb}
\usepackage{amsthm}
\usepackage{amsmath}
\usepackage{mathrsfs}
\usepackage{color}
\usepackage{enumerate}
\usepackage[numbers,sort&compress]{natbib}
\usepackage{pdfsync}
\usepackage{esint}
\usepackage{graphicx}
\usepackage{float}
\usepackage{caption}
\usepackage{subfigure}

\allowdisplaybreaks

\usepackage{tikz} % Use TikZ for drawing
\usetikzlibrary{calc}
\usetikzlibrary{intersections}
\usetikzlibrary{patterns}
\usepackage[colorlinks,
linkcolor=black,
anchorcolor=blue,
citecolor=blue,
]{hyperref}
\usepackage{tikz} % Use TikZ for drawing
\usetikzlibrary{calc}
\usetikzlibrary{intersections}
\usetikzlibrary{patterns}
\makeatother

\numberwithin{equation}{section}

\newtheorem{theorem}{Theorem}[section]
\newtheorem{definition}[theorem]{Definition}

\newtheorem{lemma}[theorem]{Lemma}
\newtheorem{remark}[theorem]{Remark}

\newtheorem{proposition}[theorem]{Proposition}
\newtheorem{corollary}[theorem]{Corollary}
\numberwithin{equation}{section}

\newcommand{\T}{{\mathbb{T}}}
\newcommand{\Z}{{\mathbb{Z}}}

\newcommand{\ga}{\gamma}

\newcommand{\ru}{\mathring{R}_q}
\newcommand{\run}{\mathring{R}_n}
\newcommand{\runq}{\mathring{R}_{n,q}}
\newcommand{\Ru}{\mathring{R}}

\newcommand{\p}{\partial}

\renewcommand{\P}{\mathbb{P}}

\renewcommand{\div}{{\mathrm{div}}}
\newcommand{\curl}{{\mathrm{curl}}}

\newcommand{\om}{\Omega}
\newcommand{\omw}{\widetilde{w}}

\newcommand{\la}{\lambda_q}
\newcommand{\laq}{\lambda_{q+1}}

\newcommand{\dq}{\delta_q}
\newcommand{\dqq}{\delta_{q+1}}
\newcommand{\dqqq}{\delta_{q+2}}

\newcommand{\Lo}{L}

\newcommand{\ft}{\mathcal{F}_{t}}
\newcommand{\F}{\mathcal{F}}

\newcommand{\R}{{\mathbb R}}

\def\a{{\alpha}}

\def\wt{\widetilde}
\def\9{{\infty}}

\def\bbr{{\mathbb{R}}}
\def\bbn{{\mathbb{N}}}
\def\bbp{{\mathbb{P}}}
\def\bbq{{\mathbb{Q}}}

\def\({\left(}
\def\){\right)}	
\def\[{\left[}	
\def\]{\right]}	

\begin{document}
	
\title[] {Non-Leray-Hopf solutions to 3D  stochastic hyper-viscous Navier-stokes equations: beyond the Lions exponents}

\author{Wenping Cao}
\address{School of Mathematical Sciences, Shanghai Jiao Tong University, China.}
\email[Wenping Cao]{caowenping1@sjtu.edu.cn}
\thanks{}

\author{Zirong Zeng}
\address{School of Mathematics, Nanjing University of Aeronautics and Astronautics, China.}
\email[Zirong Zeng]{beckzzr@nuaa.edu.cn}
\thanks{}

\author{Deng Zhang}
\address{School of Mathematical Sciences, CMA-Shanghai, Shanghai Jiao Tong University, China.}
\email[Deng Zhang]{dzhang@sjtu.edu.cn}
\thanks{}



\begin{abstract}
	We consider the 3D stochastic Navier-Stokes equations (NSE) on torus where the viscosity exponent can be larger than the Lions exponent 5/4. 
	 For arbitrarily prescribed divergence-free initial data in $L^{2}_x$, 
  we construct infinitely many probabilistically strong and analytically weak solutions in the class $L^{r}_{\om}L_{t}^{\gamma}W_{x}^{s,p}$, where $r\geq1$ and $(s, \gamma, p)$ lie in two supercritical regimes with respect to the Lady\v{z}henskaya-Prodi-Serrin (LPS) criteria. 
 It shows that 
  even in the high viscosity regime beyond the Lions exponent, 
  though solutions are unique in the Leray-Hopf class, 
  the uniqueness fails in the mixed Lebesgue spaces and, actually, 
  there exist infinitely manly 
  non-Leray-Hopf solutions which can be very close to the Leray-Hopf solutions.  
  Furthermore,  
  we prove the vanishing noise limit result, 
  which relates together 
  the  stochastic solutions 
  and the deterministic solutions 
  constructed by Buckmaster-Vicol \cite{BV-NS} 
  and the recent work  \cite{LQZZ-NSE}.
\end{abstract}

\maketitle
%\tableofcontents

\section{Introduction and main results}\label{Introduction and main results}

\subsection{Background} \label{Subsec-intro}
 We are concerned with the  3-D stochastic  Navier-Stokes equations (NSE) with high viscosity
 on the torus $\T^{3}:=\[-\pi,\pi\]^{3}$,
\begin{equation}\label{equa-SNSE-High}
	\left\{\begin{array}{lc}
		\mathrm{d}u+(\nu(-\Delta)^{\a}u+(u\cdot \nabla)u +\nabla P) \mathrm{d}t =\mathrm{d}W_{t},\\
		\div u=0,\\
		u(0)=u_{0},
	\end{array}\right.
\end{equation}
where $u=(u_{1},u_{2},u_{3})^{\top}(t,x)\in \R^{3}$ represents the velocity field,  $P=P(t,x)\in \R$
is the  pressure of the fluid,  $ \nu \geq 0$ is the viscous  coefficient, 
$\a\in[1,2)$, and  $(-\Delta)^{\a}$ is the fractional Laplacian defined via Fourier transform on the flat torus
\begin{align*}
	\mathscr{F}\((-\Delta)^{\a}u\)(\xi)=|\xi|^{2\a}\mathscr{F} (u)(\xi),\quad\xi\in \Z^{3}.
\end{align*}
The noise $W$ is a $GG^{*}$-Wiener processe on a probability space 
$\(\om, \F, (\F_{t}), \P\)$, 
where $G$ is a Hilbert-Schmidit operator from some Hilbert space $U$ to $H_{\sigma}^{4}$, 
with $H_{\sigma}^{4}$ being the subspace of all mean-free and divergence-free functions in $H_{x}^{4}$.  
In the classical case where $\a=1$, system \eqref{equa-SNSE-High} reduces to the classical incompressible NSE. 
When $\nu=0$, it 
becomes the incompressible Euler equations. 

Since the poineering work by Leray \citep{Leray} in 1934, it is well-known that the 3-D NSE has global solutions in the energy class $C_{w}(0,\infty; L_{x}^{2})\cap L(0,\infty;H_{x}^{1})$.  
This class of solutions is now referred to as Leray-Hopf solutions, 
due to the important contributions by Hopf 
\citep{Hopf} in the the case of bounded domains. 
The theory of stochastic Leray-Hopf solutions has been well developed 
in terms of martingale solutions or probabilistically weak solutions. 
We refer to  \cite{FG95} for the case of bounded domains, \cite{MR05} for the case of whole space, 
and \cite{Brz13} in general unbounded domains. 
The uniqueness of Leray-Hopf solutions, however,  remains a challenging open problem.

On the positive side, due to the classical result of Lions \citep{Lions}, 
the 3-D hyper-viscous NSE is well-posed in the Leray-Hopf class when the viscosity exponent is no less than $5/4$, 
i.e., $\alpha \geq 5/4$. 
This uniqueness phenomenon also exhibits 
for the stochastic NSE in the high viscosity regime, 
where the stochastic Leray-Hopf solutions 
satisfy the 
energy inequality
\begin{align}
 \mathbb{E} \|u(t)\|_{L_{x}^{2}}^{2}+2\int_{0}^{t}\mathbb{E}\|u(s)\|_{\dot{H}_{x}^{\a}}^{2}\mathrm{d}s\leq \|u_{0}\|_{L_{x}^{2}}^{2} 
 + Tr(GG^{*})  t
\end{align} 
for any $t\in[0,T]$, 
see \cite{BJLZ23}. 
These results, actually, reveal that the high viscosity tends to stablize the system. 
Moreover, 
in view of the  Lady\v{z}henskaya-Prodi-Serrin (LPS)  criteria, 
weak solutions in 
the scaling invariant space $L_{t}^{\gamma}L_{x}^{p}$ when ${2}/{\gamma} + {3}/{p} = 1$ 
are automatically unique Leray-Hopf solutions, see \citep{CL-NSE1}.

On the flexible side, 
in the  breakthrough work by Buckmaster and Vicol \citep{BV-NS}, 
the non-uniqueness of weak/distributional solutions to the 3D NSE was proved in the space $C_{t}L_{x}^{2}$. The method of \citep{BV-NS} is based on convex integration, which was initiated by De Lellis and Sz\'{e}kelyhidi \citep{De Lellis-Elur} in the context of Euler equations and dates back to Nash \citep{Nash}. Afterwards,  Luo-Titi \cite{LT-hyperviscous NS} and Buckmaster-Colombo-Vicol \cite{BCV} proved the non-uniqueness of $C_{t}L^2_x$ weak solutions for the 3D hyper-dissipative NSE up to the Lions exponent, i.e. $\a<5/4$.  In particular,  the Lions exponent $5/4$ is the sharp threshold for the uniqueness/non-uniqueness of weak solutions to hyper-dissipative NSE in $C_{t}L_{x}^{2}$.  

Moreover,  Cheskidov-Luo proved the  sharp non-uniqueness of weak solutions to the 3D NSE near one LPS endpoint space  $L_{t}^{\gamma}L_{x}^{\infty}$ with $1\leq \gamma < 2$ \citep{CL-NSE1}, and to the 2D NSE near another LPS endpoint space $L_{t}^{\infty}L_{x}^{p}$  
with $1\leq p < 2$ \citep{CL-NSE2}. 
  For hyper-dissipative NSE and MHD equations with viscosity beyond the Lions exponent $5/4$, the sharp non-uniqueness at both LPS endpoints was recently proved in \cite{LQZZ-NSE,ZZL-MHD}.
  
In the stochastic case, there also have been  several progresses on the  non-uniqueness problem by convex integrations. A stochastic counterpart of convex integration scheme has been developed by Hofmanov\'{a}-Zhu-Zhu \citep{HZZ-1} to prove the  non-uniqueness in law 
of weak solutions to 3D stochastic NSE.    
The existence of global-in-time non-unique probabilistically strong  solutions in $L^q_t L^2_x$ with  $1\leq q<\infty$  was obtained in \citep{HZZ-3}.  
Moreover, Yamazaki 
proved the non-uniqueness in law 
for the stochastic NSE 
in the 3-D case up to the Lions exponent  \cite{Y22.3}, 
and in the 2-D case with diffusion weaker than a full Laplacian. 
In \citep{CDZ},  
Chen-Dong-Zhu proved the sharp non-uniqueness at one endpoint space $L_{t}^{\gamma}L_{x}^{\infty}$ with $1\leq \gamma <2$. 
Very recently, 
for any prescribed ﬁnite energy divergence-free initial data, inﬁnitely many probabilistically strong solutions  has been constructed 
for the 3D  stochastic NSE in the continuous energy class $C([0, \infty);L^2)$ 
by Cheskidov and the last two named authers \citep{Cheskidov-Zeng-Zhang}. 
We also would like to refer to 
\cite{MS24} 
for stochastic transport equations, 
and \cite{CLZ,Y23,Y23.2} 
stochastic MHD equations.

\medskip 
The aim of this paper is to  understand  the non-uniqueness problem for the 3D hyper-viscous  stochastic NSE, particularly, in the high viscosity regime beyond the Lions exponent. 
We prove that the non-uniqueness phenomenon still exhibits in this high viscous regime. More precisely, 
we address the following two problems:  
\begin{enumerate}[(i)]
   \item Construction of non-unique   probabilistically strong and  analytically weak  solutions to the stochastic hyper-viscous   NSE \eqref{equa-SNSE-High} with the viscosity above the Lions exponent $ 5/4$.
   
   \item Vanishing noise limit 
   of the constructed stochastic solutions, 
   in relation to the convex integration solutions constructed for the deterministic NSE. 
\end{enumerate}

The main results of the present work are formulated in the following.

\subsection{Formulation of main results} 

Let us first present the definition of solutions taken in the probabilistically strong and analytically
weak sense.
\begin{definition}  \label{Def-Sol}
 Let $T\in [0, \infty)$. 
 We say that $ u$ is a probabilistically strong
	and analytically weak solution to the stochastic hyper-viscous NSE \eqref{equa-SNSE-High} with the initial datum $u_{0} \in L_{x}^{2}$,  if $  \textbf{P}-a.s.$ the following holds:
	\begin{enumerate}[(i)]
	\item $u  \in  L^{2}([0, T]; L_{x}^{2})$ ;
	\item for all $t\in [0, T]$, $u(t, \cdot)$ is divergence-free and has zero spatial mean;
	\item for any divergence-free test function $\phi\in C_{0}^{\infty}(\T^{3})$ and any $t\in [0, T]$,
	\begin{align*}
		&\langle u(t),\phi\rangle = \langle u_{0} ,\phi\rangle  +\int_{0}^{t}\langle  -\nu(-\Delta)^{\a}u(s), \phi\rangle -\langle (u(s)\cdot\nabla)u(s),\phi\rangle   \mathrm{d}s
	+ \langle W(t), \phi \rangle.
	\end{align*} 	
\end{enumerate}
\end{definition}

We note that  the hyper-viscous deterministic NSE has the scaling-invariant space $L_{t}^{\gamma}W_{x}^{s,p}$ with
 \begin{align}\label{LPS-1}
 	\frac{2\a}{\gamma} + \frac{3}{p} = 2\a-1+s,
 \end{align}
 under the scaling 
\begin{equation}
	u(t,x)\mapsto \lambda^{2\a-1} u(\lambda^{2\a}t, \lambda x),
	\quad P(t,x)\mapsto \lambda^{4\a-2} P(\lambda^{2\a}t, \lambda x).
\end{equation}
In particular, in the case $s=0$, the mixed Lebesgue space $L_{t}^{\gamma}L_{x}^{p}$ is scaling-critical for the  classical NSE, and \eqref{LPS-1} reduces exactly to
 the celebrated LPS criterion
\begin{align}\label{LPS-c}
	\frac{2}{\gamma} + \frac{3}{p} = 1.
\end{align}

In the present work, 
we consider  the following two   scaling-supercritical  regimes 
\begin{align}
	&\mathcal{S}_{1}:=\left\{(s,\gamma,p)\in [0, 3)\times[1, \infty]\times[1, \infty]:0\leq s< \frac{4\a-5}{\gamma}+\frac{3}{p}+1-2\a\right\},\ \ \alpha\in [\frac 54, 2), \label{S1}\\
	&\mathcal{S}_{2}:=\left\{(s,\gamma,p)\in [0, 3)\times[1, \infty]\times[1, \infty]:0\leq s<\frac{ 2\a}{\gamma}+\frac{2\a-2}{p}+1-2\a\right\},\ \ \alpha\in [1, 2). \label{S2}
\end{align} 
Note that the borderlines of $\mathcal{S}_{1}$ and $\mathcal{S}_{2}$ contain two endpoints of the scaling condition \eqref{LPS-1}.   See Figure \ref{fig:my_label} below.

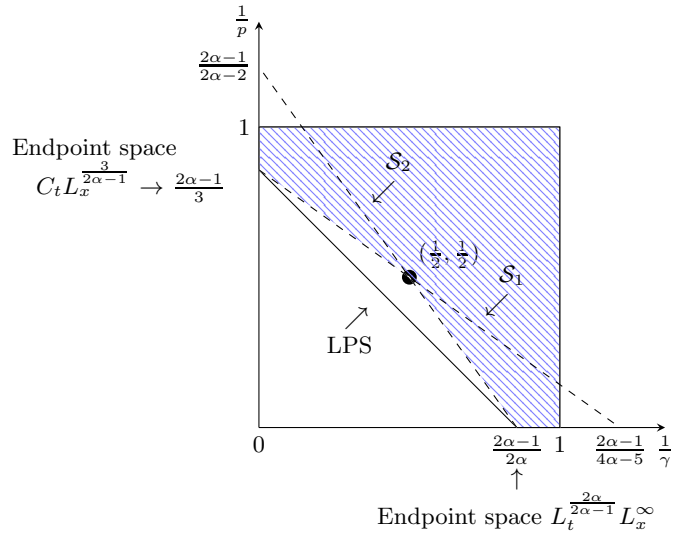
\begin{figure}[H]\label{fig}
		\centering
		\begin{tikzpicture}[scale=4,>=stealth,align=center]
			%\footnotesize
			\small
			
			%\node[circle, draw=red, line width=2pt, dashed, draw opacity=0.5] (a) at (0,0){A};
			
			\draw[->] (0,0) --node[pos=1,left] {$\frac{1}{p}$} node[pos=0,below]{$0$} (0,1.35);
			\draw[->] (0,0) --node[pos=1,below] {$\frac{1}{\gamma}$} (1.35,0);
			%\draw[pattern = north east lines, pattern color = red!25] (0,0) -- (0,4/5) -- (4/3,0) -- cycle;
			%\draw[blue,thick] (0,4/5) -- (4/3,0);
			%\node at (1/3,1/5) {Uniqueness \\ by LPS};
			%\node[name = A] at (1/2,1/2) {$(\frac{1}{2}, \frac{1}{2})$};
			
			%\draw (0,1) --node[pos=0,left] {$1$} (1,1);
			%\draw (1,0) --node[pos=0,below] {$1$} (1,1);	
			%\draw (0,0.7) --node[pos=0,left] {$\frac{2\alpha -1}{3} $} node[pos=1,below] {$\frac{2\alpha-1}{2\alpha}$} (0.7,0);
			
			\coordinate [label=below:$1$] (gamma1) at (1,0);
			\coordinate [label=left:$1$] (p1) at (0,1);
			\coordinate (11) at (1,1);
			%\draw [name path = l1] (gamma1) -- (11);
			%\draw [name path = l2] (p1) -- (11);

			\coordinate [label=left:\text{Endpoint space\quad\quad\quad }\\ $C_{t}L_{x}^{\frac{3}{2\alpha-1}}$ $\rightarrow$ $\frac{2\alpha -1}{3} $
			] (A) at (0,6/7); % (1.2-0.5) / 0.5 = 0.5 / ?, ? = 0.25/0.7 = 5/14, 5/14+ 1/2 = 6/7
			\coordinate [label=below:$\frac{2\alpha-1}{2\alpha}$ \\ $\uparrow$ \\
			\text{Endpoint space $L_t^{\frac{2\alpha}{2\alpha-1}}L_x^{\infty}$}] (B) at (6/7,0);
			\draw (A) --node[pos=0.5,below left]{$\;\nearrow$\\LPS\;} (B);
			\draw (p1) -- (11);
			\draw (11) -- (gamma1);
			% \coordinate [label=below:$ \frac{2\alpha-1}{4\alpha-5} $] (A1) at (1.2,0);
			\coordinate [label=left:$ \frac{2\alpha-1}{2\alpha-2} $] (B1) at (0,1.2);
			
			%XXXXXXXXXXXXXXXXXXXXXXXXXXXXXXXXXX
			%\coordinate [label=right:$(\frac{1}{2}, \frac{1}{2})$](AA) at (1/2,1/2);
			\coordinate (C1) at (0.5,0.5);
			\fill (C1) circle [radius=0.7pt];
			\node[above right] at (C1) {$(\frac{1}{2}, \frac{1}{2})$};
			\coordinate [label=below:$ \frac{2\alpha-1}{4\alpha-5} $] (A1) at (1.2,0);
			%XXXXXXXXXXXXXXXXXXXXXXXXXXX
			
			%\path [name intersections = {of= l1 and l3}] coordinate  (A2) at (intersection-1);
			%\path [name intersections = {of= l2 and l4}] coordinate  (B2) at (intersection-1);
			
			% \coordinate (c) at (1/7,1);
			\draw[white, pattern =north west lines, pattern color = blue!50] (A) -- (0,1) -- (11) -- (gamma1) -- (B) -- (1/2,1/2) -- cycle;
			\draw (A) -- (0,1) -- (11) -- (gamma1) -- (B) ;%??????????
			\draw[dashed, name path = l4] (B) --node[pos=0.6,above right] { \; $\mathcal{S}_{2}$\\ $\swarrow$ \;} (B1);
			
			%XXXXXXXXXXXXXXXXXXXXXXXXXXXXXX
			\draw[dashed, name path = l4] (A) --node[pos=0.6,above right] { \; $\mathcal{S}_{1}$\\ $\swarrow$ \;} (A1);
			%XXXXXXXXXXXXXXXXXXXXXXXXXXXXXXX
			
		\end{tikzpicture}
		\caption{The case $\alpha \in [\frac{5}{4},2), s=0$.}
		\label{fig:my_label}
	\end{figure}

 Since  solutions to stochastic NSE conserve the mean,  for the sake of simplicity, we mainly focus on the case where  initial data are mean-free.  We use $L_{\sigma}^{2}$ below to denote the subspace of all divergence-free and mean-free functions in $L_{x}^{2}$.

The main result of this paper is fomulated as follows, which  provides the first example of  probabilistically strong and non-Leray-Hopf solutions to the  stochastic hyper-viscous NSE \eqref{equa-SNSE-High} with  $\a\geq 5/4$.

\begin{theorem}   \label{Thm-Nonuniq-Hyper}
Let $ \widetilde{u}\in C_{0}^{\infty}([0,T]\times\T^{3})$ be any 
{deterministic} divergence-free and mean-free smooth vector fields with $ \widetilde{u}(0)=0$.
 Then, for any  $r\geq 1$, $\epsilon>0 $, for any $(s,p,\gamma)\in \mathcal{S}_{1}$ when $\a\in [\frac{5}{4},2)$ or  $(s,p,\gamma)\in \mathcal{S}_{2}$ when $\a\in [1,2)$,  and  for any  deterministic initial condition $ u_{0} \in L_{\sigma}^{2} $, there exists a solution $ u$ to (\ref{equa-SNSE-High}),
in the sense of Definition \ref{Def-Sol},
such that
	\begin{equation}\label{1.4}
		u-z\in \Lo^{2r}_{\om} L_{t}^{2}L_{x}^{2}\cap\Lo^{r}_{\om} L_{t}^{1}L_{x}^{2}\cap\Lo^{r}_{\om}L_{t}^{\ga}W_{x}^{s,p}.
	\end{equation}
	Moreover,  $ u $ is close to $ \widetilde{u}+z:$
	\begin{equation}\label{1.5}
		\|u-(\widetilde{u}+z)\|_{ \Lo^{r}_{\om} L_{t}^{1}L_{x}^{2}}+	\|u-(\widetilde{u}+z)\|_{\Lo^{r}_{\om}L_{t}^{\ga}W_{x}^{s,p}}<\epsilon.
	\end{equation}
\end{theorem}
As a consequence, we have the non-uniqueness of probabilistically strong solutions, 
which are non-Leray-Hopf, to the stochastic hyper-viscous NSE even in the viscosity regime above the Lions exponent. 

\begin{corollary}\label{coro}
    For any deterministic initial data in $L_{\sigma}^{2}$, there exist infinitely many non-Leray-Hopf solutions to \eqref{equa-SNSE-High} in the space $L_{t}^{\gamma}W_{x}^{s,p}$ where $(s,\gamma,p)\in  \mathcal{S}_{1}\cup \mathcal{S}_{2}$. 
\end{corollary}

\begin{remark}
(i). 
To the best of our knowledge, 
Theorem \ref{Thm-Nonuniq-Hyper} and Corollary \ref{coro} provide the first examples of non-Leray-Hopf solutions 
to the 3-D stochastic NSE in the high 
viscosity regime beyond the Lions exponent. 
Actually, 
 when $\alpha\geq 5/4$, 
it was proved in \cite{BJLZ23} that there exist unique Leray-Hopf solutions to \eqref{equa-SNSE-High} in the class $C_tL^2_x \cap L^2_t H^1_x$. 
In contrast, 
Theorem \ref{Thm-Nonuniq-Hyper} shows that, 
outside the Leray-Hopf class, 
there indeed exist infinitely many different 
non-Leray-Hopf solutions. 

It should be mentioned that 
contructing non-Leray-Hopf solutions 
in the high viscosity regime  
 is indeed quite hard. 
Our proof utilizes building blocks in the 
convex integration schemes 
which feature both the spatial and temporal 
intermittencies that are strong enough 
to control the high viscosity even beyond the Lions exponent. 
More detailed explanations can be found in 
Subsubsection \ref{Subsub-proof} below. 

(ii). We note that the boundaries of $\mathcal{S}_{1}$ and $\mathcal{S}_{2}$ contain two endpoints of the LPS condition. 
Hence, Corollary  \ref{coro} shows that 
the uniqueness fails 
in any LPS supercritical space close to the 
two endpoint spaces. 

(iii). 
The non-uniqueness in Corollary  \ref{coro} is sharp near the LPS endpoint 
space $C_tL^2_x$ when $\alpha=5/4$.  
This can be seen from the following uniqueness result, which 
resembles the classical LPS criterion 
and can be proved from Theorem 1.5 of 
\cite{CLZ} when the magnetic field vanishes: 

\medskip 
\paragraph{\bf Uniqueness:}
Let $(s,\gamma,p)$ satisfy  (\ref{LPS-1}) with $\a\geq 1$, $2\leq \gamma\leq \infty$, $1\leq p\leq \infty$, $s\geq 0$ and $0\leq \frac{1}{p}-\frac{s}{3}\leq \frac{1}{2}$. 
Let $u\in L^{\gamma}(0,T;W_{x}^{s,p})$ 
if $1\leq \gamma \leq \infty$, 
or $u\in C([0,T];W_{x}^{s,p})$ 
if $\gamma =\infty$, $\textbf{P}-a.s.$ 
If $u$ is a weak solution to (\ref{equa-SNSE-High}), then $u\in  C_{w}([0,T];L_{x}^{2})\cap L^{2}(0,T; \dot{H}_{x}^{\a})$ and  $u$ is unique, $\textbf{P}-a.s.$

\end{remark}

\medskip
Our next result is to provide a relationship, via the vanishing noise limit, between stochastic and deterministic solutions to the hyper-viscous NSE constructed by convex integrations.

Let $\mathscr{A}_{NSE}^{\widetilde{\beta}}$, $\wt \beta>0$, denote the set of all mean-free {and divergence-free} solutions $u\in H_{t,x}^{\widetilde{\beta}}$ to the deterministic NSE
	\begin{eqnarray}\label{deterministic-NSE}
	\left\{\begin{array}{lc}
	\p_{t}u+\nu(-\Delta)^{\a}u+(u\cdot \nabla)u+\nabla P =0,\\
		\div u=0,\\
		u(0)=u_{0},
	\end{array}\right.
\end{eqnarray}
where $u_{0}\in L_{\sigma}^{2}$, $\a\in[1,2)$. Let $\mathscr{S}^{\wt \beta}_{SNSE,r}$  denote the accumulation points in the space $\Lo^{r}_{\om}L_{t}^{2}L_{x}^{2}$ of global solutions $u^{(\epsilon_{n})}$
to the stochastic hyper-viscous NSE 
(\ref{equa-SNSE-High}), where the noise $W$ is  replaced by   $\epsilon_{n}W$ with $\epsilon_{n}\to 0$.

\begin{theorem}\label{Theorem Vanishing noise} 
	Let $\alpha\in [1,2)$, $r\geq1$ and $\wt \beta\in (0,1)$.
Then, we have
	\begin{equation}
		 \mathscr{A}_{NSE}^{\widetilde{\beta}}\subseteq  \mathscr{S}^{\wt \beta}_{SNSE,2r}.
    \end{equation}
	That is, for any mean-free weak solution $u\in H_{t,x}^{\widetilde{\beta}}$ to the deterministic NSE (\ref{deterministic-NSE}), there exists
	a sequence of weak solutions $\{u^{(\epsilon_{n})}\}_{{n}}$ to the stochastic hyper-viscous NSE  (\ref{equa-SNSE-High}) with $\epsilon_{n}W $ replacing $W $,
	such that
	\begin{equation}\label{1.3.3}
		u^{(\epsilon_{n})}\rightarrow u \quad\text{strongly in } \ \
          \Lo^{2r}_{\om}L_{t}^{2}L_{x}^{2} 
		\ \ as\ \epsilon_{n} \to 0.
	\end{equation}
\end{theorem}

\begin{remark}
Theorem \ref{Theorem Vanishing noise} shows that the set of accumulation points for the stochastic hyper-viscous NSE includes all solutions to the  deterministic NSE  in the class $H_{t,x}^{\widetilde{\beta}}$.   
We note that the deterministic solutions to NSE constructed by 
Buckmaster-Vicol \citep{BV-NS}  
are  in the space $H_{t,x}^{\widetilde{\beta}}$ (see \citep[(6.51)]{ZZL-MHD}), and  
so are the solutions to the hyper-viscous NSE
constructed in \cite{LQZZ-NSE} 
by Li, Qu and the last two named authors. 
As a consequence,  
by Theorem  \ref{Theorem Vanishing noise}, 
both solutions   in \cite{BV-NS}  
and 
\cite{LQZZ-NSE} 
can be approximated by stochastic solutions constructed in Theorem \ref{Thm-Nonuniq-Hyper} as the noise strength tends to zero. 
\end{remark}

\subsection{Strategy of the proof} 
\label{Subsub-proof}

Our proof is mainly based on intermittent convex integrations. The solutions in Theorems \ref{Thm-Nonuniq-Hyper} and \ref{Theorem Vanishing noise} are constructed in an interative way to satisfy the relaxed hyper-viscous Reynolds-NSE system (see {\eqref{2-12}} below). The crucial  quantitative control of the  velocity fields and Reynolds stresses are given in the main interation result, i.e., Proposition \ref{Proposition Main iteration} below.

\medskip  
It should be mentioned  that the building blocks  in the proof have both the temporal and spatial 
intermittencies which are exploited in an optimal way in order to control the high viscosity beyond the Lions exponent. 

As a matter of fact, 
the high dissipation  is a major obstacle for the  ill-posedness of (stochastic) PDEs. In \citep{BV-NS}, this problem is overcomed by introducing the intermittency in the building blocks of convex integration schemes. The situation becomes even more difficult when the viscosity is beyond the Lions exponent, because, due to the classical results by Lions \cite{Lions}, NSE is well-posed in $C_{t}L_{x}^{2}\cap L_{t}^{2}H_{x}^{1}$ in this high viscous regime. It then requires to construct building blocks with  intermittency strong enough  in the high viscous regime.

More delicately, different endpoint spaces  of the LPS criterion require different constructions of building blocks in convex integrations.   

To be precise,  
in the endpoint space $L_t^{\infty}L_x^{\frac{3}{2\a-1}}$, the building blocks shall be almost temporal homogeneous but strong spatial intermittent. We use the intermittent jets initially constructed in  {\citep{BCV}}. As in \citep{BCV}, the intermittent jets used here have almost 3D spatial intermittency and can   control  the dissipativity term $(-\Delta)^{\a}$ when $\a \leq 5/4$, . 
But this spatial intermittency is not strong enough to 
control the high viscosity $(-\Delta)^{\a}$ beyond the Lions exponent, i.e. $\a>5/4$.  
For this purpose, we also use the temporal concentration functions as in \citep{CL-NSE1}
 to oscillate the building blocks in time in order to gain the  extra temporal intermittency.  Moreover,   managing the endpoint case $L_t^{\infty}L_x^{\frac{3}{2\alpha-1}}$ requires us to exploit the temporal intermittency in an almost optimal  way. By carefully selecting parameters in \eqref{3-1}  below, our building blocks have almost $(4\a-5)$-D temporal intermittency.  In particular, this temporal intermittency reaches the 3-D intermittency when $\a$ is close to 2.

In contrast,  
in another endpoint space $L_t^{\frac{2\a}{2\a-1}}L_x^{\infty}$, the building blocks shall be almost  spatial homogeneous and temporal intermittent. 
Instead of the intermittent jets,  
our spatial building blocks are constructed by using concentrated Mikado flows, which provide at most 2-D intermittency when 
$\a$ approaches 2. The temporal building blocks  provide even more intermittency and almost reach 4-D  intermittency as 
$\a$ nears  2. See Section \ref{Sec-S2} for the detailed construction of building blocks, including the selection of parameters.

It is  worth noting that 
the building blocks used here have stronger intermittency 
than those in the context of MHD \cite{CLZ}. 
In particular, 
the present proof can treat the high viscosity regime where $\a\in[\frac{3}{2},2)$ and  the endpoint case $L_t^{\infty}L_x^{\frac{3}{2\alpha-1}}$, which were not treated in \cite{CLZ,ZZL-MHD}.

Let us also mention that, unlike in \cite{LQZZ-NSE},
Theorem  \textcolor{blue}{\ref{Thm-Nonuniq-Hyper}} holds for arbitrarily prescribed initial data in $L_{\sigma}^{2}$. In order to maintain the initial data of approximate relaxed solutions, we use the  temporal truncations of perturbations in the convex integration scheme.  
The delicate point here is that 
the supremum norm  of associated Reynolds  stresses may explode  near the initial time. 
This requires delicate analysis of the Reynolds errors in two distinct temporal regimes: the regime near the initial time and another regime away from the initial time.  The  detailed arguments are contained in Sections \ref{Sec-Reynolds} and \ref{Sec-Reynolds 2}.

\medskip 
\paragraph{\bf Organization} 
The paper is organized as follows. 
In Section \ref{Sec-Main-Iteration}, 
we give the main iterative estimates in the convex integration scheme. 
Assuming these estimates to hold, 
we give the proof of Theorem 
\ref{Thm-Nonuniq-Hyper}, Corollary \ref{coro} 
and Theorem \ref{Theorem Vanishing noise}.
Then, 
in Sections \ref{Sec-S1}  and 
\ref{Sec-Reynolds} we mainly deal with the supercritical regime \texorpdfstring{$\mathcal{S}_{1}$}{S1}. 
We first construct the crucial velocity perturbations in Section \ref{Sec-S1} 
and verify the corresponding iterative estimates. 
Then, in Section \ref{Sec-Reynolds}, 
we verify the iterative estimates for the key  Reynolds stress. 
At last, in 
the second supercritical regime \texorpdfstring{$\mathcal{S}_{2}$}{S2}, 
the iterative estimates 
of the velocity field and the Reynolds stress  
are proved in Sections \ref{Sec-S2} 
and \ref{Sec-Reynolds 2}, respectively.

\medskip
\paragraph{\bf Notations}
 
Let $p \in[1,+\infty]$, $N\in \bbn \cup \{\infty\}$ and $s \in \mathbb{R}$. For simplicity, we use some shorthand notations as follows.
\begin{itemize}
\item 
$C_{c,\sigma}^{N}:=	C_{c,\sigma}^{N}(\[0, T\]\times \T^{3}),$ where the indices $c$ and $\sigma$ mean ``compact support in time'' and ``divergence-free'', respectively;\\
\item 
$L_x^p:=L_x^p(\mathbb{T}^3), \,\, H_x^s:=H_x^s(\mathbb{T}^3), \,\, W_x^{s, p}:=W_x^{s, p}(\mathbb{T}^3)$ represent the usual Sobolev spaces; 
When $I=[0, T]$ 
with $T\in (0, \infty)$, we simply
write $L^{p}_{t}:=L^{p}_{[0, T]}$, and $C_{t}:=C_{[0, T]}$;\\
\item 
$L^p(I; X)$ denotes the space of integrable functions from $I$ to $X$, equipped with the usual $L^p$-norm (with the usual adaptation when $p=\infty$) for any Banach space $X$, and we write $L^p_{I}X :=L^p(I; X)$ for simplicity;\\
\item Let $C_{I}X:= C(I; X)$ stands for the space of continuous functions from $I$ to $X$
equipped with the supremum norm and $C_{I,x}:=C_{I}C_{x}$;\\
\item 
We use the following notations of norms:
\begin{align*}
		\|u\|_{L_{t, x}^p}:=\|u\|_{L_{t}^pL_{ x}^p}, \quad
	\|u\|_{C_{t, x}^N}:=\sum_{0\leq m+|\zeta|\leq N}\|\partial_t^m \nabla^\zeta u\|_{C_{t, x}},
\end{align*}
and
\begin{align*}
\|u\|_{W_{x}^{N, p}}:=\sum_{0 \leq |m| \leq N}\|\nabla^m u\|_{L_{x}^p},\quad
	\|u\|_{W_{t,x}^{N, p}}:=\sum_{0 \leq n+|\zeta| \leq N}\|\p_{t}^{n}\nabla^{\zeta} u\|_{L_{t}^pL_{x}^p},
\end{align*}
where $\zeta=\left(\zeta_1, \zeta_2, \zeta_3\right)$ is the multi-index and $\nabla^\zeta:=\partial_{x_1}^{\zeta_1} \partial_{x_2}^{\zeta_2} \partial_{x_3}^{\zeta_3}$;\\
\item 
For any $T\in (0, \infty)$, $I\subset \[0, T \]$,
and $p$, $q\geq 1$,
let
\begin{equation}\label{1.6}
	\begin{aligned}
		\Lo^{p}_{\om}L^{q}_{I}X:=\left\{u: \om\times I\rightarrow X\mid \|u\|_{\Lo^{p}_{\om}L^{q}_{I}X}:=\(\mathbb{E}\|u\|_{L^{q}_{I}X}^{p}\)^{\frac{1}{p}}<\infty\right\},
		%\Los^{p}\(\om; C_{t}X\):=\left\{u: \om\times\[\overline{t}, T \]\rightarrow X\mid \|u\|_{\Los^{p}(\om;C_{t}X)}:=\sup_{t\in \ti }\(\mathbb{E}\|u\|_{C_{t}X}^{p}\)^{\frac{1}{p}}\right\}.
	\end{aligned}
\end{equation}
and define $\Lo^{p}_{\om}C_{I}^{N}X$ and $\Lo^{p}_{\om} C_{I,x}^{N}$
in an analogous way;\\ 
\item 
We also use the notation $a\lesssim b$ if there exists a constant $C>0$ such that $a\leq Cb$.
\end{itemize}

\section{Main iteration} \label{Sec-Main-Iteration}

The key result of this section is the main iteration in Proposition \ref{Proposition Main iteration} which is crucial for the proof of Theorems \ref{Thm-Nonuniq-Hyper} and \ref{Theorem Vanishing noise}. 

To begin with, let  $z$ denote the stochastic convolutions of the noise:
\begin{equation}\label{2-1}
	\begin{aligned}
		&z(t)=e^{t\(-\nu(-\Delta)^{\a}-I\)}u_{0}+\int_{0}^{t}e^{(t-s)\(-\nu(-\Delta)^{\a}-I\)}\mathrm{d} W(s),\quad t\in [0, T],
	\end{aligned}
\end{equation}
which is the solution to the following linear stochastic equation
\begin{eqnarray}\label{2-2}
	\left\{\begin{array}{lc}
		\mathrm{d} z+(\nu(-\Delta)^{\a}+1)z \mathrm{d} t=\mathrm{d} W_{t},  \quad t\in [0, T],\\
		\div z=0,  \\
		z(0)=u_{0}.
	\end{array}\right.
\end{eqnarray}
For $t<0$ and $t>T$, we  extend $z(t)$ to be zero.

Set
\begin{align}\label{2-4}
	&z^{u}(t):=e^{t\(-\nu(-\Delta)^{\a}-I\)}u_{0},\quad Z(t):=\int_{0}^{t}e^{(t-s)\(-\nu(-\Delta)^{\a}-I\)}\mathrm{d} W(s),  \quad t\in [0, T].
\end{align}
One has
\begin{align*}
z(t)=z^{u}(t)+Z(t).
\end{align*}
The following estimate holds for the stochastic convolution. The classical case where $\a=1$ can be found in \citep{HZZ-2}.

\begin{proposition}[\citep{HZZ-2, CLZ}]\label{Proposition stochastic}
	Suppose that $ Tr(GG^{*})<\infty $, $ i=1,2$.
	Then, for any $\delta\in (0, 1/2)$, $\kappa\in [0, \a)$ and  $p\geq 2$,
	\begin{equation}\label{stochastic evolution}
		\begin{aligned}
			\textbf{E} \left[ \|Z\|^{p}_{C_{t}H_{x}^{\kappa+4}}
			+ \|Z\|^{p}_{C_{t}^{1/2-\delta} L_{x}^{2}}\right]
			\leq (p-1)^{p/2}L^{p},
		\end{aligned}
	\end{equation}
	where $ L \geq 1 $ is independent of $p$.
\end{proposition}

The frequency truncated stochastic convolution is defined by
\begin{equation}\label{2-10}
	Z_{q}:={\bbp_{\leq \la^{15}}}Z,
\end{equation}
where $\widehat{\bbp_{\leq c}Z}(\xi):=\hat{Z}(\xi)\eta_{c}(\xi)$, $\xi\in \mathbb{Z}^{3},$ with $\eta_{c}(\xi):=\eta(\frac{\xi}{c})$ for some smooth radial function $\eta: \bbr^{3}\to [0,1]$ such that
\begin{align*}
    \eta(\xi)=\left\{\begin{array}{ll}
	1, & 0 \leq |\xi| <1, \\
	0, & \xi\geq  2.\end{array}\right.
\end{align*}

We set
\begin{align*}
	z_{q}:=z^{u}+Z_{q}.
\end{align*}
Note that by the Sobolev embedding $H_{x}^{2} \hookrightarrow L_{x}^{\infty}$ and \eqref{stochastic evolution}, 
\begin{equation}\label{2-11}
	\|Z_{q}\|_{\Lo^{p}_{\om}C_{t}L_{x}^{\infty}}
	\leq
	\|Z_{q}\|_{\Lo^{p}_{\om}C_{t}H_{x}^{2}}
	\lesssim  (p-1)^{\frac 12}L.
\end{equation}

Next, we use the random shift
\begin{align}  \label{uB-z-shift}
	u = v + z,
\end{align}
and reformulate the original stochastic NSE 
\eqref{equa-SNSE-High} as follows:
\begin{eqnarray}   \label{2-5}
	\left\{\begin{array}{lc}
		\p_{t}v+\nu(-\Delta)^{\a}v-z+\big((v+z)\cdot \nabla \big)(v+z)+\nabla P  =0,\\
		\div v=0,\\
		 v(0)=0.
	\end{array}\right.
\end{eqnarray}
It is thus equivalent to construct non-unique solutions to equation \eqref{2-5}.

For this purpose, 
we  use the convex integration method. 
Let $r\geq1$ and  $L\geq 1$ be as in Proposition \ref{Proposition stochastic}, and $ c $ satisfy $ q50^{q} \leq c 60^{q}$, $q\in\bbn$. Let $ a\in \bbn$, $b\in 2\bbn$, $\beta \in (0, 1]$, and let
\begin{equation}\label{2-6}
	\la := a^{b^{q}}, \quad \dq:=\lambda_{1}^{3\beta}\la^{-2\beta}, \quad
	\varsigma_{q}:=\la^{-40}.
\end{equation}
Take $a$, $b$ large enough and $\beta$, $\delta$ small such that
\begin{equation}\label{2-7}
	a\geq (60\cdot 8\cdot 50rL^{2})^{c},\quad
	b>\frac{150000}{\varepsilon}, \quad
	 \beta \leq \frac{5}{2b^2},
	 \quad\delta<\frac{1}{80},
\end{equation}
where  $\varepsilon \in \bbq_{+}$ is sufficiently small such that for the given $(s,\gamma,p) \in S_{1}$,
\begin{equation}\label{2-8}
	\varepsilon\leq \frac{1}{20}\min \left\{2-\a, \frac{4\a-5}{\gamma}+\frac{3}{p}-(2\a-1)-s \right\},
	\ \ \text{and}\ b\a, \ \ b\varepsilon \in \bbn,
\end{equation}
and for the given $(s,\gamma,p) \in S_{2}$,
\begin{equation}\label{2-9}
	\varepsilon\leq \frac{1}{20}\min \left\{2-\a, \frac{2\a}{\gamma}+\frac{2\a-2}{p}-(2\a-1)-s \right\},
\ \ \text{and}\ b\a,\ \ b\varepsilon \in \bbn.
\end{equation}
For convenience, we 
consider the case where 
$\alpha \in \bbq$ below.  
For the case where $\alpha \notin \bbq$, 
one can choose a rational number sufficiently close to $\alpha$ 
which can still guarantee the iterative estimates 
in the main iteration.

Now, we consider the relaxed  hyper-viscous Navier-Stokes-Reynolds system: for each integer $ q\in \bbn $,
\begin{eqnarray}\label{2-12}
	\left\{\begin{array}{lc}
	\p_{t}v_{q}+\nu(-\Delta)^{\a}v_{q}-z_{q}+\div \big((v_{q}+z_{q})\otimes(v_{q}+z_{q})\big)+\nabla P_{q}  =\div \ru,\\
		\div v_{q}=0,\\
			v_{q}(0)=0.
	\end{array}\right.
\end{eqnarray}

We have the following main iteration result for the relaxed velocity $v_{q}$ and the Reynolds stress $\Ru_{q}$.

\begin{proposition} [Main iteration] \label{Proposition Main iteration}
Let $T\in(0, \infty)$ and $ r\geq 1 $ be fixed. Let $(s,p,\ga)\in \mathcal{S}_{1} $ for $ \a\in[\frac{5}{4}, 2) $,  or $(s,p,\ga)\in \mathcal{S}_{2} $ for $ \a\in[1, 2) $, and let $(v_{q},\ru) $ be an  $(\ft)-adapted$ solution to
 (\ref{2-12}) for some $ q\in\bbn  $  satisfying that for any $ 0\leq N\leq 4 $ and any $ m\geq 1 $,
\begin{align}
		&\|v_{q}\|_{\Lo^{m}_{\om}C_{t,x}^{N}}
		\leq
		\la^{{2\a(N+1)+5}}\(8(N+10)mL^{2}50^{q-1}\)^{(N+10)50^{q-1}},\label{2-14}\\
		&\|\Ru_{q}\|_{\Lo^{m}_{\om} C_{t}L_{x}^{1}}\leq \la^{4\a+12}(8 mL^{2}50^{q})^{50^{q}},\label{2-15}\\
		&\|\Ru_{q}\|_{\Lo^{r}_{\om} L_{t}^{1}L_{x}^{1}}\lesssim {\dqq}, \label{2-16}
\end{align}	
where the implicit constant is deterministic and independent of $q$, $m$ and $r$.
Then, there exists an $(\ft)-adapted$ process
$(v_{q+1},\Ru_{q+1}) $
which solves (\ref{2-12}) and obeys estimates (\ref{2-14})-(\ref{2-16}) at level $ q+1 $.
In addition, it holds that
\begin{align}
	&\|v_{q+1}-v_{q}\|_{\Lo^{2r}_{\om} L_{t}^{2}L_{x}^{2}}
	\lesssim 
 {\delta_{q+1}^{\frac{1}{2}}},	\label{2-17}\\
	&\|v_{q+1}-v_{q}\|_{\Lo^{r}_{\om} L_{t}^{1}L_{x}^{2}}
	\leq\dqqq^{\frac{1}{2}},\label{2-18}\\
	&\|v_{q+1}-v_{q}\|_{\Lo^{r}_{\om}L_{t}^{\ga}W_{x}^{s,p}}
	\leq\dqqq^{\frac{1}{2}}.\label{2-19}
\end{align}	
\end{proposition}

%The proof of Proposition \ref{Proposition Main iteration} will occupy Sections \ref{Sec-S1} and \ref{Sec-Reynolds} below. 

The proof of Proposition \ref{Proposition Main iteration} will be presented in Sections \ref{Sec-S1}-\ref{Sec-S2}. Before that, let us first prove Theorems \ref{Thm-Nonuniq-Hyper} and \ref{Theorem Vanishing noise} in Subsections \ref{subsec2.1} and \ref{subsec2.2} below, respectively, assuming Proposition \ref{Proposition Main iteration} to hold.

\subsection{Proof of Theorem \ref{Thm-Nonuniq-Hyper}} \label{subsec2.1}  

In this subsection, we prove 
Theorems \ref{Thm-Nonuniq-Hyper}.   
In the  initial step $q=0$, we set
\begin{align}
	&v_{0}=\widetilde{u},\quad \mathring{R}_{0}=\mathcal{R}(\p_{t}\widetilde{u}+\nu(-\Delta)^{\a}\widetilde{u}-z_{0})+(\widetilde{u}+z_{0})
\mathring{\otimes}(\widetilde{u}+z_{0}),    \label{7-1}   \\	
     &P_{0}=-\frac{1}{3}|\widetilde{u}_{0}+z_{0}|^{2},\label{7-2}
\end{align}
where $ \widetilde{u}\in C_{0} ^{\infty}([0,T]\times\T^{3})$  is the divergence-free and mean-free vector  given by Theorem \ref{Thm-Nonuniq-Hyper}. Note that, by choosing $a$ sufficiently large,  \eqref{2-14} is satisfied at level $q=0$.

We now proceed to verify the inductive decay estimate  \eqref{2-16} at level $q=0$. To this end, let
\begin{equation}\label{in}
	M:=\|u_{0}\|_{L^{2}_{x}}.
\end{equation}
By  (\ref{stochastic evolution}), we derive
\begin{align*}%\label{2.10-2}
	\|\mathring{R}_{0}\|_{\Lo^{r} _{\om}L_{t}^{1}L_{x}^{1}}
	&\leq\|\mathcal{R}\p_{t}\widetilde{u}\|_{L_{t}^{1}L_{x}^{2}}+\|\mathcal{R}(-\Delta)^{\a}\widetilde{u}\|_{L_{t}^{1}L_{x}^{2}} +\|\mathcal{R}z_{0}\|_{\Lo^{r} _{\om}L_{t}^{1}L_{x}^{2}} 
	+\|(\widetilde{u}+z_{0})\mathring{\otimes}(\widetilde{u}+z_{0})\|_{\Lo^{r} _{\om}L_{t}^{1}L_{x}^{1}} \notag\\
	&\lesssim
	\|\p_{t}\widetilde{u}\|_{C_{t}L_{x}^{2}}+\|\widetilde{u}\|_{C_{t,x}^{3}}+\|z_{0}\|_{\Lo^{2r}_{\om} C_{t}L_{x}^{2}} + \|\widetilde{u}\|^{2} _{C_{t}L_{x}^{2}}+ \|z_{0}\|^{2} _{\Lo^{2r}_{\om}C_{t}L_{x}^{2}}\notag\\
	&\lesssim
	L_{*}+M+(2r-1)^{\frac{1}{2}}L+(M+(2r-1)^{\frac{1}{2}}L)^{2}
	\lesssim \delta_{1},
\end{align*}
where  $L_{*}:=1+ \|\widetilde{u}\|_{C_{t,x}^{3}}^{2}$ is a deterministic constant. Note that the last step is valid by choosing $a$  and so $\delta_{1}=\lambda_{1}^{\beta}$ 
 large enough. Thus,  \eqref{2-16} is satisfied at level $q=0$.

Similarly,  by choosing $a$ possibly larger,  we have that  $\|\mathring{R}_{0}\|_{\Lo^{m}_{\om} C_{t}L_{x}^{1}}\leq\lambda_{0}^{4\a+5}8mL^{2} $, which satisfies \eqref{2-15}  at level $q=0$.

Thus, estimates \eqref{2-14}-\eqref{2-16} are satisfied at the initial step $q\geq 0$. Applying Proposition \ref{Proposition Main iteration},  we  obtain a sequence of relaxed solutions  $\{(v_{q}, \Ru_{q})\}_{q\geq 0}$ to \eqref{2-12}, which satisfy the iterative estimates \eqref{2-14}-\eqref{2-19} for all $q\geq 0$. Moreover, arguing as in \citep{CLZ}, one infers that the relaxed solutions $(v_{q},  \Ru_{q})$ are $(\F_t)$-adapted.

In view of  \eqref{2-17}-\eqref{2-19},  
we see that $\{v_{q}\}_{q\geq 0}$ is a  Cauchy sequence in the space $ \Lo^{2r}_{\om} L_{t}^{2}L_{x}^{2}\cap\Lo^{r}_{\om} L_{t}^{1}L_{x}^{2}\cap\Lo^{r}_{\om}L_{t}^{\ga}W_{x}^{s,p}$.
In particular,
there exists $ v$ in this space such that as $q\to \infty$,
\begin{align}\label{s1} 
v_{q}\rightarrow v \quad \text{in}\quad \Lo^{2r}_{\om} L_{t}^{2}L_{x}^{2}\cap\Lo^{r}_{\om} L_{t}^{1}L_{x}^{2}\cap\Lo^{r}_{\om} L_{t}^{\ga}W_{x}^{s,p},
\quad
	\Ru_{q}\rightarrow 0   \quad\text{in} \quad \Lo^{r}_{\om}L_{t}^{1}L_{x}^{1}.
\end{align}
 The limit $v$ indeed satisfies   \eqref{2-5} 
in the distributional sense (see, e.g., \citep{CLZ} for relevant arguments).
 As a consequence, $u:=v+z $ is a solution to \eqref{equa-SNSE-High} with  the initial datum $u_{0}$. By  (\ref{2-18}) and (\ref{2-19}), choosing $a$ sufficiently large, one also has
\begin{align}\label{2.19}
\|u-(\widetilde{u}+z)\|_{\Lo^{r}_{\om} L_{t}^{1}L_{x}^{2}}+\|u-(\widetilde{u}+z)\|_{\Lo^{r}_{\om}L_{t}^{\ga}W_{x}^{s,p}}
	&\leq\sum_{q=0}^{\infty}\|v_{q+1}-v_{q}\|_{\Lo^{r}_{\om} L_{t}^{1}L_{x}^{2}}+\|v_{q+1}-v_{q}\|_{\Lo^{r}_{\om}L_{t}^{\ga}W_{x}^{s,p}}\notag\\
	&\leq\sum_{q=0}^{\infty}4\dqqq^{1/2}<\epsilon,
\end{align}
which implies that \eqref{1.5} holds. 
Therefore, 
the proof of Theorem \ref{Thm-Nonuniq-Hyper} is complete. 
\hfill $\square$

\medskip 
\textbf{Proof of Corollary \ref{coro}.}  The non-uniqueness of solutions to \eqref{equa-SNSE-High} follows from the flexibility of smooth vector fields in Theorem \ref{Thm-Nonuniq-Hyper}. Actually, choose any divergence-free and mean-free vector fields  $ \{\widetilde{u}_{i}\}_{i\in \bbn_{+}} \subseteq C_{0}^{\infty}([0,T]\times\T^{3}) $ satisfying $\widetilde{u}_{i}(0)=0$ and 
\begin{equation}\label{2.20}
	\begin{aligned}
	\| \wt{u}_i\|_{L_{t}^{1}L_{x}^{2}}^{2}\geq 4T^{2}\|u_{0}\|_{L_{x}^{2}}^{2}+\frac{1}{2}L^{2}+4  (1+T^{3}Tr(G_{i}G_{i}^{*})),\quad 	\|\widetilde{u}_{i}-\widetilde{u}_{j}\|_{L_{t}^{\gamma}W_{x}^{s,p}}\geq 1, ~\forall i\neq j.
	\end{aligned}
\end{equation}
Let $ \epsilon=1/3 $. Then, by virtue of Theorem \ref{Thm-Nonuniq-Hyper}, 
we obtain a sequence of corresponding solutions $ \{u_{i}\} $  to \eqref{equa-SNSE-High} 
satisfying
\begin{align*}%\label{2.21}
	\|u_{i}-(\widetilde{u}_{i}+z)\|_{\Lo^{r}_{\om}L_{t}^{\ga}W_{x}^{s,p}}<1/3,
\end{align*}
which yields that 
\begin{align*}%\label{2.22}
\|u_{i}-u_{j}\|_{\Lo^{r}_{\om}L_{t}^{\ga}W_{x}^{s,p}}
	&\geq \|\widetilde{u}_{i}-\widetilde{u}_{j}\|_{L_{t}^{\gamma}W_{x}^{s,p}}-\|u_{i}-(\widetilde{u}_{i}+z)\|_{\Lo^{r}_{\om}L_{t}^{\ga}W_{x}^{s,p}}-\|u_{j}-(\widetilde{u}_{j}+z)\|_{\Lo^{r}_{\om}L_{t}^{\ga}W_{x}^{s,p}}\notag\\
	&\geq 1-1/3-1/3\geq 1/3. 
\end{align*} 
Thus, $\{u_{i}\}$ are different solutions.

At last, we show that $\{u_{i}\}$ are non-Leray-Hopf solutions. Assume that $\widetilde{v}$ is a Leray-Hopf solution to \eqref{equa-SNSE-High}. Then, it holds that
\begin{align}\label{1-12}
\mathbb{E} \|\widetilde{v}(t)\|_{L_{x}^{2}}^{2}+2\int_{0}^{t}\mathbb{E}\|\widetilde{v}(s)\|_{\dot{H}^{\a}}^{2}\mathrm{d}s\leq \|u_{0}\|_{L_{x}^{2}}^{2} + Tr(GG^{*})t.
\end{align}
Using \eqref{1-12} and Minkovski's inequality one has
\begin{align}\label{1-10-1}
\|\widetilde{v}\|_{L^{2}_{\om}L_{t}^{1}L_{x}^{2}}^{2}\leq \|\widetilde{v}\|_{L_{t}^{1}L^{2}_{\om}L_{x}^{2}}^{2}\leq T^{2}\|u_{0}\|_{L_{x}^{2}}^{2} +\frac{1}{2}Tr(GG^{*})T^{3} .
\end{align}

We claim that $u_i$ violates \eqref{1-10-1} for each $i\in \mathbb{N}_+$. To this end, by \eqref{1.5}, \eqref{2.20} and Proposition \ref{Proposition stochastic}, we derive
 \begin{align*} \|u_n\|_{L_{\om}^{2}L_{t}^{1}L_{x}^{2}}^{2}
     &\geq \frac{1}{2}\|z+\wt{u}\|_{L_{\om}^{2}L_{t}^{1}L_{x}^{2}}^{2}-\|u-(z+\wt{u})\|_{L_{\om}^{2}L_{t}^{1}L_{x}^{2}}^{2}\\
     &\geq
\frac{1}{4}\|\wt{u}\|_{L_{t}^{1}L_{x}^{2}}^{2}-\frac{1}{2}\|z\|_{L_{\om}^{2}L_{t}^{1}L_{x}^{2}}^{2}-\|u-(z+\wt{u})\|_{L_{\om}^{2}L_{t}^{1}L_{x}^{2}}^{2}\\
&> T^{2}\|u_{0}\|_{L_{x}^{2}}^{2}+T^{3}Tr(GG^{*}),
 \end{align*}
which shows that $u_i$ violates \eqref{1-10-1}, as claimed. Hence, $\{u_{i}\}$ are non-Leray-Hopf solutions. The proof  is complete.
\hfill$\square$

\subsection{Proof of Theorem \ref{Theorem Vanishing noise}}\label{subsec2.2} 

Let $z^{(\epsilon_{n})}$ solve the linear stochastic equation  \eqref{2-2} with $\epsilon_{n}W$ replacing $W$. One has
\begin{align}\label{1.3.5.11}
z^{(\epsilon_{n})}(t):=z^{u}(t)+Z^{(\epsilon_{n})}(t),\quad t\geq 0,
\end{align}
where
\begin{align} \label{1.3.5.12}
 z^{u}(t):=e^{t\(-\nu(-\Delta)^{\a}-I\)}u_{0},\quad Z^{(\epsilon_{n})}(t):=\epsilon_{n}\int_{0}^{t}e^{(t-s)\(-\nu(-\Delta)^{\a}-I\)}dW(s).
 \end{align}

Set
\begin{align*}
	&u_{n}:=(u*_{x}\varrho_{\lambda_{n}^{-1}})*_{t} \vartheta_{\lambda_{n}^{-1}},\quad
	z_{n}^{(\epsilon_{n})}:=(z^{(\epsilon_{n})}*_{x}\varrho_{\lambda_{n}^{-1}})*_{t}\vartheta_{\lambda_{n}^{-1}},\\
	 &z_{n}^{u}:=(z^{u}*_{x}\varrho_{\lambda_{n}^{-1}})*_{t} \vartheta_{\lambda_{n}^{-1}},\quad
	Z_{n}^{(\epsilon_{n})}:=(Z^{(\epsilon_{n})}*_{x}\varrho_{\lambda_{n}^{-1}})*_{t} \vartheta_{\lambda_{n}^{-1}},
\end{align*}
and 
	\begin{align}\label{zn}
	\wt{z}_{n}^{(\epsilon_{n})}:=z_{n}^{u}+\bbp_{\leq \lambda_{n}^{15}}Z_{n}^{(\epsilon_{n})},
\end{align}
where $\varrho$ and $\vartheta$ are as in \eqref{2-20} below.

Let
\begin{align}\label{1.3.5.17}
v_{n}:=u_{n}-z_{n}^{u}.
\end{align}

Since $u$ is a weak solution to \eqref{deterministic-NSE}, we deduce the following equation of $v_{n}$:
\begin{eqnarray}\label{1.3.55}
	\left\{\begin{array}{lc}
		\p_{t}v_{n}+\nu(-\Delta)^{\a}v_{n}-\wt{z}_{n}^{(\epsilon_{n})}+\div\big((v_{n}+\wt{z}_{n}^{(\epsilon_{n})})\mathring{\otimes} (v_{n}+\wt{z}_{n}^{(\epsilon_{n})})\big)+\nabla P_{n}
	    =\div \run,\\
		\div v_{n}=0, \quad v_{n}(0)=0,
	\end{array}\right.
\end{eqnarray}
where the stresses $\run$ is given by
\begin{align}	\label{1.3.56}
	&\run:=(v_{n}+\wt{z}_{n}^{(\epsilon_{n})})\mathring{\otimes} (v_{n}+\wt{z}_{n}^{(\epsilon_{n})})-\big((u\mathring{\otimes}u)*_{x}\varrho_{\lambda_{n}^{-1}})*_{t}\vartheta_{\lambda_{n}^{-1}}-\mathcal{R}(\bbp_{\leq f(n)}Z_{n}^{(\epsilon_{n})}\big),
\end{align}
and
\begin{equation}\label{1.3.58}
	P_{n}:= P*_{x}\varrho_{\lambda_{n}^{-1}}*_{t}\vartheta_{\lambda_{n}^{-1}}-|v_{n}+\wt{z}_{n}^{(\epsilon_{n})}|^{2}+(|u|^{2}*_{x}\varrho_{\lambda_{n}^{-1}})*_{t}\vartheta_{\lambda_{n}^{-1}}.
\end{equation}

\begin{lemma}\label{Lemma 2.4}
 For $a$ sufficiently large, $(v_{n}, \run)$ satisfies
 estimates (\ref{2-14})-(\ref{2-16}) at level $q=n  (\geq1)$.   
\end{lemma}

\begin{proof} 
Let us start with the decay estimate (\ref{2-16}).
Let
 \begin{align}\label{1.3.5.22}
 \epsilon_{n}:=\lambda_{n}^{-1},\quad \widetilde{M}:=\|u\|_{H_{t,x}^{\widetilde{\beta}}}.
 \end{align}
Using the Minkowski inequality and the Slobodetskii-type norm of Sobolev spaces
we have (see \citep[(6.35)]{ZZL-MHD})
\begin{equation}\label{1.3.59}
	\|u-u_{n}\|_{L_{t}^{2}L_{x}^{2}}\lesssim \lambda_{n}^{-\widetilde{\beta}}\|u\|_{H_{t,x}^{\widetilde{\beta}}}\lesssim \lambda_{n}^{-\widetilde{\beta}}\widetilde{M}.
\end{equation}
By   (\ref{1.3.5.22}) and   H\"{o}lder's inequality, 
\begin{align}\label{1.3.60}
	&\quad\|(v_{n}+\wt{z}_{n}^{(\epsilon_{n})})\mathring{\otimes}(v_{n}+\wt{z}_{n}^{(\epsilon_{n})})-u_{n}\mathring{\otimes}u_{n}\|_{L_{t}^{1}L_{x}^{1}}\notag\\
	&\lesssim\|(\wt{z}_{n}^{(\epsilon_{n})}-z_{n}^{u})\mathring{\otimes}(u_{n}-z_{n}^{u}+\wt{z}_{n}^{(\epsilon_{n})})\|_{L_{t}^{1}L_{x}^{1}}
	+\|u_{n}\mathring{\otimes}(\wt{z}_{n}^{(\epsilon_{n})}-z_{n}^{u})\|_{L_{t}^{1}L_{x}^{1}}\notag\\
	&\lesssim
	\|\wt{z}_{n}^{(\epsilon_{n})}-z_{n}^{u}\|_{L_{t}^{2}L_{x}^{2}}^{2}+\|u_{n}\|_{L_{t}^{2}L_{x}^{2}}\|\wt{z}_{n}^{(\epsilon_{n})}-z_{n}^{u}\|_{L_{t}^{2}L_{x}^{2}}\notag\\
	&\lesssim
	\|\bbp_{\leq \lambda_{n}^{15}}Z_{n}^{(\epsilon_{n})}\|_{L_{t}^{2}L_{x}^{2}}^{2}+\widetilde{M}\|\bbp_{\leq \lambda_{n}^{15}}Z_{n}^{(\epsilon_{n})}\|_{L_{t}^{2}L_{x}^{2}}.
\end{align}
Furthermore,  we have (see  \citep[(6.36)]{ZZL-MHD})
\begin{align}\label{1.3.62}
		&\quad\|u_{n}\mathring{\otimes} u_{n}-((u \mathring {\otimes} u)*_{x}\varrho_{\lambda_{n}^{-1}})*_{t}\vartheta_{\lambda_{n}^{-1}}\|_{L_{t}^{1}L_{x}^{1}}
	\lesssim	
	\lambda_{n}^{-2\widetilde{\beta}}\|u\|^{2}_{H_{t,x}^{\widetilde{\beta}}}\lesssim \lambda_{n}^{-2\widetilde{\beta}}\widetilde{M}^{2}.
\end{align}
Thus, using \eqref{1.3.56}  and Proposition \ref{Proposition stochastic} and combining \eqref{1.3.60} and \eqref{1.3.62} altogether  we obtain
\begin{align}\label{1.3.63}
	\|\run\|_{\Lo^{r}_\om L_{t}^{1}L_{x}^{1}}
	&\lesssim	\|\bbp_{\leq \lambda_{n}^{15}}Z_{n}^{(\epsilon_{n})}\|^{2}_{\Lo^{2r}_{\om}L_{t}^{2}L_{x}^{2}}+\widetilde{M}\|\bbp_{\leq \lambda_{n}^{15}}Z_{n}^{(\epsilon_{n})}\|_{\Lo^{2r}_{\om}L_{t}^{2}L_{x}^{2}}\notag\\
	&\quad
	+\lambda_{n}^{-2\widetilde{\beta}}\widetilde{M}^{2}
	+\|\mathcal{R}(\bbp_{\leq \lambda_{n}^{15}}Z_{n}^{(\epsilon_{n})})\|_{\Lo^{2r}_{\om}L_{t}^{2}L_{x}^{2}}
	\notag\\
	&\lesssim \epsilon_{n}^{2}(2r-1)L^{2}+\epsilon_{n}(2r-1)^{\frac{1}{2}}L+\lambda_{n}^{-2\widetilde{\beta}}\notag\\
	&\lesssim \lambda_{n}^{-1}+\lambda_{n}^{-2\widetilde{\beta}},
\end{align}
where  (\ref{1.3.5.22}) was applied in the last step.
Thus, estimate  (\ref{2-16}) is verified at level $q=n$.

Moreover, using the Sobolev embedding $H_{t,x}^{3}\hookrightarrow L_{t,x}^{\infty}$ we have that for $0\leq N\leq 4$,
\begin{align*}%\label{1.3.65}
	\|v_{n}\|_{\Lo^{m}_{\om}C_{t,x}^{N}}
	&\lesssim
	\|u_{n}\|_{H_{t,x}^{3+N}}+	\|z_{n}^{u}\|_{\Lo^{2m}_{\om}H_{t,x}^{3+N}}\notag\\
&\lesssim
\lambda_{n}^{3+N}\|u\|_{L_{t}^{2}L_{x}^{2}}+\lambda_{n}^{3+N}\|z^{u}\|_{\Lo^{2m}_{\om}L_{t}^{2}L_{x}^{2}}\notag\\
&\lesssim
\lambda_{n}^{3+N}(\widetilde{M}+M)
\lesssim
\lambda_{n}^{3+N},
\end{align*}
which yields (\ref{2-14}) at level $q=n$.

As for estimate (\ref{2-15}),  by  (\ref{1.3.5.17}), (\ref{1.3.5.22}), \eqref{zn} and H\"{o}lder's inequality,
\begin{align}\label{1.3.66}
	&\quad\|(v_{n}+\wt{z}_{n}^{(\epsilon_{n})})\mathring{\otimes}(v_{n}+\wt{z}_{n}^{(\epsilon_{n})})-u_{n}\mathring{\otimes}u_{n}\|_{C_{t}L_{x}^{1}}\notag\\
%	&\lesssim	
	%\|(v_{n}+\wt{z}_{n}^{(\epsilon_{n})})\mathring{\otimes}(v_{n}+\wt{z}_{n}^{(\epsilon_{n})})-u_{n}\mathring{\otimes}(v_{n}+\wt{z}_{n}^{(\epsilon_{n})})\|_{C_{t}L_{x}^{1}}
%	+\|u_{n}\mathring{\otimes}(v_{n}+\wt{z}_{n}^{(\epsilon_{n})})-u_{n}\mathring{\otimes}u_{n}\|_{C_{t}L_{x}^{1}}\notag\\
	&\lesssim\|(\wt{z}_{n}^{(\epsilon_{n})}-z_{n}^{u})\mathring{\otimes}(u_{n}-z_{n}^{u}+\wt{z}_{n}^{(\epsilon_{n})})\|_{C_{t}L_{x}^{1}}
	+\|u_{n}\mathring{\otimes}(\wt{z}_{n}^{(\epsilon_{n})}-z_{n}^{u})\|_{C_{t}L_{x}^{1}}\notag\\
	&\lesssim
	\|\wt{z}_{n}^{(\epsilon_{n})}-z_{n}^{u}\|_{C_{t}L_{x}^{2}}^{2}+\|u_{n}\|_{C_{t}L_{x}^{2}}\|\wt{z}_{n}^{(\epsilon_{n})}-z_{n}^{u}\|_{C_{t}L_{x}^{2}}\notag\\
	&\lesssim
	\|\bbp_{\leq \lambda_{n}^{15}}Z_{n}^{(\epsilon_{n})}\|_{C_{t}L_{x}^{2}}^{2}+\widetilde{M}\|\bbp_{\leq \lambda_{n}^{15}}Z_{n}^{(\epsilon_{n})}\|_{L_{t}^{2}L_{x}^{2}}.
\end{align}
Applying the Sobolev embedding
 $W_{t,x}^{5,1} \hookrightarrow L_{t,x}^{\infty}$  yields
 \begin{align}\label{1.3.68}
 	&\quad\|u_{n}\mathring{\otimes} u_{n}-\big((u \mathring {\otimes} u)*_{x}\varrho_{\lambda_{n}^{-1}}\big)*_{t}\vartheta_{\lambda_{n}^{-1}}\|_{C_{t}L_{x}^{1}}\notag\\
 	&\lesssim	
 	\|u_{n}\mathring{\otimes}u_{n}\|_{W_{t,x}^{5,1}}+\|\big((u \mathring {\otimes} u)*_{x}\varrho_{\lambda_{n}^{-1}}\big)*_{t}\vartheta_{\lambda_{n}^{-1}}\|_{W_{t,x}^{5,1}}\notag\\
 	&\lesssim	
 	\sum_{0\leq N_{1}+N_{2}\leq 5}
 	\(\|\p_{t}^{N_{1}}\nabla^{N_{2}}(u_{n}\mathring{\otimes}u_{n})\|_{L_{t}^{1}L_{x}^{1}}+\|\p_{t}^{N_{1}}\nabla^{N_{2}}\big((u \mathring {\otimes} u)*_{x}\varrho_{\lambda_{n}^{-1}}*_{t}\vartheta_{\lambda_{n}^{-1}}\big)\|_{L_{t}^{1}L_{x}^{1}}\)\notag\\
 	&\lesssim	
 	\lambda_{n}^{5}\|u_{n}\|^{2}_{L_{t}^{2}L_{x}^{2}}+\lambda_{n}^{5}\|u\|^{2}_{L_{t}^{2}L_{x}^{2}}\lesssim	
 	\lambda_{n}^{5}\widetilde{M}^{2}.
 \end{align}
Thus, combining Proposition \ref{Proposition stochastic}, (\ref{1.3.56}), (\ref{1.3.66}) and (\ref{1.3.68}) altogether we obtain
\begin{align}\label{1.3.70-1}
	\|\run\|_{\Lo^{m}_{\om}C_{t}L_{x}^{1}}
&\lesssim	
\|\bbp_{\leq \lambda_{n}^{15}}Z_{n}^{(\epsilon_{n})}\|^{2}_{\Lo^{2m}_{\om}C_{t}L_{x}^{2}}+\widetilde{M}\|\bbp_{\leq \lambda_{n}^{15}}Z_{n}^{(\epsilon_{n})}\|_{\Lo^{2m}_{\om}C_{t}L_{x}^{2}}+\lambda_{n}^{5}\widetilde{M}^{2}
+\|\mathcal{R}(\bbp_{\leq \lambda_{n}^{15}}Z_{n}^{(\epsilon_{n})})\|_{\Lo^{2m}_{\om}C_{t}L_{x}^{2}}
\notag\\
&\lesssim
\epsilon_{n}^{2}(2m-1)L^{2}+\epsilon_{n}(2m-1)^{\frac{1}{2}}L+\lambda_{n}^{5}\notag\\
&\lesssim
\lambda_{n}^{-2}(2m-1)L^{2}+\lambda_{n}^{-1}(2m-1)^{\frac{1}{2}}L+\lambda_{n}^{5}\lesssim \lambda_{n}^{5}(2m-1)L^{2}.
\end{align}
This yields estimate  (\ref{2-15}) at level $q=n$. The proof is complete. 
\end{proof}

\medskip 
\paragraph{\bf Proof of Theorem \ref{Theorem Vanishing noise} (continued).} 
By the virtue of Lemma \ref{Lemma 2.4}, we apply Proposition \ref{Proposition Main iteration} to get  a sequence of convex integration solutions $(v_{n,q},\runq)_{q\geq n}$ satisfying (\ref{2-14})-(\ref{2-16}),
$\{v_{n,q}\}_{q\geq n}$ is a Cauchy sequences in $\Lo^{2r}_{\om}L_{t}^{2}L_{x}^{2}$ and $ \{\runq\}_{q\geq n}$ converges strongly to $0$ in $\Lo^{r}_{\om}L_{t}^{1}L_{x}^{1}$.  Consequently, by taking the limit as 
$q\rightarrow \infty$, we obtain a weak solution $u^{(\epsilon_{n})}\in \Lo^{2r}_{\om}L_{t}^{2}L_{x}^{2}$ to \eqref{equa-SNSE-High} with $\epsilon_{n}W$ replacing $W$.

Moreover, by (\ref{1.3.5.17}), 
\begin{align}\label{1.3.73.01}
	\|u^{(\epsilon_{n})}-u_{n}\|_{\Lo^{2r}_{\om}L_{t}^{2}L_{x}^{2}}
	&\leq \|u^{(\epsilon_{n})}-(v_{n}+z^{(\epsilon_{n})})\|_{\Lo^{2r}_{\om}L_{t}^{2}L_{x}^{2}}+\|(v_{n}+z^{(\epsilon_{n})})-u_{n}\|_{\Lo^{2r}_{\om}L_{t}^{2}L_{x}^{2}}\notag\\
%	&\leq \|u^{(\epsilon_{n})}-(v_{n}+z^{(\epsilon_{n})})\|_{\Lo^{2r}_{\om}L_{t}^{2}L_{x}^{2}}+\|z_{n}^{u}-z^{u}\|_{\Lo^{2r}_{\om}L_{t}^{2}L_{x}^{2}}+\|Z^{(\epsilon_{n})}\|_{\Lo^{2r}_{\om}L_{t}^{2}L_{x}^{2}}\notag\\
	&\leq\sum_{q\geq n}\|v_{n,q+1}-v_{n,q}\|_{\Lo^{2r}_{\om}L_{t}^{2}L_{x}^{2}}+\|z_{n}^{u}-z^{u}\|_{\Lo^{2r}_{\om}L_{t}^{2}L_{x}^{2}}+\|Z^{(\epsilon_{n})}\|_{\Lo^{2r}_{\om}L_{t}^{2}L_{x}^{2}},
\end{align}
where $z^{u}$ and $Z^{(\epsilon_{n})}$ are given by (\ref{1.3.5.12}).

In order to control  $\|z_{n}^{u}-z^{u}\|_{\Lo^{2r}_{\om}L_{t}^{2}L_{x}^{2}}$, we consider two different temporal intervals $(\lambda_{n}^{-1/2}, T]$
and   $[0,\lambda_{n}^{-1/2}]$.
In the first case where $t\in(\lambda_{n}^{-1/2}, T]$,  using  the standard mollification estimates and \citep[(5.30), (5.31)]{CLZ} we have
 \begin{align*}%\label{y1}
 	\|z_{n}^{u}-z^{u}\|_{L_{(\lambda_{n}^{-1/2}, T]}^{2}L_{x}^{2}}\lesssim \lambda_{n}^{-\frac{1}{2}}(\|z^{u}\|_{C_{(\lambda_{n}^{-1/2}-\lambda_{n}^{-1}, T]}^{\frac{1}{2}}L_{x}^{2}}+\|z^{u}\|_{C_{(\lambda_{n}^{-1/2}, T]}H_{x}^{\frac{1}{2}}})\lesssim \lambda_{n}^{-\frac{1}{2}} (1+\lambda_{n}^{\frac{1}{4}})M,
 	\end{align*}
where $M$ is as in (\ref{in}). Moreover, in the  case where $t\in[0,\lambda_{n}^{-1/2}]$, we have
  \begin{align*}%\label{y2}
 	\|z_{n}^{u}-z^{u}\|_{L_{[0,\lambda_{n}^{-1/2}]}^{2}L_{x}^{2}}\leq \|z_{n}^{u}\|_{L_{[0,\lambda_{n}^{-1/2}]}^{2}L_{x}^{2}}+\|z^{u}\|_{L_{[0,\lambda_{n}^{-1/2}]}^{2}L_{x}^{2}}\lesssim \lambda_{n}^{-1/4}\|z^{u}\|_{C_{[0,\lambda_{n}^{-1/2}]}L_{x}^{2}}\lesssim \lambda_{n}^{-\frac{1}{4}}M.
 \end{align*}
Hence, we obtain
  \begin{align}\label{y3}
	\|z_{n}^{u}-z^{u}\|_{L_{t}^{2}L_{x}^{2}}\lesssim \lambda_{n}^{-\frac{1}{2}} (1+\lambda_{n}^{\frac{1}{4}})M+\lambda_{n}^{-\frac{1}{4}}M\lesssim \lambda_{n}^{-\frac{1}{4}}M.
\end{align}
By   Proposition \ref{Proposition stochastic}, (\ref{1.3.73.01}) and (\ref{y3}), we obtain
\begin{align}\label{1.3.73.1}
	\|u^{(\epsilon_{n})}-u_{n}\|_{\Lo^{2r}_{\om}L_{t}^{2}L_{x}^{2}}
	&\lesssim 	\sum_{q\geq n} \dqq^{\frac{1}{2}}+\lambda_{n}^{-\frac{1}{4}}M+\lambda_{n}^{-1}(2r-1)^{\frac{1}{2}}L 	\lesssim \sum_{q\geq n} \dqq^{\frac{1}{2}}+\lambda_{n}^{-\frac{1}{4}}.
\end{align}
Therefore,  it follows from  \eqref{2-7}, \eqref{1.3.59} and \eqref{1.3.73.1} that  for $n> 10$,
\begin{align}\label{1.3.76-1}
	\|u^{(\epsilon_{n})}-u\|_{\Lo^{2r}_{\om}L_{t}^{2}L_{x}^{2}}
	\leq
	\|u^{(\epsilon_{n})}-u_{n}\|_{\Lo^{2r}_{\om}L_{t}^{2}L_{x}^{2}} 	+\|u-u_{n}\|_{\Lo^{2r}_{\om}L_{t}^{2}L_{x}^{2}}
 \lesssim \sum_{q\geq n}\dqq^{\frac{1}{2}}+\lambda_{n}^{-\frac{1}{4}}+\lambda_{n}^{-\widetilde{\beta}}\widetilde{M}
	 \leq \frac{1}{n}\to 0,
\end{align}
where we choose $a$ sufficiently large enough. 

Therefore, the proof of Theorem \ref{Theorem Vanishing noise} is complete.\hfill$\square$

\subsection{Mollification}
The remaining of the present paper is to prove the main iteration in Proposition \ref{Proposition Main iteration}.   In order to avoid the loss of derivatives in the convex integration scheme,
we mollify the velocity fields, Reynolds stress and noise as follows.

Let $\varrho \in C_c^{\infty}(\bbr^3 ; \mathbb{R}_{+})$  with supp $\varrho\subset B_1(0)$, and $\vartheta \in C_c^{\infty}(\bbr ; \mathbb{R}_{+})$ with ${\rm supp}\, \vartheta \subset[0,1]$.
Let
\begin{align}  \label{2-20}
	\varrho_{\ell} := \ell^{-3} \varrho(\cdot / \ell),\ \
	\vartheta_{\ell} := \ell^{-1} \vartheta(\cdot / \ell),
\end{align}
where 
 \begin{equation}\label{2-21}
 	\begin{aligned}
 		\ell:=\la^{-80}.
	\end{aligned}
\end{equation}
By  (\ref{2-6}) and (\ref{2-7}), we have
\begin{equation}\label{2-22}
	\begin{aligned}
	(8\cdot 80rL^{2}50^{q})^{50^{q}}\leq \la, \quad	\ell^{-170 }<\lambda_{q+1}^{\frac{1}{8}\varepsilon}.
	\end{aligned}
\end{equation}

Define the spatial-temporal mollifications of $ v_{q}$,   $z_{q}$, $\Ru_{q}$  as follows:
\begin{align}
       & v_{\ell}:=(v_{q} *_x \varrho_{\ell}) *_t \vartheta_{\ell},\label{2-23}\\
		& z_{\ell}:=z^{u}_{\ell}+Z_{\ell}\ \
          {\rm with} \ \
         	z^{u}_{\ell}:=(z^{u} *_x \varrho_{\ell}) *_t \vartheta_{\ell},\ \
           Z_{\ell}:=(Z_{q} *_x \varrho_{\ell}) *_t \vartheta_{\ell},\label{2-24}\\
	   &\Ru_{\ell}:=(\Ru_{q} *_x \varrho_{\ell}) *_t \vartheta_{\ell}. \label{2-25}
	\end{align}
 For $t<0$ and $t>T$, we let $v_{q}$, $\ru$,  and $z_{q}$ be zero.\\
Then by (\ref{2-12}), $ (v_{\ell}, \Ru_{\ell}) $ satisfies
\begin{align}\label{2-26}
	\left\{\begin{array}{lc}
	\p_{t}v_{\ell}+\nu(-\Delta)^{\a}v_{\ell}-z_{\ell}
 +\div\big((v_{\ell}+z_{\ell})\otimes(v_{\ell}+z_{\ell})\big)+\nabla P_{\ell}=\div( \Ru_{\ell}+\Ru_{com1}),\\
		\div v_{\ell}=0,\\
		 v_{\ell}(0)=0,
	\end{array}\right.
\end{align}
where the traceless symmetric commutator stress $ \Ru_{com1} $ is of form
\begin{equation}\label{2-27}
	\begin{aligned}
	\Ru_{com1}:=&(v_{\ell}+z_{\ell})\mathring{\otimes}(v_{\ell}+z_{\ell})-\((v_{q}+z_{q})\mathring{\otimes}(v_{q}+z_{q})\)*_x \varrho_{\ell} *_t \vartheta_{\ell},
	\end{aligned}
\end{equation}
and the pressure $	P_{\ell} $ is given by
\begin{equation}\label{3.6}
	\begin{aligned}
		P_{\ell}:=(P_{q} *_x \varrho_{\ell}) *_t \vartheta_{\ell}-\frac{1}{3}|v_{\ell}+z_{\ell}|^{2}+\frac{1}{3}\((|v_{q}+z_{q}|^{2})*_x \varrho_{\ell}\) *_t \vartheta_{\ell}.
	\end{aligned}
\end{equation}

\section{Velocity perturbations in the supercritical regime \texorpdfstring{$\mathcal{S}_{1}$}{S1}} \label{Sec-S1} 

From this section to Section \ref{Sec-Reynolds} we aim to prove the main iteration in Proposition \ref{Proposition Main iteration} in the first supercritical regime $\mathcal{S}_{1}$. 
Let us first deal with the velocity perturbations in this section.

\subsection{Intermittent  jets}  
 
We use  the spatial-temporal intermittent building blocks which are indexed by five parameters $r_{\perp}$, $r_{\parallel}$, $\lambda$, $\tau$ and $\sigma$ :
\begin{equation}\label{3-1}
	\begin{aligned}
r_{\perp}:=\lambda_{q+1}^{-1+2 \varepsilon},\quad r_{\parallel}:=\lambda_{q+1}^{-1+4 \varepsilon},\quad \lambda:=\lambda_{q+1}, \quad
\tau:=\lambda_{q+1}^{4\alpha-5+11\varepsilon},\quad
\mu:=\lambda_{q+1}^{2\alpha-1+2\varepsilon},\quad \sigma:=\lambda_{q+1}^{2 \varepsilon},
	\end{aligned}
\end{equation}
where $\varepsilon$ satisfies (\ref{2-8}). Note that
\begin{align}\label{3-1-2}
\sigma\tau<r_{\perp}r_{\parallel}^{-1}\lambda_{q+1}\mu, \ \ \lambda_{q+1}<r_{\perp}r_{\parallel}^{-1}\lambda_{q+1}\mu,  \ \ \ell^{-12}\lambda_{q+1}^{-1}<r_{\perp}r_{\parallel}^{-1}<1.
\end{align}

\medskip 
\paragraph{\bf $\bullet $ {Spatial building blocks.}} 
Let $\Phi: \bbr^{2} \rightarrow \bbr$ be a smooth cut-off function supported on a ball of radius 1 and normalize $\Phi$ such that $\phi:=-\Delta \Phi$ satisfies
\begin{equation}\label{3-2}
	\begin{aligned}
\frac{1}{2 \pi} \int_{\bbr} \phi^2(x) dx=1.
	\end{aligned}
\end{equation}
Let $\psi: \bbr \rightarrow \bbr$ be a smooth and mean-zero function supported on $[-1,1]$ and satisfy
\begin{equation}\label{3-3}
	\begin{aligned}
		\frac{1}{2 \pi} \int_{\bbr} \psi^2(x) dx=1.
	\end{aligned}
\end{equation}
The  rescaled cut-off functions are defined by
\begin{flalign*}
\phi_{r_{\perp}}(x_{1},x_{2}):=r_{\perp}^{-1}\phi(\frac{x_{1}}{r_{\perp}},\frac{x_{2}}{r_{\perp}}), \quad 
\Phi_{r_{\perp}}(x_{1},x_{2}):=r_{\perp}^{-1}\Phi(\frac{x_{1}}{r_{\perp}},\frac{x_{2}}{r_{\perp}}), \quad 
\psi_{r_{\perp}}(x):=r_{\parallel}^{-\frac{1}{2}} \psi\left(\frac{x}{r_{\parallel}}\right) .
\end{flalign*}
 By an abuse of notation, we periodize $\phi_{r_{\perp}}$, $\Phi_{r_{\perp}}$ and $\psi_{r_{\parallel}}$, so that  $\phi_{r_{\perp}}$, $\Phi_{r_{\perp}}$ are treated as a periodic function defined on $\T^{2}$ and $\psi_{r_{\parallel}}$ is treated as a periodic function defined on $\T$.
 
The intermittent jets first introduced in \citep{BCV} are defined by
\begin{equation}\label{3-4}
	\begin{aligned}
W_{(k)}:=\psi_{r_{\parallel}}\big(\lambda r_{\perp} N_{\Lambda}( k_{1} \cdot x+\mu t)\big)\phi_{r_{\perp}}\big(\lambda r_{\perp} N_{\Lambda} k \cdot x, \lambda r_{\perp} N_{\Lambda} k_{2} \cdot x\big) k_1, \quad k \in \Lambda.
	\end{aligned}
\end{equation}
Here,  $\left(k, k_1, k_2\right)\subseteq \bbr^{3}$ are the orthonormal bases as in the Geometric Lemma \ref{Lemma First Geometric}, $N_{\Lambda}$ is given by (\ref{6.3}) below, and  $\Lambda$ is the wave vector set. 
% We have the flexibility to select the shift  $a_{k}\in\bbr^{3}$ in such a manner that the supports of $W_{k}$ and $W_{k'}$ are disjiont for $k\neq k'$. By making  $r_{\perp}$ sufficiently small, we can  guarantee the existence of such $a_{k}$  (see e.g. \citep{BV-convex}).  
 Note that, $\{W_{(k)}\}$ are $\left(\mathbb{T} /(\lambda r_{\perp})\right)^3$-periodic.

For  simplicity, we set
\begin{align}
&\psi_{(k_{1})}(x):=\psi_{r_{\parallel}}\big(\lambda r_{\perp} N_{\Lambda}( k_{1} \cdot x+\mu t)\big),\label{3-5}\\
&\phi_{(k)}(x):=\phi_{r_{\perp}}\big(\lambda r_{\perp} N_{\Lambda} k \cdot x,\lambda r_{\perp} N_{\Lambda} k_{2} \cdot x\big), \label{3-6}\\
&\Phi_{(k)}(x):=\Phi_{r_{\perp}}\big(\lambda r_{\perp} N_{\Lambda} k \cdot x,\lambda r_{\perp} N_{\Lambda} k_{2} \cdot x\big),\label{3-7}
\end{align}
and rewrite
\begin{align} \label{3-8}
	  W_{(k)}=\psi_{(k_{1})}\phi_{(k)} k_1,\ k \in \Lambda.  
\end{align}

Since $W_{(k)}$ is not divergence-free,  we  need a corrector
	\begin{align}\label{3-9}
	\widetilde{W}_{(k)}^c & :=\frac{1}{\lambda^2 N_{\Lambda}^2}\nabla\psi_{(k_{1})}\times \curl(\Phi_{(k)} k_1), \quad k \in \Lambda, 
\end{align}
and set
	\begin{align}\label{3-10}
	W_{(k)}^c & :=\frac{1}{\lambda^2 N_{\Lambda}^2}\psi_{(k_{1})}\Phi_{(k)} k_1, \quad k \in \Lambda.
\end{align}
Then, one has 
	\begin{align}\label{3-11}
	W_{(k)}+\widetilde{W}_{(k)}^c& :=\curl\curl(\frac{1}{\lambda^2 N_{\Lambda}^2}\psi_{(k_{1})}\Phi_{(k)} k_1)=\curl\curl W_{(k)}^c,
\end{align}
which yields that
	\begin{align}\label{3-12}
	\div (W_{(k)}+\widetilde{W}_{(k)}^c)& =0.
\end{align}

The following lemma provides the analytical estimates for intermittent jets.

\begin{lemma} [{\citep{BV-convex}}, Estimates of intermittent jets] \label{Lemma spatial building blocks}  

For any $p \in[1, \infty]$ and $N \in \bbn$, one has
\begin{align}
	&\|\nabla^N\p_{t}^{M} \psi_{(k_{1})}\|_{L_x^p} \lesssim r_{\parallel}^{\frac{1}{p}-\frac{1}{2}}\(\frac{r_{\perp}\lambda}{r_{\parallel}}\)^{N}\(\frac{r_{\perp}\lambda\mu}{r_{\parallel}}\)^{M},\label{3-13}\\
&\|\nabla^N \phi_{(k)}\|_{L_x^p}+\|\nabla^N \Phi_{(k)}\|_{L_x^p} \lesssim r_{\perp}^{\frac{2}{p}-1}\lambda^N.\label{3-14}
\end{align}
Moreover,
	\begin{align}\label{3-15}
	&\|\nabla^N\p_{t}^{M}  W_{(k)}\|_{C_{t} L_x^p}+\frac{r_{\parallel}}{r_{\perp}}\|\nabla^N\p_{t}^{M}  \widetilde{W}_{(k)}\|_{ C_{t}L_x^p}+\lambda^2\|\nabla^N\p_{t}^{M} W_{(k)}^c\|_{ C_{t}L_x^p} \lesssim r_{\perp}^{\frac{2}{p}-1}r_{\parallel}^{\frac{1}{p}-\frac{1}{2}} \lambda^N\(\frac{r_{\perp}\lambda\mu}{r_{\parallel}}\)^{M}, ~ k \in \Lambda.
\end{align} 
The implicit constants above are deterministic and  independent of the parameters $r_{\perp}$, $r_{\parallel}$, $\mu$ and $\lambda$.
\end{lemma}

\paragraph{\bf $\bullet$ Temporal building blocks.} 
To address the hyper-viscosity,   especially for the case beyond the Lions exponent, we 
also need the  temporal intermittency  in building blocks, which were first introduced in \citep{CL-Sobolev transport}.

Let $\{g_{k}\}_{k\in \Lambda}\subset C_c^{\infty}([0, T])$ be any cut-off function satisfying
\begin{equation}\label{3-16}
	\begin{aligned}
\fint_0^T g_{k}^{2}(t) \mathrm{d} t=1,\ \ k\in \Lambda,
\end{aligned}
\end{equation}
and $g_{k}$, $g_{k'}$ have disjoint temporal supports if $k\neq k'$. Note that, one can choose $g_{k}=g(t-\a_{k})$, where $g\in C_c^{\infty}([0, T])$ with very small support and $\{\a_{k}\}_{k\in\Lambda}$ are  suitable temporal shifts to guarantee the disjoint supports of $\{g_{k}\}_{k\in \Lambda}$.

We rescale  $g_{k}$ by
\begin{equation}\label{3-17}
	\begin{aligned}
g_{k,\tau}(t):=\tau^{\frac{1}{2}} g_{k}(\tau t) ,\ \ k\in \Lambda,
\end{aligned}
\end{equation}
and then  periodize  $ g_{k,\tau} $ (still denoted by $ g_{k, \tau}$) to be treated as a periodic function on  $[0, T]$.  Moreover,  let
\begin{equation}\label{3-18}
	\begin{aligned}
h_{k,\tau}(t):=\int_0^t\left(g_{k,\tau}^2(s)-1\right)  \mathrm{d} s,\quad t\in [0, T],\ \ k\in \Lambda, 
\end{aligned}
\end{equation}
and
\begin{equation}\label{3-19}
	\begin{aligned}
g_{(k)}:=g_{k,\tau}(\sigma t), \quad h_{(k)}(t):=h_{k,\tau}(\sigma t) .
\end{aligned}
\end{equation}
The following Lemma gives the estimates of the temporal building blocks.

\begin{lemma} 
[\citep{CL-transport equation,CL-NSE1}, Estimates of temporal building blocks] 
\label{Lemma temporal building blocks} 
For any  $\gamma \in[1,+\infty]$, $M \in \mathbb{N}$, we have
\begin{equation}\label{3-20}
	\begin{aligned}
\|\partial_t^M g_{(k)}\|_{L_t^\gamma} \lesssim \sigma^M \tau^{M+\frac{1}{2}-\frac{1}{\gamma}},\ \ k\in \Lambda,
\end{aligned}
\end{equation}
where the implicit constants are deterministic and  independent of $\tau$ and $\sigma$. Moreover, it holds that
\begin{equation}\label{3-21}
	\begin{aligned}
\|h_{(k)}\|_{L_t^{\infty}} \lesssim 1,\quad
\|h_{(k)}\|_{C_t^{N}}\lesssim (\sigma\tau)^{N},\quad  N\geq1,\ \ k\in \Lambda.
\end{aligned}
\end{equation}
\end{lemma}

\subsection{Amplitudes} 

The amplititudes of  velocity perturbations are important to decrease the old Reynolds stress in order  to satisfy the interative decay estimate \eqref{2-16}. One key role here is played by the following Geometic Lemma.

\begin{lemma} [\textbf{Geometric Lemma}, {\citep{BV-convex}}]\label{Lemma First Geometric}
	 There exists a set $\Lambda \subset \mathbb{S}^2 \cap \mathbb{Q}^3$ that consists of vectors $k$ with associated orthonormal bases $\left(k, k_1, k_2\right)$, $\varepsilon_u>0$, and smooth positive functions $\gamma_{(k)}$ : $B_{\varepsilon_u}(\mathrm{Id}) \rightarrow \mathbb{R}$, where $B_{\varepsilon_u}(\mathrm{Id})$ is the ball of radius $\varepsilon_u$ centered at the identity in the space of $3 \times 3$ symmetric matrices, such that for $S \in B_{\varepsilon_u}(\mathrm{Id})$ we have the following identity:
\begin{equation}\label{6.1}
	\begin{aligned}
S=\sum_{k \in \Lambda} \gamma_{(k)}^2(S) k_1 \otimes k_1.
	\end{aligned}
 \end{equation}
\end{lemma}

We note that there exists $N_{\Lambda} \in \mathbb{N}$ such that
\begin{equation}\label{6.3}
	\begin{aligned}
\left\{N_{\Lambda} k, N_{\Lambda} k_1, N_{\Lambda} k_2\right\} \subset N_{\Lambda} \mathbb{S}^2 \cap \mathbb{Z}^3.
\end{aligned}
\end{equation}
Let  $M_*$  denote the geometric constant such that
\begin{equation}\label{6.4}
	\begin{aligned}
\sum_{k \in \Lambda}\|\gamma_{(k)}\|_{C^4\left(B_{\varepsilon_u}(\mathrm{Id})\right)} \leq M_* .
\end{aligned}
\end{equation}

\medskip 
Let $\chi:[0,+\infty) \rightarrow \mathbb{R}$ be a smooth cut-off function such that
\begin{equation}\label{3-22}
	\begin{aligned}
\chi(z)=\left\{\begin{array}{ll}
	1, & 0 \leq z \leq 1; \\
	z, & z \geq 2;
\end{array}\right.
\end{aligned}
\end{equation}
and
\begin{equation}\label{3-23}
	\begin{aligned}
\frac{1}{2} z \leq \chi(z) \leq 2 z \quad \text { for } \quad z \in(1,2) .
\end{aligned}
\end{equation}

Let
\begin{equation}\label{3-24}
	\begin{aligned}
\varrho(t, x):=2 \varepsilon_{u}^{-1} \dqq \chi\big(\frac{|\Ru_{\ell}(t, x)|}{\dqq}\big),
\end{aligned}
\end{equation}
where $\varepsilon_{u}$ is the small radius in Geometric Lemma \ref{Lemma First Geometric}. Note that 
$\Ru_{\ell}/\varrho \in B_{\varepsilon_{u}}(0)$. 

Define the amplitudes of the velocity perturbations by
\begin{equation}\label{3-25}
	\begin{aligned}
a_{(k)}(t,x):=\varrho^{1/2}(t,x)\gamma_{(k)}\big(\mathrm{Id}-\frac{\Ru_{\ell}(t,x)}{\varrho(t,x)}\big), \quad k\in \Lambda,
\end{aligned}
\end{equation}
where $\gamma_{(k)}$ is the smooth function in Geometric Lemma \ref{Lemma First Geometric}. 
Note that $a_{(k)}$  is $(\ft)$-adapted.

Applying  Geometric Lemma \ref{Lemma First Geometric}, we derive the following algebraic identity, which is impotant to decrease the impact of the old Reynolds stress:
\begin{align}\label{3-26-0}
	\sum_{k \in \Lambda} a_{(k)}^2 g_{(k)}^2(W_{(k)} \otimes W_{(k)})
	= & \varrho\mathrm{Id} -\Ru_{\ell}+\sum_{k \in \Lambda} a_{(k)}^2 g_{(k)}^2 \bbp_{\neq 0}(W_{(k)} \otimes W_{(k)}) \notag\\
	& +\sum_{k \in \Lambda} a_{(k)}^2(g_{(k)}^2-1) \fint_{\T^3} W_{(k)} \otimes W_{(k)}  \mathrm{d} x.
\end{align}
Here, $\bbp_{\neq 0}$ denotes the spatial projection onto nonzero Fourier modes. 

Moreover, one has the following estimates.

\begin {lemma} [Estimates of amplitudes]  \label{Lemma amplitudes 1} It holds that for $0 \leq N \leq 9$,  $k \in \Lambda$,
\begin{equation}\label{3-26}
	\begin{aligned}
\|a_{(k)}\|_{C_{t, x}^N}
 \lesssim (1+\|\Ru_{q}\|_{C_{t}L_{x}^{1}}^{N+2})\ell^{-6N-7},\quad
 	\|a_{(k)}\|_{L_{t}^{2}L_{x}^{2}}
 \lesssim \delta_{q+1}^{\frac{1}{2}}+\|\Ru_{q}\|^{\frac{1}{2}}_{L_{t}^{1}L_{x}^{1}},
\quad \|a_{(k)}\|_{\Lo^{2r}_{\om}L_{t}^{2}L_{x}^{2}}
 \lesssim \delta_{q+1}^{\frac{1}{2}}, 
\end{aligned}
\end{equation}
where the implicit constants are deterministic and independent of $q$.
\end{lemma}

Since  Lemma \ref{Lemma amplitudes 1} can be proved by 
using the arguments  of Lemma 4.3 in \citep{CLZ}, we omit the detailed proof here.

\subsection{Velocity perturbations} \label{subsec3.3}
The velocity perturbations 
consist of 
the principle part, the incompressibility corrector and two types of temporal correctors.

\medskip 
\paragraph{\bf $\bullet$ Principal part.} 
The principal part of the velocity perturbations is defined by  
\begin{equation}\label{3-43}
	\begin{aligned}
		w_{q+1}^{(p)}:=\sum_{k \in \Lambda} a_{(k)} g_{(k)} W_{(k)}.
\end{aligned}\end{equation}
A keypoint  is that the zero-frequency component of the tensor product $w_{q+1}^{(p)}\otimes w_{q+1}^{(p)}$ can decrease the effect of the  Reynolds stress $\Ru_{\ell}$:
\begin{align}\label{3-43-1}
w_{q+1}^{(p)}\otimes w_{q+1}^{(p)}+\Ru_{\ell}
	= & \varrho\mathrm{Id}+\sum_{k \in \Lambda} a_{(k)}^2 g_{(k)}^2 \bbp_{\neq 0}(W_{(k)} \otimes W_{(k)}) \notag\\
	& +\sum_{k \in \Lambda} a_{(k)}^2(g_{(k)}^2-1) \fint_{\T^3} W_{(k)} \otimes W_{(k)}  \mathrm{d} x.
\end{align}

\medskip 
\paragraph{\bf $\bullet$ Incompressibility corrector.} 
Since $w_{q+1}^{(p)}$ is not divergence-free, we need the corresponding incompressibility corrector  defined by
\begin{equation}\label{3-44}
	\begin{aligned}
	w_{q+1}^{(c)} & :=\sum_{k \in \Lambda } g_{(k)}\left(\curl (\nabla a_{(k)} \times  W_{(k)}^c)+\nabla a_{(k)}\times \curl W_{(k)}^c+a_{(k)} \widetilde{W}_{(k)}^c\right) .
\end{aligned}\end{equation}
Note that
\begin{equation}\label{3-45}
	\begin{aligned}
	w_{q+1}^{(p)}+w_{q+1}^{(c)}&=\sum_{k \in \Lambda} \operatorname{curlcurl}(a_{(k)} g_{(k)} W_{(k)}^c),
\end{aligned}\end{equation}
and so, one has the incompressibility:
\begin{equation}\label{3-46}
	\begin{aligned}
\operatorname{div}(w_{q+1}^{(p)}+w_{q+1}^{(c)})=0.
\end{aligned}\end{equation}

\medskip 
\paragraph{\bf $ \bullet$ Temporal corrector to balance spatial oscillations.}  
The first  temporal corrector to balance the spatial oscillations is defined by
\begin{equation}\label{3-47}
	\begin{aligned}
		w_{q+1}^{(t)}:= & -\mu^{-1} \sum_{k \in \Lambda} \mathbb{P}_H \mathbb{P}_{\neq 0}(a_{(k)}^2g_{(k)}^{2}\psi_{(k_{1})}^{2}\phi_{(k)}^{2}).
\end{aligned}\end{equation}
Here,  $ \mathbb{P}_H =\mathrm{Id}-\nabla\Delta^{-1}\div $ denotes  the Helmholtz projection.  The  purpose of constructing  $w_{q+1}^{(t)}$ is to decrease the high spatial frequency oscillations of $\div(W_{(k)}\otimes W_{(k)})$.  Actually,  applying Leibniz's rule we derive
\begin{align}\label{3-47-2}
	\p_{t}w_{q+1}^{(t)}+\sum_{k \in   \Lambda}\mathbb{P}_{\neq 0}\(a_{(k)}^{2}g_{(k)}^{2}\div(W_{(k)}\otimes W_{(k)})\)=&  -\mu^{-1} \sum_{k \in \Lambda} \mathbb{P}_{\neq 0}\(\p_{t}(a_{(k)}^{2}g_{(k)}^{2})\psi_{(k_{1})}^{2}\phi_{(k)}^{2}\)\notag\\
	&+(\nabla\Delta^{-1}\div)\mu^{-1}\mathbb{P}_{\neq 0}\p_{t}(a_{(k)}^{2}g_{(k)}^{2}\psi_{(k_{1})}^{2}\phi_{(k)}^{2}).
\end{align}
Note that, the righ-hand side of (\ref{3-47-2}) only contains the low spatial frequency part $\p_{t}(a_{(k)}^{2}g_{(k)}^{2})$ and the pressure term.

\medskip 
\paragraph{\bf $\bullet$ Temporal corrector to balance temporal oscillations.} In order to  cancel the high temporal oscillation frequency generated by $g_{(k)}$ in (\ref{3-26-0}),  we use another temporal corrector  defined by
\begin{equation}\label{3-48}
	\begin{aligned}
	w_{q+1}^{(o)}:= & -\sigma^{-1} \sum_{k \in \Lambda} \mathbb{P}_H \mathbb{P}_{\neq 0}\left(h_{(k)} \fint_{\mathbb{T}^3} W_{(k)} \otimes W_{(k)} \mathrm{d} x \nabla(a_{(k)}^2)\right). 
\end{aligned}\end{equation}
It follows from (\ref{3-18}) and the Leibniz rule that
	\begin{align}\label{3-48-2}
	&\quad\p_{t}w_{q+1}^{(o)}+\sum_{k \in   \Lambda}\mathbb{P}_{\neq 0}\((g_{(k)}^{2}-1)\fint_{\T^3}W_{(k)}\otimes W_{(k)}\mathrm{d}x\nabla (a_{(k)}^{2})\)\notag\\
	&=  -\sigma^{-1} \sum_{k \in \Lambda} \mathbb{P}_{\neq 0}\(h_{(k)}\fint_{\T^3}W_{(k)}\otimes W_{(k)}\mathrm{d}x\p_{t}\nabla (a_{(k)}^{2})\)\notag\\
	&\quad+(\nabla\Delta^{-1}\div)\sigma^{-1}\mathbb{P}_{\neq 0}\p_{t}\left(h_{(k)} \fint_{\mathbb{T}^3} W_{(k)} \otimes W_{(k)} \mathrm{d} x \nabla(a_{(k)}^2)\right).
\end{align}
We note that  the right-hand side of (\ref{3-48-2}) contains the low spatial frequency part $\p_{t}\nabla(a_{(k)}^{2})$ and the pressure term.

\medskip 
\paragraph{\bf $\bullet$ Cut-off perturbations} Next, we choose a smooth cut-off function $\Theta_{q+1} \in C^{\infty}([0, T])$ as follows
\begin{equation}\label{3-49}
	\begin{aligned}
\Theta_{q+1}(t)\in[0,1], \ \ \Theta_{q+1}(t)=\left\{\begin{array}{ll}
	0,\quad t \leq \varsigma_{q}/ 2, \\
	1,\quad \varsigma_{q}\leq t \leq T,
\end{array}
%\quad \Theta(t)\textcolor{orange}{\in(0,1), \ \ t\in ( \varsigma_{q}/ 2,  \varsigma_{q})},
\quad \text { and } \quad\|\Theta_{\mathrm{q+1}}\|_{C^{n}_{t}} \lesssim \varsigma_{q}^{-n},\right.
\end{aligned}\end{equation}
where $\varsigma_{q}=\lambda_{q}^{-40}$ as in (\ref{2-6}).
Then, set
\begin{equation}\label{3-50}
	\begin{aligned}
&\widetilde{w}_{q+1}^{(p)}:=\Theta_{q+1} w_{q+1}^{(p)}, \quad \widetilde{w}_{q+1}^{(c)}:=\Theta_{q+1} w_{q+1}^{(c)}, \quad \widetilde{w}_{q+1}^{(t)}:=\Theta_{q+1}^2 w_{q+1}^{(t)},\quad\widetilde{w}_{q+1}^{(o)}:=\Theta_{q+1}^2 w_{q+1}^{(o)}.
\end{aligned}\end{equation}
For  the convenience, we set
\begin{align}
	&w^{(*_{1})+(*_{2})}_{q+1}:=w^{(*_{1})}_{q+1}+w^{(*_{2})}_{q+1},\ \ \omw^{(*_{1})+(*_{2})}_{q+1}=\omw^{(*_{1})}_{q+1}+\omw^{(*_{2})}_{q+1}, \quad\text{where} \quad *_{1},~*_{2}\in \{p, c, t, o \}. \label{pco-1}
\end{align}

The velocity perturbation $w_{q+1}$ at level $q+1$ is defined by
\begin{equation}\label{3-54}
	\begin{aligned}
w_{q+1}=\widetilde{w}_{q+1}^{(p)}+\widetilde{w}_{q+1}^{(c)}+\widetilde{w}_{q+1}^{(t)}+\widetilde{w}_{q+1}^{(o)}. 
\end{aligned}\end{equation}
Note that, $w_{q+1}$ is  mean-zero, divergence-free and $\left(\mathcal{F}_t\right)$-adapted.

Consequently, the  velocity field at level ${q+1}$  is defined by
\begin{equation}\label{3-55}
	\begin{aligned}
v_{q+1}:=v_{\ell}+w_{q+1}.
\end{aligned}\end{equation}

In the end of this section,  we  give the estimates of the velocity  perturbations in the following lemma.
\begin{lemma}[Estimates of velocity perturbations]\label{Lemma Estimates of perturbations}
	For any $\rho \in(1, \infty), \gamma \in[1, \infty]$ and every integer $0 \leq N \leq 5$, the following estimates hold :
	\begin{align}
			&\|\nabla^N w_{q+1}^{(p)}\|_{L_t^\gamma L_x^\rho} \lesssim  (1+\|\Ru_{q}\|_{C_{t}L_{x}^{1}}^{N+2})\ell^{-7}r_{\perp}^{\frac{2}{\rho}-1}r_{\parallel}^{\frac{1}{\rho}-\frac{1}{2}}\laq^{N}\tau^{\frac{1}{2}-\frac{1}{\gamma}},\label{3-56} \\
			&\|\nabla^N w_{q+1}^{(c)}\|_{L_t^\gamma L_x^\rho}
			\lesssim 	(1+\|\Ru_{q}\|_{C_{t}L_{x}^{1}}^{N+4})\ell^{-7}r_{\perp}^{\frac{2}{\rho}}r_{\parallel}^{\frac{1}{\rho}-\frac{3}{2}}\laq^{N}\tau^{\frac{1}{2}-\frac{1}{\gamma}},\label{3-57}\\
			&\|\nabla^N w_{q+1}^{(t)}\|_{L_t^\gamma L_x^\rho} \lesssim (1+\|\Ru_{q}\|_{C_{t}L_{x}^{1}}^{N+4})\ell^{-14}\mu^{-1}r_{\perp}^{\frac{2}{\rho}-2}r_{\parallel}^{\frac{1}{\rho}-1}\laq^{N}\tau^{1-\frac{1}{\gamma}},\label{3-58}\\
			&\|\nabla^N w_{q+1}^{(o)}\|_{L_t^\gamma L_x^\rho} \lesssim (1+\|\Ru_{q}\|_{C_{t}L_{x}^{1}}^{N+5})\ell^{-6N-20}\sigma^{-1}.\label{3-59}
			\end{align}
	 Moreover,
		\begin{align}
		&\|w_{q+1}^{(p)}\|_{C_{t,x}^{N}}
			\lesssim
		(1+\|\Ru_{q}\|_{C_{t}L_{x}^{1}}^{N+2})\lambda_{q+1}^{2\a(N+1)+1},\label{3-60}\\
		&\|w_{q+1}^{(c)}\|_{C_{t,x}^{N}}
		\lesssim
		(1+\|\Ru_{q}\|_{C_{t}L_{x}^{1}}^{N+4})\lambda_{q+1}^{2\a(N+1)+1},\label{3-61}\\	
		&\|w_{q+1}^{(t)}\|_{C_{t,x}^{N}}
		\lesssim
		(1+\|\Ru_{q}\|_{C_{t}L_{x}^{1}}^{N+5})\lambda_{q+1}^{2\a(N+1)+1},\label{3-62}\\
		&\|w_{q+1}^{(o)}\|_{C_{t,x}^{N}}
		\lesssim
		(1+\|\Ru_{q}\|_{C_{t}L_{x}^{1}}^{N+6})\lambda_{q+1}^{4\a N-5N+3}.\label{3-63}
		\end{align}
The above implicit constants are deterministic and  independent of $q$.
\end{lemma}
\textit{Proof.} By Lemmas \ref{Lemma spatial building blocks}, \ref{Lemma temporal building blocks}, \ref{Lemma amplitudes 1} and (\ref{3-43}), we get
\begin{align}\label{3-64}
\|\nabla^N w_{q+1}^{(p)}\|_{L_t^\gamma L_x^\rho}
&\lesssim \sum_{k \in \Lambda}   \sum_{N_{1}+N_{2}=N}\|a_{(k)}\|_{C_{t,x}^{N_{1}}}\|\nabla^{N_{2}} W_{(k)}\|_{C_{t}L_x^\rho}\|g_{(k)}\|_{L_t^\gamma}\notag\\
%&\lesssim \sum_{N_{1}+N_{2}=N}(1+\|\Ru_{q}\|_{C_{t}L_{x}^{1}}^{N_{1}+2})\ell^{-6N_{1}-7} r_{\perp}^{\frac{2}{\rho}-1}r_{\parallel}^{\frac{1}{\rho}-\frac{1}{2}}\laq^{N_{2}}\tau^{\frac{1}{2}-\frac{1}{\gamma}}\notag\\
&\lesssim (1+\|\Ru_{q}\|_{C_{t}L_{x}^{1}}^{N+2})\ell^{-7} r_{\perp}^{\frac{2}{\rho}-1}r_{\parallel}^{\frac{1}{\rho}-\frac{1}{2}}\laq^{N}\tau^{\frac{1}{2}-\frac{1}{\gamma}},
\end{align}
where the last step was due to $\ell^{-7}<\lambda_{q+1}$, which verifies (\ref{3-56}).

Similarly, we obtain
	\begin{align}\label{3-65}
		 \|\nabla^N w_{q+1}^{(c)}\|_{L_t^\gamma L_x^\rho} 
	&\lesssim\sum_{k \in \Lambda}   \big(\sum_{N_{1}+N_{2}=N}\| a_{(k)}\|_{C_{t,x}^{N_{1}+2}} \| \nabla^{N_{2}}W_{(k)}^{c}\|_{C_{t}W_{x}^{1,\rho}}+\| a_{(k)}\|_{C_{t,x}^{N_{1}}}	  \|\nabla^{N_{2}}\widetilde{W}_{(k)}^{c}\|_{C_{t}L_{x}^{\rho}}\big)\|g_{(k)}\|_{L_t^\gamma}\notag\\
%&\lesssim\big(\sum_{N_{1}+N_{2}=N}(1+\|\Ru_{q}\|_{C_{t}L_{x}^{1}}^{N_{1}+4})\ell^{-6(N_{1}+2)-7} r_{\perp}^{\frac{2}{\rho}-1}r_{\parallel}^{\frac{1}{\rho}-\frac{1}{2}}\laq^{N_{2}+1}\laq^{-2}\notag\\
%&\hspace{2cm}  +(1+\|\Ru_{q}\|_{C_{t}L_{x}^{1}}^{N_{1}+2})\ell^{-6N_{1}-7}r_{\perp}^{\frac{2}{\rho}}r_{\parallel}^{\frac{1}{\rho}-\frac{3}{2}}\laq^{N_{2}}\big)\tau^{\frac{1}{2}-\frac{1}{\gamma}}\notag\\
%&\lesssim (1+\|\Ru_{q}\|_{C_{t}L_{x}^{1}}^{N+4})(\ell^{-19} r_{\perp}^{\frac{2}{\rho}-1}r_{\parallel}^{\frac{1}{\rho}-\frac{1}{2}}\laq^{N-1}+\ell^{-7} r_{\perp}^{\frac{2}{\rho}}r_{\parallel}^{\frac{1}{\rho}-\frac{3}{2}}\laq^{N})\tau^{\frac{1}{2}-\frac{1}{\gamma}}\notag\\
		&\lesssim
		(1+\|\Ru_{q}\|_{C_{t}L_{x}^{1}}^{N+4})\ell^{-7} r_{\perp}^{\frac{2}{\rho}}r_{\parallel}^{\frac{1}{\rho}-\frac{3}{2}}\laq^{N}\tau^{\frac{1}{2}-\frac{1}{\gamma}},
\end{align}
where the last step was due to \eqref{3-1-2},  which verifies (\ref{3-57}).

For the temporal correctors, using Lemmas \ref{Lemma spatial building blocks}, \ref{Lemma temporal building blocks}, \ref{Lemma amplitudes 1},  \eqref{3-47} and \eqref{3-48}, we have
	\begin{align}\label{3-66}
	\|\nabla^N w_{q+1}^{(t)}\|_{L_t^\gamma L_x^\rho}
	&\lesssim
	\mu^{-1} \sum_{k \in \Lambda} \sum_{N_{1}+N_{2}+N_{3}=N}\|a_{(k)}^{2}\|_{C_{t,x}^{N_{1}}}\|\nabla^{N_{2}}\psi_{(k_{1})}^{2}\|_{C_{t}L_{x}^{\rho}}\|\nabla^{N_{3}}\phi_{(k)}^{2}\|_{L_{x}^{\rho}}\|g_{(k)}\|^{2}_{L_{t}^{2\gamma}}\notag\\
	%&\lesssim \mu^{-1} \sum_{k \in \Lambda}\sum_{N_{1}+N_{2}+N_{3}=N}(1+\|\Ru_{q}\|_{C_{t}L_{x}^{1}}^{N_{1}+4})\ell^{-6N_{1}-14}r_{\parallel}^{\frac{1}{\rho}-1}(\frac{r_{\perp}\lambda_{q+1}}{r_{\parallel}})^{N_{2}}r_{\perp}^{\frac{2}{\rho}-2}\laq^{N_{3}}\tau^{1-\frac{1}{\gamma}}\notag\\
	&\lesssim
	(1+\|\Ru_{q}\|_{C_{t}L_{x}^{1}}^{N+4})\ell^{-14}\mu^{-1}r_{\perp}^{\frac{2}{\rho}-2}r_{\parallel}^{\frac{1}{\rho}-1}\lambda_{q+1}^{N}\tau^{1-\frac{1}{\gamma}},
\end{align}
%the last is due to $r_{\perp}<r_{\parallel}$ and $\ell^{-6}<\lambda_{q+1}$, which verifies (\ref{3-58}).
%Similarly, by Lemmas \ref{Lemma spatial building blocks}-\ref{Lemma amplitudes 1}  and (\ref{3-48}), we have
and
	\begin{align}\label{3-67}
		\|\nabla^N w_{q+1}^{(o)}\|_{L_t^\gamma L_x^\rho}
		%&\lesssim \sigma^{-1} \sum_{k \in \Lambda} \|\nabla^N\mathbb{P}_H \mathbb{P}_{\neq 0}\big(h_{(k)} \fint_{\mathbb{T}^3} W_{(k)} \otimes W_{(k)} \mathrm{d} x (a_{(k)}^2)\big) \|_{L_t^\gamma L_x^\rho}\notag\\
        &\lesssim
        \sigma^{-1} \sum_{k \in \Lambda} \|h_{(k)}\|_{C_{t}}\|a_{(k)}^2\|_{C_{t,x}^{N+1}}
        \lesssim
         (1+\|\Ru_{q}\|_{C_{t}L_{x}^{1}}^{N+5})\ell^{-6N-20}\sigma^{-1}.
\end{align}
Thus, we obtain \eqref{3-58} and \eqref{3-59}.

Next, we   prove  the $ C_{t, x}^{N}$-estimates. Applying Lemmas \ref{Lemma spatial building blocks}, \ref{Lemma temporal building blocks} and \ref{Lemma amplitudes 1}, \eqref{2-22}, (\ref{3-1}) and (\ref{3-43}), we  get
	\begin{align}\label{3-68}
		\| w_{q+1}^{(p)}\|_{C_{t, x}^{N}}
		&\lesssim
		\sum_{k \in \Lambda}\|a_{(k)}\|_{C_{t, x}^{N}}	\sum_{N_{1}+N_{2}+N_{3}\leq N}\|\p_{t}^{N_{1}}g_{(k)}\|_{L_t^{\infty}}   \|\nabla^{N_{2}}\p_{t}^{N_{3}}W_{(k)}\|_{C_{t}L_{x}^{\infty}}\notag\\
		%&\lesssim	(1+\|\Ru_{q}\|_{C_{t}L_{x}^{1}}^{N+2})\ell^{-6N-7}\sum_{N_{1}+N_{2}+N_{3}\leq N}\sigma^{N_{1}}\tau^{N_{1}+\frac{1}{2}}  r_{\perp}^{-1}r_{\parallel}^{-\frac{1}{2}}\laq^{N_{2}}(\frac{r_{\perp}\lambda_{q+1}\mu}{r_{\parallel}})^{N_{3}}\notag\\
		&\lesssim (1+\|\Ru_{q}\|_{C_{t}L_{x}^{1}}^{N+2})\ell^{-6N-7} r_{\perp}^{-1}r_{\parallel}^{-\frac{1}{2}}\laq^{N}\mu^{N}\tau^{\frac{1}{2}}\notag\\
		&\lesssim
(1+\|\Ru_{q}\|_{C_{t}L_{x}^{1}}^{N+2})\lambda^{2\a(N+1)+1},
\end{align}
where the second inequality was due to \eqref{3-1-2} and $\varepsilon\leq 3/80$.

 By Lemmas \ref{Lemma spatial building blocks}, \ref{Lemma temporal building blocks} and  \ref{Lemma amplitudes 1}, \eqref{2-22}, (\ref{3-1}) and (\ref{3-44}), we obtain
	\begin{align}\label{3-69}
		\| w_{q+1}^{(c)}\|_{C_{t, x}^{N}}
		 &\lesssim
		\sum_{k \in \Lambda}
		\sum_{N_{1}+N_{2}+N_{3}\leq N}\|\p_{t}^{N_{1}}g_{(k)}\|_{L_t^{\infty}}(
		\| a_{(k)}\|_{C_{t, x}^{N+2}}\|\nabla^{N_{2}}\p_{t}^{N_{3}} W_{(k)}^{c}\|_{C_{t}W_{x}^{1,\infty}}\notag\\
		&\hspace{5.5cm}
		+\| a_{(k)}\|_{C_{t, x}^{N}}\|\nabla^{N_{2}}\p_{t}^{N_{3}}\widetilde{W}_{(k)}^{c}\|_{C_{t}L_{x}^{\infty}})\notag\\
		%&\lesssim\sum_{N_{1}+N_{2}+N_{3}\leq N}\sigma^{N_{1}}\tau^{N_{1}+\frac{1}{2}}\big( (1+\|\Ru_{q}\|_{C_{t}L_{x}^{1}}^{N+4})\ell^{-6(N+2)-7}r_{\perp}^{-1}r_{\parallel}^{-\frac{1}{2}}\laq^{N_{2}+1}(\frac{r_{\perp}\lambda_{q+1}\mu}{r_{\parallel}})^{N_{3}}\lambda^{-2}\notag\\
  %&\hspace{4.5cm} +(1+\|\Ru_{q}\|_{C_{t}L_{x}^{1}}^{N+2})\ell^{-6N-7}r_{\perp}^{-1}r_{\parallel}^{-\frac{1}{2}}\laq^{N_{2}}(\frac{r_{\perp}\lambda_{q+1}\mu}{r_{\parallel}})^{N_{3}}r_{\perp}r_{\parallel}^{-1}\big)\notag\\
		&\lesssim (1+\|\Ru_{q}\|_{C_{t}L_{x}^{1}}^{N+4})\ell^{-6N-7} \tau^{\frac{1}{2}}r_{\perp}^{-1}r_{\parallel}^{-\frac{1}{2}}(\frac{r_{\perp}\lambda_{q+1}\mu}{r_{\parallel}})^{N}	\notag \\
		&\lesssim
		(1+\|\Ru_{q}\|_{C_{t}L_{x}^{1}}^{N+4})\lambda_{q+1}^{2\a(N+1)+1}.
\end{align}
The second inequality was due to \eqref{3-1-2} and  $\varepsilon\leq 3/80$.

Using the Sobolev embedding $ W^{1,4}(\T^{3}) \hookrightarrow L^{\infty}(\T^{3})$, \eqref{2-22} and the fact that $\mathbb{P}_H \mathbb{P}_{\neq 0}$ is bounded in $ W^{1,4}(\T^{3})$, we  obtain
\begin{align}\label{3-70}
	\quad\|w_{q+1}^{(t)}\|_{C_{t, x}^{N}}
	&\lesssim
\mu^{-1} \sum_{k \in \Lambda} \sum_{N_{1}+N_{2}\leq N} \|\mathbb{P}_H\mathbb{P}_{\neq 0}( a_{(k)}^2g_{(k)}^{2}\psi_{(k_{1})}^{2}\phi_{(k)}^{2})\|_{C_{t}^{N_{1}}W_{x}^{N_{2}+1,4}}\notag\\
     %&\lesssim \mu^{-1} \sum_{k \in \Lambda}   \|a_{(k)}^{2}\|_{C_{t, x}^{N+1}} \sum_{N_{1}+N_{2}\leq N}\big(\sum_{N_{11}+N_{12}\leq N_{1}}\|\p_{t}^{N_{11}}g_{(k)}^{2}\|_{L_{t}^{\infty}}\notag\\
 %&\hspace{6cm}\times(\sum_{N_{21}+N_{22}\leq N_{2}+1}\|\nabla^{N_{21}}\p_{t}^{N_{12}}\psi_{(k_{1})}^{2}\|_{C_{t}L_{x}^{\infty}}\|\nabla^{N_{22}}\phi_{(k)}^{2}\|_{L_{x}^{\infty}})\big)\notag\\
	%&\lesssim \mu^{-1} (1+\|\Ru_{q}\|_{C_{t}L_{x}^{1}}^{N+5})\ell^{-6N-20}\notag\\
 %&\quad\times \sum_{N_{1}+N_{2}\leq N}\big(\sum_{N_{11}+N_{12}\leq N_{1}}\sigma^{N_{11}}\tau^{N_{11}+\frac{1}{2}}(\sum_{N_{21}+N_{22}\leq N_{2}+1}r_{\parallel}^{-1}(\frac{r_{\perp}\lambda_{q+1}}{r_{\parallel}})^{N_{21}}(\frac{r_{\perp}\lambda_{q+1}\mu}{r_{\parallel}})^{N_{12}}r_{\perp}^{-2}\lambda_{q+1}^{N_{22}})\big )\notag\\
	&\lesssim
	\mu^{-1} (1+\|\Ru_{q}\|_{C_{t}L_{x}^{1}}^{N+5})\ell^{-6N-20}\tau r_{\perp}^{-2}r_{\parallel}^{-1}(\frac{r_{\perp}\lambda_{q+1}\mu}{r_{\parallel}})^{N}\notag\\
	&\lesssim	(1+\|\Ru_{q}\|_{C_{t}L_{x}^{1}}^{N+5})\lambda_{q+1}^{2\a(N+1)+1}.
\end{align}

Similarly, we have
	\begin{align}\label{3-71}
		\|w_{q+1}^{(o)}\|_{C_{t, x}^{N}}
		%&\lesssim \sigma^{-1} \sum_{k \in \Lambda}  \sum_{N_{1}+N_{2}\leq N}\|\mathbb{P}_H\mathbb{P}_{\neq 0}(h_{(k)} \nabla (a_{(k)}^2)) \|_{C_{t}^{N_{1}}W_{x}^{N_{2}+1,4}}\notag\\
		&\lesssim
		\sigma^{-1} \sum_{k \in \Lambda} \sum_{N_{1}+N_{2}\leq N} \|h_{(k)}\|_{C_{t}^{N_{1}}}\|\nabla (a_{(k)}^2)\|_{C_{t,x}^{N_{2}+1}}
		\lesssim
			(1+\|\Ru_{q}\|_{C_{t}L_{x}^{1}}^{N+6})\lambda_{q+1}^{4\a N-5N+3}.
\end{align}
 Therefore, the proof of Lemma \ref{Lemma Estimates of perturbations} is complete.\hfill$\square$

\subsection{Verification of inductive estimates for velocity fields} 
We are now ready to verify the inductive estimates  (\ref {2-14}), (\ref {2-17})-(\ref {2-19}) at level $q+1$ of the new velocity $v_{q+1}$.

For $0\leq N\leq 4$, by  \eqref{3-50} and \eqref{3-54}   we obtain
	\begin{align}\label{3-73.0}
	\|w_{q+1}\|_{C_{t,x}^{N}}	
	&\lesssim
		\|\Theta_{q+1}\|_{C_{t}^{N}}\|w_{q+1}^{(p)+(c)}\|_{C_{t,x}^{N}}+	\|\Theta_{q+1}\|^{2}_{C_{t}^{N}}\|w_{q+1}^{(t)+(o)}\|_{C_{t,x}^{N}},
\end{align}
which along with \eqref{3-49} and \eqref{3-60}-\eqref{3-63} yields
	\begin{align}\label{3-73}
	\|w_{q+1}\|_{C_{t,x}^{N}}	
		%&\lesssim	(1+\|\Ru_{q}\|_{C_{t}L_{x}^{1}}^{N+2})\lambda_{q+1}^{2\a(N+1)+1}\varsigma_{q}^{-N}+(1+\|\Ru_{q}\|_{C_{t}L_{x}^{1}}^{N+4})\lambda_{q+1}^{2\a(N+1)+1}\varsigma_{q}^{-N}\notag\\
 % &\quad+(1+\|\Ru_{q}\|_{C_{t}L_{x}^{1}}^{N+5})\lambda_{q+1}^{2\a N+\frac{7}{2}}\varsigma_{q}^{-N}	+(1+\|\Ru_{q}\|_{C_{t}L_{x}^{1}}^{N+6})\lambda_{q+1}^{4\a N-5N+3}\varsigma_{q}^{-N}\notag\\
		&\lesssim
		(1+\|\Ru_{q}\|_{C_{t}L_{x}^{1}}^{N+6})\lambda_{q+1}^{2\a(N+1)+3}.
\end{align}
Thus,  by (\ref{2-14}) and (\ref{3-73}), for $0\leq N\leq 4$,  we have
	\begin{align}\label{3-74}
	\|v_{q+1}\|_{\Lo^{m}_{\om}C_{t,x}^{N}}
		&\lesssim	\|v_{\ell}\|_{\Lo^{m}_{\om}C_{t,x}^{N}}+\|w_{q+1}\|_{\Lo^{m}_{\om}C_{t,x}^{N}}	\notag\\
	%&\lesssim \|v_{q}\|_{\Lo^{m}_{\om}C_{t,x}^{N}}+(1+\|\Ru_{q}\|_{\Lo^{(N+6)m}_{\om}C_{t}L_{x}^{1}}^{N+6})\lambda_{q+1}^{2\a(N+1)+3}	\notag\\
	%&\lesssim	\la^{{2\a\(N+1\)+5}}\( 8(N+10)mL^{2}50^{q-1}\)^{\(N+10\)50^{q-1}}\notag\\
	%&\quad+\(\lambda_{q}^{(4\a+12)}\( 8(N+6)mL^{2}50^{q}\)^{50^{q}}\)^{N+6}\lambda_{q+1}^{2\a(N+1)+3}\notag\\
		&\leq
		\(8(N+10)mL^{2}50^{q}\)^{\(N+10\)50^{q}}\laq^{2\a(N+1)+5},
\end{align}
where in the last step we used \eqref{2-7}. Thus, estimate (\ref{2-14}) is verified at level $q+1$.

Next, we shall verify the $\Lo^{2r}_{\om}L_{t}^{2}L_{x}^{2}$ decay estimate (\ref{2-17}). For this purpose, we apply the following $L^{p}-$decorrelation result.
\begin{lemma}[\citep{CL-transport equation}; see also \citep{ZZL-MHD-sharp}]\label{Lemma Decorrelation1}
 Let $\theta \in \mathbb{N}$ and $f, g: \mathbb{T}^d \rightarrow \mathbb{R}$ be smooth functions. Then for every $p \in[1,+\infty]$,
\begin{equation}\label{6.7}
	\begin{aligned}
\left|\|f g(\theta \cdot)\|_{L^p\left(\mathbb{T}^d\right)}-\|f\|_{L^p\left(\mathbb{T}^d\right)}\|g\|_{L^p\left(\mathbb{T}^d\right)}\right| \lesssim \theta^{-\frac{1}{p}}\|f\|_{C^1\left(\mathbb{T}^d\right)}\|g\|_{L^p\left(\mathbb{T}^d\right)}.
\end{aligned}
\end{equation}
\end{lemma}
Using  Lemma \ref{Lemma Decorrelation1} we infer that
	\begin{align}\label{3-75}
		\|w_{q+1}^{(p)}\|_{L_{t}^{2}L_{x}^{2}}
		&\lesssim
		\sum_{k \in \Lambda }
		(\|a_{(k)}\|_{L_{t}^2L_{x}^2}\|g_{(k)}\|_{L_t^2}\|\psi_{(k_{1})}\phi_{(k)}\|_{C_{t}L_x^2}+\sigma^{-\frac{1}{2}}\|a_{(k)}\|_{C_{t, x}^1}\|g_{(k)}\|_{L_t^2}\|\psi_{(k_{1})}\phi_{(k)}\|_{C_{t}L_x^2}),
\end{align}
which along with  Lemmas \ref{Lemma spatial building blocks}, \ref{Lemma temporal building blocks}, \ref{Lemma amplitudes 1}  yields that
	\begin{align}\label{3-76}
	\|w_{q+1}^{(p)}\|_{L_{t}^{2}L_{x}^{2}}
	\lesssim
	\dqq^{\frac{1}{2}}+\|\Ru_{q}\|_{C_{t}L_{x}^{1}}^{\frac{1}{2}}+\sigma^{-\frac{1}{2}}(1+\|\Ru_{q}\|_{C_{t}L_{x}^{1}}^{3})\ell^{-13}.
\end{align}
Taking  into account  (\ref{2-7}), (\ref{3-49}), (\ref{3-50}), (\ref{3-57})-(\ref{3-59}) we have
	\begin{align}\label{3-77}
		\|w_{q+1}\|_{L_{t}^{2}L_{x}^{2}}
		&\lesssim
		\|\Theta_{q+1}\|_{C_{t}}\|w_{q+1}^{(p)+(c)}\|_{L_{t}^{2}L_{x}^{2}}+	\|\Theta_{q+1}\|^{2}_{C_{t}}\|w_{q+1}^{(t)+(o)}\|_{L_{t}^{2}L_{x}^{2}}\notag\\
		%&\lesssim \dqq^{\frac{1}{2}}+\|\Ru_{q}\|_{C_{t}L_{x}^{1}}^{\frac{1}{2}}+\sigma^{-\frac{1}{2}}(1+\|\Ru_{q}\|_{C_{t}L_{x}^{1}}^{3})\ell^{-13}+(1+\|\Ru_{q}\|_{C_{t}L_{x}^{1}}^{4})\ell^{-7}r_{\perp}r_{\parallel}^{-1}\notag\\
	%&\quad+(1+\|\Ru_{q}\|_{C_{t}L_{x}^{1}}^{4})\ell^{-14}\mu^{-1}r_{\perp}^{-1}r_{\parallel}^{-\frac{1}{2}}\tau^{\frac{1}{2}}+(1+\|\Ru_{q}\|_{C_{t}L_{x}^{1}}^{5})\ell^{-20}\sigma^{-1}	\notag\\
		&\lesssim
		\dqq^{\frac{1}{2}}+\|\Ru_{q}\|_{C_{t}L_{x}^{1}}^{\frac{1}{2}}+(1+\|\Ru_{q}\|_{C_{t}L_{x}^{1}}^{5})\ell^{-14}\mu^{-1}r_{\perp}^{-1}r_{\parallel}^{-\frac{1}{2}}\tau^{\frac{1}{2}}.
\end{align}
Then, by   (\ref{2-7}), (\ref{2-15}), (\ref{2-16}), (\ref{2-22}) and  (\ref{3-49}), 
\begin{align}\label{3-78}
\|w_{q+1}\|_{\Lo^{2r}_{\om}L_{t}^{2}L_{x}^{2}}
		&\lesssim
		\dqq^{\frac{1}{2}}+\|\Ru_{q}\|_{\Lo^{r}_{\om}C_{t}L_{x}^{1}}^{\frac{1}{2}}+(1+\|\Ru_{q}\|_{\Lo^{10r}_{\om}C_{t}L_{x}^{1}}^{5})\ell^{-14}\mu^{-1}r_{\perp}^{-1}r_{\parallel}^{-\frac{1}{2}}\tau^{\frac{1}{2}} 
		%&\lesssim \dqq^{\frac{1}{2}}+\big(\lambda_{q}^{(4\a+12)}(8\cdot10rL^{2}50^{q})^{50^{q}}\big)^{5}\ell^{-14}\laq^{-\frac{1}{2}\varepsilon}\notag\\
		%&\lesssim \dqq^{\frac{1}{2}}+\lambda_{q+1}^{-\frac{3}{8}\varepsilon}
		\lesssim
		\dqq^{\frac{1}{2}}.
\end{align}
Thus, by (\ref{2-14}) and  the standard mollification estimates, we have
	\begin{align}\label{3-79}
	\|v_{q+1}-v_{q}\|_{\Lo^{2r}_{\om}L_{t}^{2}L_{x}^{2}}
	&\lesssim \|v_\ell-v_{q}\|_{\Lo^{2r}_{\om}L_{t}^{2}L_{x}^{2}}+	\|w_{q+1}\|_{\Lo^{2r}_{\om}L_{t}^{2}L_{x}^{2}}\notag\\
	& \lesssim \ell\|v_{q}\|_{\Lo^{2r}_{\om}C_{t,x}^{1}} +\delta_{q+1}^{\frac{1}{2}} 
    %&\lesssim  \ell\lambda^{4\a+5}(8\cdot22rL^{2}50^{q-1})^{50^{q-1}\cdot11}+\delta_{q+1}^{\frac{1}{2}}
    \lesssim\delta_{q+1}^{\frac{1}{2}},
\end{align}
which yields (\ref{2-17}) at level $q+1$.

Regarding  the $\Lo^{r}_{\om}L_{t}^{1}L_{x}^{2}$-estimate (\ref{2-18}) at $q+1$,  by  (\ref{3-49}),  (\ref{3-50}), and (\ref{3-56})-(\ref{3-59}) in Lemma \ref{Lemma Estimates of perturbations},  we get
	\begin{align}\label{3-80}
		\|w_{q+1}\|_{L_{t}^{1}L_{x}^{2}}
		&\lesssim \|\Theta_{q+1}\|_{C_{t}}\|w_{q+1}^{(p)+(c)}\|_{L_{t}^{1}L_{x}^{2}}+	\|\Theta_{q+1}\|^{2}_{C_{t}}\|w_{q+1}^{(t)+(o)}\|_{L_{t}^{1}L_{x}^{2}}\notag\\
	%&\lesssim (1+\|\Ru_{q}\|_{C_{t}L_{x}^{1}}^{2})\ell^{-7}\tau^{-\frac{1}{2}}+(1+\|\Ru_{q}\|_{C_{t}L_{x}^{1}}^{4})\ell^{-7}r_{\perp}r_{\parallel}^{-1}\tau^{-\frac{1}{2}}\notag\\
	%&\quad+(1+\|\Ru_{q}\|_{C_{t}L_{x}^{1}}^{4})\ell^{-14}\mu^{-1}r_{\perp}^{-1}r_{\parallel}^{-\frac{1}{2}}\tau^{-\frac{1}{2}}+(1+\|\Ru_{q}\|_{C_{t}L_{x}^{1}}^{5})\ell^{-20}\sigma^{-1}\notag\\
	&\lesssim
	(1+\|\Ru_{q}\|_{C_{t}L_{x}^{1}}^{5})\ell^{-20}\sigma^{-1},
\end{align}
which along with (\ref{2-7}), (\ref{2-15}), (\ref{2-16})  and (\ref{2-22})  implies
	\begin{align}\label{3-81}
	\|v_{q+1}-v_q\|_{\Lo^{r}_{\om}L_{t}^{1}L_{x}^{2}}
		& \lesssim
		\|v_\ell-v_q\|_{\Lo^{r}_{\om}L_{t}^{1}L_{x}^{2}}+\|w_{q+1}\|_{\Lo^{r}_{\om}L_{t}^{1}L_{x}^{2}}\notag\\
		&\lesssim	\ell\|v_q\|_{\Lo^{r}_{\om}C_{t,x}^{1}}+	(1+\|\Ru_{q}\|_{\Lo^{5r}_{\om}C_{t}L_{x}^{1}}^{5})\ell^{-20}\sigma^{-1} 
		%& \lesssim \ell\la^{4\a+5}(8\cdot 11rL^{2}50^{q-1})^{11\cdot50^{q-1}} +\big(\lambda_{q}^{(4\a+12)}(8\cdot5rL^{2}50^{q})^{50^{q}}\big)^{5}\ell^{-20}\laq^{-2\varepsilon}\notag\\
		\leq
		\dqqq^{\frac{1}{2}}.
\end{align}
 This verifies estimate (\ref{2-18}) at level $q+1$.

At last, for  the inductive estimate (\ref{2-19}), by (\ref{3-49}), (\ref{3-50}) and Lemma \ref{Lemma Estimates of perturbations}, we get
	\begin{align}\label{3-82}
	\|w_{q+1}\|_{L_{t}^{\ga}W_{x}^{s,p}}	
	&\lesssim
	\|\Theta_{q+1}\|_{C_{t}}\|w_{q+1}^{(p)+(c)}\|_{L_{t}^{\ga}W_{x}^{s,p}}	+	\|\Theta_{q+1}\|^{2}_{C_{t}}\|w_{q+1}^{(t)+(o)}\|_{L_{t}^{\ga}W_{x}^{s,p}}	\notag\\
	%&\lesssim (1+\|\Ru_{q}\|_{C_{t}L_{x}^{1}}^{s+2})\ell^{-7}r_{\perp}^{\frac{2}{p}-1}r_{\parallel}^{\frac{1}{p}-\frac{1}{2}}\laq^{s}\tau^{\frac{1}{2}-\frac{1}{\ga}}+(1+\|\Ru_{q}\|_{C_{t}L_{x}^{1}}^{s+4})\ell^{-7}r_{\perp}^{\frac{2}{p}}r_{\parallel}^{\frac{1}{p}-\frac{3}{2}}\laq^{s}\tau^{\frac{1}{2}-\frac{1}{\ga}}\notag\\
	%&\quad+(1+\|\Ru_{q}\|_{C_{t}L_{x}^{1}}^{s+4})\ell^{-14}\mu^{-1}r_{\perp}^{\frac{2}{p}-2}r_{\parallel}^{\frac{1}{p}-1}\laq^{s}\tau^{1-\frac{1}{\ga}}+(1+\|\Ru_{q}\|_{C_{t}L_{x}^{1}}^{s+6})\ell^{-6s-20}\sigma^{-1}\notag\\
&\lesssim
	(1+\|\Ru_{q}\|_{C_{t}L_{x}^{1}}^{s+6})\ell^{-6s-20}\sigma^{-1},
\end{align}
where the last step was due to  (\ref{2-8}) and the fact that
\begin{align}\label{3-83}
	\ell^{-7}r_{\perp}^{\frac{2}{p}-1}r_{\parallel}^{\frac{1}{p}-\frac{1}{2}}\laq^{s}\tau^{\frac{1}{2}-\frac{1}{\ga}}
	&=\ell^{-7}\lambda_{q+1}^{s-\frac{(4\a-5)}{\gamma}+\frac{3}{p}+(\frac{7}{2}-\frac{1}{\gamma}+\frac{8}{p})\varepsilon+2\a-1}
	\leq \lambda_{q+1}^{-8\varepsilon}\leq \ell^{-6s-20}\sigma^{-1}.
\end{align}
Then, using the embedding $C_{x}^{3}(\T^{3}) \hookrightarrow W_{x}^{s,p}(\T^{3})$ for  $(s, \gamma, p) \in \mathcal{S}_{1}$, (\ref{2-7}), (\ref{2-15}), (\ref{2-16})  and (\ref{2-22}), we get
	\begin{align}\label{3-84}
		\|v_{q+1}-v_q\|_{\Lo^{r}_{\om}L_{t}^{\ga}W_{x}^{s,p}}
		& \lesssim
		\|v_\ell-v_q\|_{\Lo^{r}_{\om}L_{t}^{\ga}W_{x}^{s,p}}+\|w_{q+1}\|_{\Lo^{r}_{\om}L_{t}^{\ga}W_{x}^{s,p}} \notag\\
		& \lesssim
		\|v_\ell-v_q\|_{\Lo^{r}_{\om}L_{t}^{\ga}C_{x}^{3}}+ (1+	\|\Ru_{q}\|_{\Lo^{(s+6)r}_{\om}C_{t}L_{x}^{1}}^{s+6})\ell^{-6s-20}\sigma^{-1} \notag\\
		&\lesssim
		\ell\|v_q\|_{\Lo^{r}_{\om}C_{t,x}^{4}}+	 (1+	\|\Ru_{q}\|_{\Lo^{9r}_{\om}C_{t}L_{x}^{1}}^{9})\ell^{-6s-20}\sigma^{-1}
	\lesssim
		\dqqq^{\frac{1}{2}},
\end{align}
 which justifies estimate (\ref{2-19}) at level $q+1$. Therefore, the proof is complete.
 \hfill $\square$

\section{Reynolds stresses in the regime \texorpdfstring{$\mathcal{S}_{1}$}{S1}}\label{Sec-Reynolds}

This section is to devoted to constructing and verifying the inductive estimates (\ref{2-15}) and (\ref{2-16}) at level $q+1$ 
for  the new Reynolds stress $\Ru_{q+1}$  
in the supercritical regime $\mathcal{S}_{1}$ when $\a\in[\frac{5}{4},2)$.

Let us first recall from 
{\citep{LS-Eluerflows}} the inverse-divergence operator $\mathcal{R}$  defined by
	\begin{align}\label{6.5}
&\left(\mathcal{R}v\right)^{k l}:=\partial_k \Delta^{-1} v^l+\partial_l \Delta^{-1} v^k-\frac{1}{2}\left(\delta_{k l}+\partial_k \partial_l \Delta^{-1}\right) \operatorname{div} \Delta^{-1} v, 
\end{align}
where $v$ is mean-free. Note that, the operator $\mathcal{R}$ maps mean-free functions to symmetric and trace-free matrices and satisfies the identity
\begin{flalign*}
\operatorname{div} \mathcal{R}(v)=v .
\end{flalign*}
Moreover, $|\nabla| \mathcal{R}$ is a Calderon-Zygmund operator and is  bounded in the spaces $L^p$, $1<p<+\infty$ 
(see \citep{LS-Eluerflows}).

\subsection{Reynolds stress}
%Subtracting system (\ref{2-26})  from system (\ref{2-12}) at level $q+1$, we derive
We substract from system (\ref{2-26})  from  (\ref{2-12}) at level $q+1$ to derive
\begin{align}\label{4-1}
		\div\Ru_{q+1}-\nabla P_{q+1}
		&=\underbrace{\p_{t}\omw^{(p)+(c)}_{q+1}+\nu(-\Delta)^{\a}w_{q+1}-(z_{q+1}-z_{\ell})+\div\big((v_{\ell}+z_{\ell})\otimes w_{q+1}+w_{q+1}\otimes ( v_{\ell}+z_{\ell})\big)}_{\div \Ru_{lin}+\nabla P_{lin}}\notag\\  %lin
		&\quad+
		\underbrace{\div(\omw^{(p)}_{q+1}\otimes                       \omw^{(c)+(t)+(o)}_{q+1}
		+\omw^{(c)+(t)+(o)}_{q+1}\otimes w_{q+1})}_{\div\Ru_{corr}+\nabla P_{corr}}+\div(\Ru_{com1})\notag\\ %corr
		&\quad
		+\underbrace{\div( v_{q+1}\otimes z_{q+1}+z_{q+1}\otimes v_{q+1}
		-v_{q+1}\otimes z_{\ell}-z_{\ell}\otimes v_{q+1}+z_{q+1}\otimes z_{q+1}
		-z_{\ell}\otimes z_{\ell})}_{\div \Ru_{com2}+\nabla P_{com2}}\notag\\
		&\quad+\underbrace{\div(\omw^{(p)}_{q+1}\otimes\omw^{(p)}_{q+1}+\Ru_{\ell})+\p_{t}\omw^{(t)}_{q+1}+\p_{t}\omw^{(o)}_{q+1}}_{\div\Ru_{osc}+\nabla P_{osc}}-\nabla P_{\ell},
\end{align}
where $\Ru_{com1}$ is given by \eqref{2-27}. 

Using the inverse divergence operator  $\mathcal{R}$  we define
	\begin{align}
	&\Ru_{lin}
	:=\mathcal{R}\p_{t}\omw_{q+1}^{(p)+(c)}+\mathcal{R}\(\nu(-\Delta)^{\a}w_{q+1}\)+\mathcal{R}(z_{\ell}-z_{q+1})+(v_{\ell}+z_{\ell})\mathring{\otimes}w_{q+1}+w_{q+1}\mathring{\otimes}( v_{\ell}+z_{\ell}),\label{4-2}\\
	&\Ru_{corr}:=
	\omw_{q+1}^{(p)}\mathring{\otimes} \omw_{q+1}^{(c)+(t)+(o)}	+\omw_{q+1}^{(c)+(t)+(o)}\mathring{\otimes}w_{q+1},\label{4-3}\\
	&\Ru_{com2}
	:= v_{q+1}\mathring{\otimes}z_{q+1}+z_{q+1}\mathring{\otimes} v_{q+1}-v_{q+1}\mathring{\otimes}z_{\ell}
	-z_{\ell}\mathring{\otimes} v_{q+1}
	+z_{q+1}\mathring{\otimes} z_{q+1} -z_{\ell}\mathring{\otimes}z_{\ell}.\label{4-4}
\end{align}
For the remaining oscillation error,  using \eqref{3-43-1} we have
\begin{align}\label{4-5}
	\div(\omw^{(p)}_{q+1}\otimes\omw^{(p)}_{q+1}+\Ru_{\ell})
	&=(1-\Theta_{q+1}^{2})\div \Ru_{\ell}+\Theta_{q+1}^{2}\div\big(\sum_{k\in\Lambda} a_{(k)}^{2}(g_{(k)}^{2}-1) \fint_{\T^3}W_{(k)}\otimes W_{(k)} \mathrm{d}x\big)\notag\\
	&\quad+\Theta_{q+1}^{2}\div(\varrho \mathrm{Id})+\Theta_{q+1}^{2}\div\big(\sum_{k\in\Lambda}a_{(k)}^{2}g_{(k)}^{2}\mathbb{P}_{\neq 0}(W_{(k)}\otimes W_{(k)})\big).
\end{align}
By (\ref{3-47-2}), (\ref{3-48-2}) and (\ref{4-5}), we have
\begin{align}\label{4-6}
	&\quad\div(\omw^{(p)}_{q+1}\otimes\omw^{(p)}_{q+1}+\Ru_{\ell})+\p_{t}\omw^{(t)}_{q+1}+\p_{t}\omw^{(o)}_{q+1}\notag\\
	&=(1-\Theta_{q+1}^{2})\div \Ru_{\ell}+\p_{t}\Theta_{q+1}^{2}w^{(t)+(o)}_{q+1}+\Theta_{q+1}^{2}\div(\varrho \mathrm{Id})+\Theta_{q+1}^{2}\sum_{k \in \Lambda}\mathbb{P}_{\neq 0}\big(g_{(k)}^{2}\mathbb{P}_{\neq 0}(W_{(k)}\otimes W_{(k)})\nabla (a_{(k)}^{2})\big)	\notag\\
	&\quad-\Theta_{q+1}^{2}\mu^{-1} \sum_{k \in \Lambda} \mathbb{P}_{\neq 0}(\p_{t}(a_{(k)}^{2}g_{(k)}^{2})\psi_{(k_{1})}^{2}\phi_{(k)}^{2})
	-\Theta_{q+1}^{2}\sigma^{-1} \sum_{k \in \Lambda} \mathbb{P}_{\neq 0}\big(h_{(k)}\fint_{\T^3}W_{(k)}\otimes W_{(k)}\mathrm{d}x\p_{t}\nabla (a_{(k)}^{2})\big)\notag\\
		&\quad+(\nabla\Delta^{-1}\div)\mu^{-1}\mathbb{P}_{\neq 0}\p_{t}(a_{(k)}^{2}g_{(k)}^{2}\psi_{(k_{1})}^{2}\phi_{(k)}^{2})+(\nabla\Delta^{-1}\div)\sigma^{-1}\mathbb{P}_{\neq 0}\p_{t}\big(h_{(k)} \fint_{\mathbb{T}^3} W_{(k)} \otimes W_{(k)} \mathrm{d} x \nabla(a_{(k)}^2)\big).
\end{align}

Thus, we define the oscillation error  by
\begin{align}
   \Ru_{osc}:=\Ru_{osc.1}+\Ru_{osc.2}+\Ru_{osc.3}+\Ru_{osc.4},   \label{4-7}
\end{align}
where
	\begin{align}
		&\Ru_{osc.1}
		:=(1-\Theta_{q+1}^{2}) \Ru_{\ell}+\mathcal{R}(\p_{t}\Theta_{q+1}^{2}w^{(t)+(o)}_{q+1}),\label{4-8}\\
		&\Ru_{osc.2}:=\Theta_{q+1}^{2}\sum_{k \in \Lambda}\mathcal{R}\mathbb{P}_{\neq 0}\big(g_{(k)}^{2}\mathbb{P}_{\neq 0}(W_{(k)}\otimes W_{(k)})\nabla (a_{(k)}^{2})\big),\label{4-9}\\
		&\Ru_{osc.3}
		:=-\Theta_{q+1}^{2}\mu^{-1} \sum_{k \in \Lambda} \mathcal{R}\mathbb{P}_{\neq 0}\big(\p_{t}(a_{(k)}^{2}g_{(k)}^{2})\psi_{(k_{1})}^{2}\phi_{(k)}^{2}k_{1}\big),\label{4-10}\\
	&\Ru_{osc.4}
	:=-\Theta_{q+1}^{2}\sigma^{-1} \sum_{k \in \Lambda} \mathcal{R}\mathbb{P}_{\neq 0}\big(h_{(k)}\fint_{\T^3}W_{(k)}\otimes W_{(k)}\mathrm{d}x\p_{t}\nabla (a_{(k)}^{2})\big).    \label{4-11}
\end{align}

Therefore,  the new Reynolds stress at level $ q + 1 $ is defined by
$$\Ru_{q+1}:=\Ru_{lin}+\Ru_{corr}+\Ru_{com1}+\Ru_{com2}+\Ru_{osc}.$$

\subsection{Verification of growth estimate}
We shall verify the growth estimate  (\ref{2-15}) for the new Reynolds stress at level $q+1$.

\medskip 
\paragraph{\bf Linear errors.}
First, by Lemma \ref{Lemma Estimates of perturbations}, (\ref{3-49})  and  (\ref{3-50}),
	\begin{align}\label{4-12}
			\|\mathcal{R}\p_{t}\omw_{q+1}^{(p)+(c)}\|_{C_{t}L_x^{1}}
		&\lesssim\| \Theta_{q+1}\|_{C_{t}^{1}} \|w_{q+1}^{(p)+(c)} \|_{C_{t,x}^{1}}
		%&\lesssim\varsigma_{q}^{-1}\big((1+\|\Ru_{q}\|_{C_{t}L_{x}^{1}}^{3})\lambda_{q+1}^{4\a+1}+(1+\|\Ru_{q}\|_{C_{t}L_{x}^{1}}^{5})\lambda_{q+1}^{4\a+1}\big)\notag\\
		\lesssim
		(1+\|\Ru_{q}\|_{C_{t}L_{x}^{1}}^{5})\lambda_{q+1}^{4\a+1}\varsigma_{q}^{-1}.
\end{align}
By  (\ref{3-49}), (\ref{3-54}), Lemma \ref{Lemma Estimates of perturbations}
and the Sobolev embedding $H^{3}_{x}(\T^{3})\hookrightarrow H^{2\a-1}_{x}(\T^{3})$
with $\a\in[\frac{5}{4},2)$, we obtain
	\begin{align}\label{4-13}
	\|\mathcal{R}(\nu(-\Delta)^{\a}w_{q+1})\|_{C_{t}L_x^{1}}{\lesssim
\|w_{q+1}\|_{C_{t}H_x^{2\a-1}}}
	&\lesssim
\|w_{q+1}\|_{C_{t}H_x^{3}}
	%&\lesssim (1+\|\Ru_{q}\|_{C_{t}L_{x}^{1}}^{5})\ell^{-7}\laq^{3}\tau^{\frac{1}{2}}+(1+\|\Ru_{q}\|_{C_{t}L_{x}^{1}}^{7})\ell^{-7}r_{\perp}r_{\parallel}^{-1}\laq^{3}\tau^{\frac{1}{2}}\notag\\
 %&\quad+(1+\|\Ru_{q}\|_{C_{t}L_{x}^{1}}^{7})\ell^{-14}\mu^{-1}r_{\perp}^{-1}r_{\parallel}^{-\frac{1}{2}}\laq^{3}\tau+(1+\|\Ru_{q}\|_{C_{t}L_{x}^{1}}^{8})\ell^{-38}\sigma^{-1}\notag\\
 \lesssim
 (1+\|\Ru_{q}\|_{C_{t}L_{x}^{1}}^{8})\ell^{-7}\laq^{3}\tau^{\frac{1}{2}},
\end{align}
where  the first inequality was due to the boundeness of Calderon-Zygmund operators in the space $L_{x}^{2}$.

For the remaining linear error in  \eqref{4-2}, by Lemma \ref{Lemma Estimates of perturbations} and H\"{o}lder's inequality, we obtain
	\begin{align}\label{4-14}
	&\quad\|\mathcal{R}(z_{\ell}-z_{q+1})\|_{C_{t}L_x^{2}}
	+\|(v_{\ell}+z_{\ell})\mathring{\otimes}w_{q+1}+w_{q+1}\mathring{\otimes}( v_{\ell}+z_{\ell})\|_{C_{t}L_x^{2}}\notag\\
	&\lesssim\|z\|_{C_{t}L_x^{2}}+(\|v_{q}\|_{C_{t,x}}+\|z\|_{C_{t} L_{x}^{2}})\|w_{q+1}\|_{C_{t,x}}\notag\\
	%&\lesssim \|z\|_{C_{t}L_x^{2}}+(\|v_{q}\|_{C_{t,x} }+\|z\|_{C_{t} L_{x}^{2}})(1+\|\Ru_{q}\|_{C_{t}L_{x}^{1}}^{6})(\lambda_{q+1}^{2\a+1}+\lambda_{q+1}^{\frac{7}{2}}+\lambda_{q+1}^{3})\notag\\
	&\lesssim
   (\|v_{q}\|_{C_{t,x} }+\|z\|_{C_{t} L_{x}^{2}})(1+\|\Ru_{q}\|_{C_{t}L_{x}^{1}}^{6})\lambda_{q+1}^{2\a+1}.
\end{align}
Thus, we conclude from the above estimates \eqref{4-12}-\eqref{4-14} that
	\begin{align}\label{4-15}
	\|\Ru_{lin}\|_{C_{t}L_x^{1}}
	%&\lesssim(1+\|\Ru_{q}\|_{C_{t}L_{x}^{1}}^{5})\lambda_{q+1}^{4\a+1}\varsigma_{q}^{-1} +(1+\|\Ru_{q}\|_{C_{t}L_{x}^{1}}^{8})\ell^{-7}\laq^{3}\tau^{\frac{1}{2}}\notag\\
	%&\quad	+\|z\|_{C_{t} L_{x}^{2}}+(\|v_{q}\|_{C_{t,x} }+\|z\|_{C_{t} L_{x}^{2}})(1+\|\Ru_{q}\|_{C_{t}L_{x}^{1}}^{6})\lambda_{q+1}^{2\a+1}\notag\\
	&\lesssim
	(1+\|\Ru_{q}\|_{C_{t}L_{x}^{1}}^{8})\lambda_{q+1}^{4\a+2}+(\|v_{q}\|_{C_{t,x} }+\|z\|_{C_{t} L_{x}^{2}})(1+\|\Ru_{q}\|_{C_{t}L_{x}^{1}}^{6})\lambda_{q+1}^{2\a+1}. 
\end{align}

\medskip 
\paragraph{\bf Oscillation error.} 
By (\ref{3-49}),  (\ref{4-8}) and Lemma \ref{Lemma Estimates of perturbations}, we  get
\begin{align}\label{4-16}
	\| \Ru_{osc.1}\|_{C_{t}L_{x}^{1}}
&\lesssim
\|1-\Theta_{q+1}^{2}\|_{C_{t}}\|\Ru_{q}\|_{C_{t}L_{x}^{1}}+\|\Theta_{q+1}^{2}\|_{C_{t}^{1}}\|w^{(t)}_{q+1}+w^{(o)}_{q+1}\|_{C_{t,x}}
\notag\\
%&\lesssim\|\Ru_{q}\|_{C_{t}L_{x}^{1}}+\varsigma_{q}^{-1}(1+\|\Ru_{q}\|_{C_{t}L_{x}^{1}}^{5})\lambda_{q+1}^{\frac{7}{2}}+\varsigma_{q}^{-1}(1+\|\Ru_{q}\|_{C_{t}L_{x}^{1}}^{6})\lambda_{q+1}^{3}\notag\\
 &\lesssim
  (1+\|\Ru_{q}\|_{C_{t}L_{x}^{1}}^{6})\lambda_{q+1}^{2\a+1}\varsigma_{q}^{-1}.
\end{align}
Moreover, by (\ref{3-49}), (\ref{4-9}) and Lemmas \ref{Lemma spatial building blocks}, \ref{Lemma temporal building blocks}, \ref{Lemma amplitudes 1},
\begin{align}\label{4-17}
	\| \Ru_{osc.2}\|_{C_{t}L_{x}^{1}}
	&\lesssim
	\|\Theta_{q+1}^{2}\|_{C_{t}}
    \|\mathcal{R}\mathbb{P}_{\neq 0}(g_{(k)}^{2}\mathbb{P}_{\neq 0}(W_{(k)}\otimes W_{(k)})\nabla (a_{(k)}^{2}))\|_{C_{t}L_{x}^{2}}\notag\\
    &\lesssim
	\sum_{k\in\Lambda}\|a_{(k)}^{2}\|_{C_{t,x}^{1}}\|W_{(k)}\otimes W_{(k)}\|_{C_{t}L_{x}^{2}}
	\|g_{(k)}^{2}\|_{L_{t}^{\infty}}\notag\\
	&\lesssim
	 (1+\|\Ru_{q}\|_{C_{t}L_{x}^{1}}^{5})\ell^{-20}r_{\perp}^{-1}r_{\parallel}^{-\frac{1}{2}}\tau.
\end{align}
Similarly, by (\ref{3-49}), (\ref{4-10}) and Lemmas \ref{Lemma spatial building blocks}, \ref{Lemma temporal building blocks}, \ref{Lemma amplitudes 1},
\begin{align}\label{4-18}
	\|\Ru_{osc.3}\|_{C_{t}L_{x}^{1}}
	&\lesssim
	\mu^{-1} \|\Theta_{q+1}^{2}\|_{C_{t}}\sum_{k \in \Lambda}\| \mathcal{R}\mathbb{P}_{\neq 0}(\p_{t}(a_{(k)}^{2}g_{(k)}^{2})\psi_{(k_{1})}^{2}\phi_{(k)}^{2})\|_{C_{t}L_{x}^{2}}\notag\\
	&\lesssim
	\mu^{-1} \sum_{k\in\Lambda}\|a_{(k)}^{2}\|_{C_{t,x}^{1}}\|\psi_{(k_{1})}^{2}\|_{C_{t}L_{x}^{2}}\|\phi_{(k)}^{2}\|_{C_{t}L_{x}^{2}}\|g_{(k)}^{2}\|_{C_{t}^{1}}\notag\\
	&\lesssim
	(1+\|\Ru_{q}\|_{C_{t}L_{x}^{1}}^{5})\ell^{-20}\mu^{-1} r_{\perp}^{-1}r_{\parallel}^{-\frac{1}{2}}\sigma\tau^{2}.
\end{align}
For the last component $\Ru_{osc.4}$, by (\ref{3-49}), (\ref{3-21}), (\ref{4-11})  and Lemma \ref{Lemma amplitudes 1},
\begin{align}  \label{4-19}
	\|\Ru_{osc.4}\|_{C_{t}L_{x}^{1}}
	&\lesssim
		\sigma^{-1}\|\Theta_{q+1}^{2}\|_{C_{t}}\sum_{k\in\Lambda}\|\mathcal{R}\mathbb{P}_{\neq 0}(h_{(k)}\fint_{\T^3}W_{(k)}\otimes W_{(k)}\mathrm{d}x\p_{t}\nabla (a_{(k)}^{2}))\|_{C_{t}L_{x}^{2}}\notag\\
	 &\lesssim
	\sigma^{-1}\sum_{k\in\Lambda}\|h_{(k)}\|_{C_{t}}
	            \|a_{(k)}^{2}\|_{C_{t,x}^{2}}   \notag\\
	 &\lesssim
		(1+\|\Ru_{q}\|_{C_{t}L_{x}^{1}}^{6})\ell^{-26}\sigma^{-1}.
\end{align}
Thus, combining \eqref{4-16}-\eqref{4-19} altogether,  we obtain
\begin{align}  \label{4-20}
		\|\Ru_{osc}\|_{C_{t}L_x^{1}}
		%&\lesssim	(1+\|\Ru_{q}\|_{C_{t}L_{x}^{1}}^{6})\lambda_{q+1}^{\frac{7}{2}}\varsigma_{q}^{-1}+  (1+\|\Ru_{q}\|_{C_{t}L_{x}^{1}}^{5})\ell^{-20}r_{\perp}^{-1}r_{\parallel}^{-\frac{1}{2}}\tau\notag\\
		%&\quad+(1+\|\Ru_{q}\|_{C_{t}L_{x}^{1}}^{5})\ell^{-20}\mu^{-1} r_{\perp}^{-1}r_{\parallel}^{-\frac{1}{2}}\sigma\tau^{2}+	(1+\|\Ru_{q}\|_{C_{t}L_{x}^{1}}^{6})\ell^{-26}\sigma^{-1}\notag\\
		&\lesssim
		(1+\|\Ru_{q}\|_{C_{t}L_{x}^{1}}^{6})(\lambda_{q+1}^{2\a+1}\varsigma_{q}^{-1}+\ell^{-20}\lambda_{q+1}^{4\a-\frac{7}{2}+7\varepsilon}+\ell^{-20}\lambda_{q+1}^{6\a-\frac{15}{2}+18\varepsilon}+\ell^{-26}\lambda_{q+1}^{-2\varepsilon})\notag\\
		&\lesssim
		(1+\|\Ru_{q}\|_{C_{t}L_{x}^{1}}^{6})\lambda_{q+1}^{4\a+1},
\end{align}
where   the last  step was due to (\ref{2-6}), (\ref{2-22}) and (\ref{3-1}).

\medskip 
\paragraph{\bf Corrector error.}
By Lemma \ref{Lemma Estimates of perturbations}, (\ref{3-49}) and (\ref{4-4}),  we have
\begin{align}\label{4-21}
	\|\Ru_{corr}\|_{C_{t}L_{x}^{1}}
	&\lesssim
	\|\omw_{q+1}^{(c)+(t)+(o)} \|_{C_{t,x}}(\|\omw_{q+1}^{(p)}\|_{C_{t,x}}+\|w_{q+1}\|_{C_{t,x}})\notag\\
	%&\lesssim(1+\|\Ru_{q}\|_{C_{t}L_{x}^{1}}^{6})(\lambda_{q+1}^{2\a+1}+\lambda_{q+1}^{\frac{7}{2}}+\lambda_{q+1}^{3})\times (1+\|\Ru_{q}\|_{C_{t}L_{x}^{1}}^{6})(\lambda_{q+1}^{2\a+1}+\lambda_{q+1}^{2\a+1}+\lambda_{q+1}^{\frac{7}{2}}+\lambda_{q+1}^{3})\notag\\
	&\lesssim
	(1+\|\Ru_{q}\|_{C_{t}L_{x}^{1}}^{12})\lambda_{q+1}^{4\a+2}.
\end{align}

\medskip 
\paragraph{\bf Commutator errors.} For the remaining commutator errors, by
 (\ref{2-27}), we have
	\begin{align}\label{4-22}
		\|\Ru_{com1}\|_{C_{t}L_x^{1}}
	&\lesssim \|v_{q}\|_{C_{t,x}}^{2}+\|z\|_{C_{t}L_{x}^{2}}^{2}.
\end{align}
Moreover, by (\ref{4-4}),
\begin{align}\label{4-23}
	\|\Ru_{com2}\|_{C_{t}L_x^{1}}
	&\lesssim
	\|v_{q+1}\|_{C_{t,x}}^{2}+\|z\|_{C_{t} L_{x}^{2}}^{2}.
\end{align}

\medskip 
\paragraph{\bf Verification of inductive estimate  (\ref{2-15}).} 
Now, combining the above estimates (\ref{4-15}) and \eqref{4-20}-\eqref{4-23} altogether, we conclude that
\begin{align}\label{4-24}
\|\Ru_{q+1}\|_{C_{t}L_x^{1}}
&\lesssim
	(1+\|\Ru_{q}\|_{C_{t}L_{x}^{1}}^{8})\lambda_{q+1}^{4\a+2}+(\|v_{q}\|_{C_{t,x} }+\|z\|_{C_{t} L_{x}^{2}})(1+\|\Ru_{q}\|_{C_{t}L_{x}^{1}}^{6})\lambda_{q+1}^{2\a+1}+(1+\|\Ru_{q}\|_{C_{t}L_{x}^{1}}^{6})\lambda_{q+1}^{4\a+1}\notag\\           %lin, osc
	&\quad+(1+\|\Ru_{q}\|_{C_{t}L_{x}^{1}}^{12})\lambda_{q+1}^{4\a+2}
	+\|v_{q}\|_{C_{t,x}}^{2}+\|z\|_{C_{t}L_{x}^{2}}^{2}  % corr com1
	+\|v_{q+1}\|_{C_{t,x}}^{2} \notag\\%com2
	&\lesssim
	(\|v_{q}\|_{C_{t,x} }+\|z\|_{C_{t} L_{x}^{2}})(1+\|\Ru_{q}\|_{C_{t}L_{x}^{1}}^{6})\lambda_{q+1}^{2\a+1}+(1+\|\Ru_{q}\|_{C_{t}L_{x}^{1}}^{12})\lambda_{q+1}^{4\a+2}\notag\\           %lin, osc
	&\quad
	+\|v_{q}\|_{C_{t,x}}^{2}+\|z\|_{C_{t}L_{x}^{2}}^{2}  % corr com1
	+\|v_{q+1}\|_{C_{t,x}}^{2}.
\end{align}
Taking the $m$-th moment in (\ref{4-24}) and using (\ref{stochastic evolution}),  (\ref{2-14}), (\ref{2-15}), (\ref{3-74}) and H\"{o}lder's inequality we obtain
\begin{align*}\label{4-25}
	\|\Ru_{q+1}\|_{\Lo^{m}_{\om}C_{t}L_x^{1}}
		&\lesssim
		(\|v_{q}\|_{\Lo^{2m}_{\om}C_{t,x}}+\|z\|_{\Lo^{2m}_{\om}C_{t}L_{x}^{2}})(1+\|\Ru_{q}\|_{\Lo^{12m}_{\om}C_{t}L_{x}^{1}}^{6})\lambda_{q+1}^{2\a+1}+	(1+\|\Ru_{q}\|_{\Lo^{12m}_{\om}C_{t}L_{x}^{1}}^{12})\lambda_{q+1}^{4\a+2}\notag\\
		&\quad+\|v_{q}\|_{\Lo^{2m}_{\om}C_{t,x}}^{2}
		%linÏî osc corr com1
			+\|z\|_{\Lo^{2m}(\om;C_{t}L_{x}^{2})}^{2}  %com1Ïî
		+\|v_{q+1}\|_{\Lo^{2m}(\om;C_{t,x})}^{2} %com2
		\notag\\
		%&\lesssim\big(\la^{(2\a+5)}(8\cdot 20mL^{2}50^{q-1})^{10\cdot50^{q-1}}+M+(2m-1)^{\frac{1}{2}}L\big) \big(\lambda_{q}^{(4\a+12)}(8\cdot 12mL^{2}50^{q})^{50^{q}}\big)^{6}\lambda_{q+1}^{2\a+1}\notag\\
		%&\quad+\big(\lambda_{q}^{(4\a+12)}(8\cdot 12mL^{2}50^{q})^{50^{q}}\big)^{12}\lambda_{q+1}^{4\a+2}+\big(\la^{(2\a+5)}(8\cdot 20mL^{2}50^{q-1})^{10\cdot50^{q-1}}\big)^{2}\notag\\
		%&\quad+M^{2}+(2m-1)L^{2}+\big(\laq^{(2\a+5)}(8\cdot 20mL^{2}50^{q})^{10\cdot50^{q}}\big)^{2}	\notag\\
		&\leq
		\lambda_{q+1}^{(4\a+12)}(8mL^{2}50^{q+1})^{50^{q+1}},
\end{align*}
which verifies the inductive estimate (\ref{2-15}) at level $q+1$.

\subsection{Verification of decay estimate \eqref{2-16}: away from the initial time}
To verify the decay estimate \eqref{2-16} at level $q+1$, we first consider the difficult regime $(\frac{\varsigma_{q}}{2}, T]$ away from the initial time,
where $\varsigma_{q}$ is given by \eqref{2-6}.

 Choose
\begin{equation}\label{4-26}
	\begin{aligned}
		\rho:=\frac{3-8 \varepsilon}{3-9 \varepsilon} \in(1,2),
\end{aligned}\end{equation}
where $\varepsilon$ is given by (\ref{2-8}). Then,
\begin{equation}\label{4-27}
	\begin{aligned}
		(3-8\varepsilon)(1-\frac{1}{\rho})= \varepsilon,
\end{aligned}\end{equation}
and
\begin{equation}\label{4-28}
	\begin{aligned}
		r_{\perp}^{\frac{2}{\rho}-2}r_{\parallel}^{\frac{1}{\rho}-1}=\lambda_{q+1}^{\varepsilon},\quad 
		r_{\perp}^{\frac{2}{\rho}-1}r_{\parallel}^{\frac{1}{\rho}-\frac{1}{2}}=\lambda_{q+1}^{-\frac{3}{2}+5\varepsilon},\quad 
		r_{\perp}^{\frac{2}{\rho}}r_{\parallel}^{\frac{1}{\rho}-\frac{3}{2}}=\lambda_{q+1}^{-\frac{3}{2}+3\varepsilon}.
\end{aligned}\end{equation}

We consider the linear error, the oscillation error, the corrector error and the commutator errors in the following four subsections, respectively.

\subsubsection{Linear error:}
Let us first consider the linear error $\Ru_{lin}$ in \eqref{4-2}.

\paragraph{\bf Control of $\p_{t}\omw_{q+1}^{(p)+(c)}$.} 
By \eqref{3-50}, we have
\begin{align*}
\|\mathcal{R}\p_{t}\omw_{q+1}^{(p)+(c)}\|_{L_{(\frac{\varsigma_{q}}{2}, T]}^{1}L_x^{\rho}} 
	&\leq
	\|\mathcal{R}(\p_t \Theta_{q+1}) w_{q+1}^{(p)+(c)}\|_{L_{t}^{1}L_x^{\rho}}
	+\|\mathcal{R} \Theta_{q+1} \partial_t w_{q+1}^{(p)+(c)} \|_{L_{t}^{1}L_x^{\rho}} \notag\\
	&=:K_{1}+K_{2}.
\end{align*}
Note that, by ({\ref{3-49}}), (\ref{3-56}) and (\ref{3-57}),
\begin{align*}
		K_{1}
  \lesssim 
  \|\Theta_{q+1}\|_{C_{t}^{1}} \|w_{q+1}^{(p)+(c)}\|_{L_{t}^{1}L_x^{\rho}}
		%&\lesssim\varsigma_{q}^{-1}(1+\|\Ru_{q}\|_{C_{t}L_{x}^{1}}^{2})\ell^{-7}r_{\perp}^{\frac{2}{\rho}-1}r_{\parallel}^{\frac{1}{\rho}-\frac{1}{2}}\tau^{-\frac{1}{2}}+\varsigma_{q}^{-1}(1+\|\Ru_{q}\|_{C_{t}L_{x}^{1}}^{4})\ell^{-7}r_{\perp}^{\frac{2}{\rho}}r_{\parallel}^{\frac{1}{\rho}-\frac{3}{2}}\tau^{-\frac{1}{2}}\notag\\
		\lesssim
(1+\|\Ru_{q}\|_{C_{t}L_{x}^{1}}^{4})\ell^{-7}\varsigma_{q}^{-1}r_{\perp}^{\frac{2}{\rho}-1}r_{\parallel}^{\frac{1}{\rho}-\frac{1}{2}}\tau^{-\frac{1}{2}},
	\end{align*}
where the last was due to $r_{\perp}^{-1}r_{\parallel}<1$.

Moreover, by Lemmas \ref{Lemma spatial building blocks}, \ref{Lemma temporal building blocks}, \ref{Lemma amplitudes 1}, (\ref{3-45}) and ({\ref{3-49}}),
\begin{align*}
    \quad K_{2}
	&\lesssim
	\sum_{k \in\Lambda}
	\|\mathcal{R}\curl\curl\p_{t}(a_{(k)}g_{(k)}W_{(k)}^{c})\|_{L_{t}^{1}L_x^{\rho}}\notag\\
	&\lesssim
	\sum_{k \in \Lambda}\big(
	\|g_{(k)}\|_{L_{t}^{1}}(\|a_{(k)}\|_{C_{t,x}^{2}}\|W_{(k)}^{c}\|_{C_{t}W_{x}^{1,\rho}}+\|a_{(k)}\|_{C_{t,x}^{1}}\|\p_{t}W_{(k)}^{c}\|_{W_{x}^{1,\rho}})+	\|\p_{t}g_{(k)}\|_{L_{t}^{1}}\|a_{(k)}\|_{C_{t,x}^{2}}\|W_{(k)}^{c}\|_{C_{t}W_{x}^{1,\rho}}\big)\notag\\
	%&\lesssim\tau^{-\frac{1}{2}}\big((1+\|\Ru_{q}\|_{C_{t}L_{x}^{1}}^{4})\ell^{-19}r_{\perp}^{\frac{2}{\rho}-1}r_{\parallel}^{\frac{1}{\rho}-\frac{1}{2}}\lambda_{q+1}^{-1} +(1+\|\Ru_{q}\|_{C_{t}L_{x}^{1}}^{2})\ell^{-13}r_{\perp}^{\frac{2}{\rho}-1}r_{\parallel}^{\frac{1}{\rho}-\frac{1}{2}}\lambda_{q+1}^{-1}(\frac{r_{\perp}\lambda_{q+1}\mu}{r_{\parallel}})\big)\notag\\
   %&\quad+\sigma\tau^{\frac{1}{2}}(1+\|\Ru_{q}\|_{C_{t}L_{x}^{1}}^{2})\ell^{-13}r_{\perp}^{\frac{2}{\rho}-1}r_{\parallel}^{\frac{1}{\rho}-\frac{1}{2}}\lambda_{q+1}^{-1}\notag\\
	&\lesssim
	(1+\|\Ru_{q}\|_{C_{t}L_{x}^{1}}^{4})\ell^{-13}r_{\perp}^{\frac{2}{\rho}}r_{\parallel}^{\frac{1}{\rho}-\frac{3}{2}}\mu	\tau^{-\frac{1}{2}},
\end{align*}
where the last was due to \eqref{3-1-2}.

Thus, we obtain
\begin{align}\label{4-28-0}
\|\mathcal{R}\p_{t}\omw_{q+1}^{(p)+(c)}\|_{L_{(\frac{\varsigma_{q}}{2}, T]}^{1}L_x^{\rho}}
	%&\lesssim(1+\|\Ru_{q}\|_{C_{t}L_{x}^{1}}^{4})\ell^{-7}\varsigma_{q}^{-1}r_{\perp}^{\frac{2}{\rho}-1}r_{\parallel}^{\frac{1}{\rho}-\frac{1}{2}}\tau^{-\frac{1}{2}}+	(1+\|\Ru_{q}\|_{C_{t}L_{x}^{1}}^{4})\ell^{-13}r_{\perp}^{\frac{2}{\rho}}r_{\parallel}^{\frac{1}{\rho}-\frac{3}{2}}\lambda_{q+1}^{-1}\mu	\tau^{-\frac{1}{2}}\notag\\
	&\lesssim
	(1+\|\Ru_{q}\|_{C_{t}L_{x}^{1}}^{4})\ell^{-13}r_{\perp}^{\frac{2}{\rho}}r_{\parallel}^{\frac{1}{\rho}-\frac{3}{2}}\mu	\tau^{-\frac{1}{2}}.
\end{align}

\medskip 
\paragraph{\bf  Control of hyper-viscosity. }
For the hyper-viscosity $(-\Delta)^{\a}$,
we use the interpolation and Lemma \ref{Lemma Estimates of perturbations} to obtain 
\begin{align*}
	\|\mathcal{R}\nu(-\Delta)^{\a}\omw_{q+1}^{(p)}\|_{L^{1}_{(\frac{\varsigma_{q}}{2}, T]}L_x^{\rho}}  
	&\lesssim \|\Theta_{q+1}\|_{C_{t}}\|w_{q+1}^{(p)}\|^{\frac{4-2\a}{3}}_{L_{t}^{1}L_x^{\rho}}\|w_{q+1}^{(p)}\|^{\frac{2\a-1}{3}}_{L_{t}^{1}W_x^{3,\rho}}\notag\\
	&\lesssim (1+\|\Ru_{q}\|_{C_{t}L_{x}^{1}}^{2\a+1})\ell^{-7}r_{\perp}^{\frac{2}{\rho}-1}r_{\parallel}^{\frac{1}{\rho}-\frac{1}{2}}\laq^{2\a-1}\tau^{-\frac{1}{2}}.
\end{align*}
Similarly, we have
\begin{align*}
	&\|\mathcal{R}\nu(-\Delta)^{\a}\omw_{q+1}^{(c)}\|_{L^{1}_{(\frac{\varsigma_{q}}{2}, T]}L_x^{\rho}}
	\lesssim (1+\|\Ru_{q}\|_{C_{t}L_{x}^{1}}^{2\a+3})\ell^{-7}r_{\perp}^{\frac{2}{\rho}}r_{\parallel}^{\frac{1}{\rho}-\frac{3}{2}}\laq^{2\a-1}\tau^{-\frac{1}{2}},\\
	&\|\mathcal{R}\nu(-\Delta)^{\a}\omw_{q+1}^{(t)}\|_{L^{1}_{(\frac{\varsigma_{q}}{2}, T]}L_x^{\rho}}
	\lesssim (1+\|\Ru_{q}\|_{C_{t}L_{x}^{1}}^{2\a+3})\ell^{-14}\mu^{-1}r_{\perp}^{\frac{2}{\rho}-2}r_{\parallel}^{\frac{1}{\rho}-1}\laq^{2\a-1},\\
	&\|\mathcal{R}\nu(-\Delta)^{\a}\omw_{q+1}^{(o)}\|_{L^{1}_{(\frac{\varsigma_{q}}{2}, T]}L_x^{\rho}}
	\lesssim
	(1+\|\Ru_{q}\|_{C_{t}L_{x}^{1}}^{2\a+4})\ell^{-6(2\a-1)-20}\sigma^{-1}.
	\end{align*}
Noting that $	\ell^{-7}r_{\perp}^{\frac{2}{\rho}-1}r_{\parallel}^{\frac{1}{\rho}-\frac{1}{2}}\laq^{2\a-1}\tau^{-\frac{1}{2}}=\ell^{-7}\lambda_{q+1}^{-\frac{1}{2}\varepsilon}$ and $\ell^{-14}\mu^{-1}r_{\perp}^{\frac{2}{\rho}-2}r_{\parallel}^{\frac{1}{\rho}-1}\laq^{2\a-1}=\ell^{-14}\lambda_{q+1}^{-\varepsilon}$, we thus get
\begin{align}\label{4-29}
\|\mathcal{R}\nu(-\Delta)^{\a}w_{q+1}\|_{L^{1}_{(\frac{\varsigma_{q}}{2}, T]}L_x^{\rho}}
	&\lesssim (1+\|\Ru_{q}\|_{C_{t}L_{x}^{1}}^{2\a+4})\ell^{-7}r_{\perp}^{\frac{2}{\rho}-1}r_{\parallel}^{\frac{1}{\rho}-\frac{1}{2}}\laq^{2\a-1}\tau^{-\frac{1}{2}}.
\end{align}

\medskip 
\paragraph{\bf  Control of noise term.}
For the noise term,  we use  the standard mollification estimates and (\ref{2-24}) to obtain
 \begin{align}\label{4-30}
 \|\mathcal{R}(z_{\ell}-z_{q+1})\|_{L_{(\frac{\varsigma_{q}}{2}, T]}^{1}L_x^{\rho}} 
 	&\lesssim
 \|z_{\ell}^{u}-z^{u}\|_{L_{(\frac{\varsigma_{q}}{2}, T]}^{1}L_x^{\rho}}
 +\|Z_{\ell}-Z_{q}\|_{L_{t}^{1}L_x^{\rho}} 
 +\|Z_{q}-Z_{q+1}\|_{L_{t}^{1}L_x^{\rho}} \notag\\
 %	&\lesssim
 %\|z_{\ell}^{u}-z^{u}*_{x}\varrho_{\ell}\|_{L_{(\frac{\varsigma_{q}}{2}, T]}^{1}L_x^{\rho}} +\|z^{u}*_{x}\varrho_{\ell}-z^{u}\|_{L_{(\frac{\varsigma_{q}}{2}, T]}^{1}L_x^{\rho}}+\|Z_{\ell}-Z_{q}*_{x}\varrho_{\ell}\|_{L_{t}^{1}L_x^{\rho}}\notag\\ &\quad+\|Z_{q}*_{x}\varrho_{\ell}-Z_{q}\|_{L_{t}^{1}L_x^{\rho}}+\|Z_{q}-Z_{q+1}\|_{L_{t}^{1}L_x^{\rho}} \notag\\
 	%&\lesssim \|z^{u}*_{t}\vartheta_{\ell}-z^{u}\|_{L_{(\frac{\varsigma_{q}}{2}, T]}^{1}L_x^{\rho}}+\|z^{u}*_{x}\varrho_{\ell}-z^{u}\|_{L_{(\frac{\varsigma_{q}}{2}, T]}^{1}L_x^{\rho}} +\|Z*_{t}\vartheta_{\ell}-Z_{q}\|_{L_{t}^{1}L_x^{\rho}}\notag\\ &\quad  +\|Z_{q}*_{x}\varrho_{\ell}-Z_{q}\|_{L_{t}^{1}L_x^{\rho}}+\lambda_{q}^{-15(1-\delta)}\|Z\|_{C_{t} H_{x}^{1-\delta}}\notag\\
 	&\lesssim
 	\ell^{\frac{1}{2}}(\|z^{u}\|_{C^{\frac{1}{2}}_{(\frac{\varsigma_{q}}{2}-\ell, T]} L_{x}^{2}}+\|z^{u}\|_{C_{(\frac{\varsigma_{q}}{2}, T]} H_{x}^{\frac{1}{2}}})+
 	\ell^{\frac{1}{2}-\delta}\|Z\|_{C^{\frac{1}{2}-\delta}_{t}L_{x}^{2}}+\la^{-15(1-\delta)}\|Z\|_{C_{t} H_{x}^{1-\delta}},
 \end{align}
where the last was due to $\|Z_{q}-Z_{q+1}\|_{L_{t}^{1}L_x^{\rho}}\lesssim \la^{-15(1-\delta)}\|Z\|_{C_{t} H_{x}^{1-\delta}}$.

Note that, by 
{\citep[(5.31), (5.34)]{CLZ}},  for $s_{1},~s_{2}\in (0,1)$,
\begin{align}\label{4-31}
\|z^{u}\|_{C_{(\frac{\varsigma_{q}}{2}, T]} H_{x}^{s_{1}}}\lesssim (1+\varsigma_{q}^{-\frac{s_{1}}{2\a}})M,\quad \|z^{u}\|_{C_{(\frac{\varsigma_{q}}{2}-\ell,T]}^{s_{2}}L_{x}^{2}}
	\lesssim (1+\varsigma_{q}^{-s_{2}}) M.
\end{align}
Thus, it follows from (\ref{4-30}) and (\ref{4-31}) that
\begin{align}\label{4-32}
\|\mathcal{R}\(z_{\ell}-z_{q+1}\)\|_{L_{(\frac{\varsigma_{q}}{2}, T]}^{1}L_x^{\rho}} 
&\lesssim
	\ell^{\frac{1}{2}}(1+\varsigma_{q}^{-\frac{1}{2}})
	M+\la^{-15(1-\delta)}(\|Z\|_{C^{\frac{1}{2}-\delta}_{t} L_{x}^{2}}+\|Z\|_{C_{t} H_{x}^{1-\delta}}).
\end{align}

\medskip 
\paragraph{\bf   Control of  remaining terms:}
Regarding the remaining terms in  (\ref{4-2}), estimating as in (\ref{4-29}) we have
\begin{align}\label{4-33}		
     &\quad\|(v_{\ell}+z_{\ell})\mathring{\otimes}w_{q+1}+w_{q+1}\mathring{\otimes}( v_{\ell}+z_{\ell})\|_{L_{(\frac{\varsigma_{q}}{2}, T]}^{1}L_x^{\rho}}\notag\\ 
	&\lesssim
	(\|v_{q}\|_{C_{t,x}}+\|z_{q}\|_{C_{(\frac{\varsigma_{q}}{2}-\ell, T]}L_{x}^{\infty}})\|w_{q+1}\|_{L_{t}^{1}L_x^{\rho}} \notag\\
%&\lesssim	(\|v_{q}\|_{C_{t,x}}+\|z^{u}\|_{C_{(\frac{\varsigma_{q}}{2}-\ell, T]}L_{x}^{\infty}}+\|Z_{q}\|_{C_{t} L_{x}^{\infty}})\notag\\
%&\quad\times	(1+\|\Ru_{q}\|_{C_{t}L_{x}^{1}}^{5})(\ell^{-7}r_{\perp}^{\frac{2}{\rho}-1}r_{\parallel}^{\frac{1}{\rho}-\frac{1}{2}}\tau^{-\frac{1}{2}} +\ell^{-7}r_{\perp}^{\frac{2}{\rho}}r_{\parallel}^{\frac{1}{\rho}-\frac{3}{2}}\tau^{-\frac{1}{2}}		+\ell^{-14}\mu^{-1}r_{\perp}^{\frac{2}{\rho}-2}r_{\parallel}^{\frac{1}{\rho}-1}	+\ell^{-20}\sigma^{-1})\notag\\
		&\lesssim
		(\|v_{q}\|_{C_{t,x}}
		+\|z^{u}\|_{C_{(\frac{\varsigma_{q}}{2}-\ell,T]}L_{x}^{\infty}}
		+\|Z_{q}\|_{C_{t}L_{x}^{\infty} })
		(1+\|\Ru_{q}\|_{C_{t}L_{x}^{1}}^{5})\ell^{-20}\sigma^{-1}.
\end{align}
Note that, by 
{\citep[(5.34)]{CLZ}},  
\begin{align}\label{4-34}
	\|z^{u}\|_{C_{(\frac{\varsigma_{q}}{2}-\ell, T]} L_{x}^{\infty}}
	\lesssim \varsigma_{q}^{-\frac{3}{4\a}}M,
\end{align}	
which along with \eqref{4-33}  yields that
\begin{align}\label{4-35}
\text{R.H.S.\ of\ (\ref{4-33})}
	\lesssim
	(\|v_{q}\|_{C_{t,x}}+\varsigma_{q}^{-\frac{3}{4\a}}M+\|Z_{q}\|_{C_{t}L_{x}^{\infty} })(1+\|\Ru_{q}\|_{C_{t}L_{x}^{1}}^{5})\ell^{-20}\sigma^{-1}.
\end{align}

Now,  summing up  (\ref{4-28-0}), (\ref{4-29}), (\ref{4-32}) and  (\ref{4-35}) we obtain
\begin{align*} 
		\|\Ru_{lin}\|_{L_{(\frac{\varsigma_{q}}{2}, T]}^{1}L_x^{1}}
		&\lesssim
	   	(1+\|\Ru_{q}\|_{C_{t}L_{x}^{1}}^{4})\ell^{-13}r_{\perp}^{\frac{2}{\rho}}r_{\parallel}^{\frac{1}{\rho}-\frac{3}{2}}\mu	\tau^{-\frac{1}{2}}
	   	+(1+\|\Ru_{q}\|_{C_{t}L_{x}^{1}}^{2\a+4})\ell^{-7}r_{\perp}^{\frac{2}{\rho}-1}r_{\parallel}^{\frac{1}{\rho}-\frac{1}{2}}\laq^{2\a-1}\tau^{-\frac{1}{2}}\\
		&\quad
		+\ell^{\frac{1}{2}}(1+\varsigma_{q}^{-\frac{1}{2}})
		M+\la^{-15(1-\delta)}(\|Z\|_{C^{\frac{1}{2}-\delta}_{t} L_{x}^{2}}+\|Z\|_{C_{t} H_{x}^{1-\delta}})\\
		&\quad+	(\|v_{q}\|_{C_{t,x}}+\varsigma_{q}^{-\frac{3}{4\a}}M+\|Z_{q}\|_{C_{t}L_{x}^{\infty} })(1+\|\Ru_{q}\|_{C_{t}L_{x}^{1}}^{5})\ell^{-20}\sigma^{-1},
\end{align*}
which along with (\ref{2-7}),  (\ref{2-14}), (\ref{2-15}), (\ref{2-22}) and (\ref{3-1})  yields that
\begin{align}\label{4-36}
		\|\Ru_{lin}\|_{\Lo^{r}_{\om} L_{(\frac{\varsigma_{q}}{2}, T]}^{1}L_{x}^{1}}
		&\lesssim
		 (1+\|\Ru_{q}\|_{\Lo^{4r}_{\om}C_{t}L_{x}^{1}}^{4})\ell^{-13}r_{\perp}^{\frac{1}{\rho}}r_{\parallel}^{\frac{1}{\rho}-\frac{3}{2}}\mu\tau^{-\frac{1}{2}}
	    +(1+\|\Ru_{q}\|_{\Lo^{8r}_{\om}C_{t}L_{x}^{1}}^{8})\ell^{-7}r_{\perp}^{\frac{2}{\rho}-1}r_{\parallel}^{\frac{1}{\rho}-\frac{1}{2}}\laq^{2\a-1}\tau^{-\frac{1}{2}} \notag\\
		&\quad
		+\ell^{\frac{1}{2}}(1+\varsigma_{q}^{-\frac{1}{2}})M
		+\la^{-15(1-\delta)}(\|Z\|_{\Lo^{2r}_{\om}C_{t}^{1/2-\delta}L_{x}^{2} }+\|Z\|_{\Lo^{2r}_{\om}C_{t}H_{x}^{1-\delta}})\notag\\
		&\quad
		+(\|v_{q}\|_{\Lo^{2r}_{\om}C_{t,x}}+\varsigma_{q}^{-\frac{3}{4\a}}M+\|Z_{q}\|_{\Lo^{2r}_{\om}C_{t}L_{x}^{\infty}} )
	(1+\|\Ru_{q}\|_{\Lo^{10r}_{\om}C_{t}L_{x}^{1}}^{5})\ell^{-20}\sigma^{-1}\notag\\
	%&\lesssim \big(\lambda_{q}^{(4\a+12)}(8\cdot 4rL^{2}50^{q})^{50^{q}}\big)^{4}\ell^{-13}\lambda_{q+1}^{-\frac{1}{2}\varepsilon} +\big(\lambda_{q}^{(4\a+12)}(8\cdot 8rL^{2}50^{q})^{50^{q}}\big)^{8}\ell^{-7}\lambda_{q+1}^{-\frac{1}{2}\varepsilon} +\lambda_{q}^{-20}+\la^{-14}(2r-1)^{\frac{1}{2}}L\notag\\
	%&\quad+\(\big(\lambda_{q}^{(2\a+5)}(8\cdot 20rL^{2}50^{q-1})^{50^{q-1}}\big)^{10}+\varsigma_{q}^{-\frac{3}{4\a}}+(2r-1)^{\frac{1}{2}}L\)\big(\lambda_{q}^{(4\a+12)}(8\cdot 10rL^{2}50^{q})^{50^{q}}\big)^{5}\ell^{-20}\lambda_{q+1}^{-2\varepsilon}\notag\\
		%&\lesssim	\lambda_{q}^{-14}+\lambda_{q+1}^{-\frac{3}{8}\varepsilon}
		 &\lesssim
		\dqqq.
\end{align}

\subsubsection{Oscillation error}
Let us now turn to the oscillation errors in \eqref{4-8}-\eqref{4-11}. First, by (\ref{3-49}), (\ref{3-50}), (\ref{3-58}), (\ref{3-59})  and (\ref{4-8}),  we get
\begin{align}\label{4-37-1}
	\|\Ru_{osc.1}\|_{L_{(\frac{\varsigma_{q}}{2}, T]}^{1}L_{x}^{1}}
	&\lesssim
\|1-\Theta_{q+1}^{2}\|_{L_{t}^{1}}\|\Ru_{q}\|_{C_{t}L_{x}^{1}}+\|\Theta_{q+1}^{2}\|_{C_{t}^{1}}(\|w^{(t)}_{q+1}\|_{L_{t}^{1} L_{x}^{\rho}}+\|w^{(o)}_{q+1})\|_{L_{t}^{1} L_{x}^{\rho}})\notag\\
	&\lesssim
	\varsigma_{q}\|\Ru_{q}\|_{C_{t}L_{x}^{1}}+\varsigma_{q}^{-1}(1+\|\Ru_{q}\|_{C_{t}L_{x}^{1}}^{5})(\ell^{-14}\mu^{-1}r_{\perp}^{\frac{2}{\rho}-2}r_{\parallel}^{\frac{1}{\rho}-1}+\ell^{-20}\sigma^{-1})\notag\\
	&\lesssim
	\varsigma_{q}\|\Ru_{q}\|_{C_{t}L_{x}^{1}}+(1+\|\Ru_{q}\|_{C_{t}L_{x}^{1}}^{5})\ell^{-20}\sigma^{-1}\varsigma_{q}^{-1}.
\end{align}

For the high-low spatial oscillation  error $\Ru_{osc.2}$ in  (\ref{4-9}),
 we need the following lemma:
 
 \begin{lemma}[\citep{LT-hyperviscous NS}] 
 %[\citep[Lemma 6]{LT-hyperviscous NS}]
 \label{Lemma commutator estimate1}
 Let $a \in C^2\left(\mathbb{T}^3\right)$. For all $1<p<+\infty$, we have
\begin{equation}\label{6.8}
	\begin{aligned}
	\||\nabla|^{-1} \mathbb{P}_{\neq 0}\left(a \mathbb{P}_{\geq k} f\right)\|_{L^p\left(\mathbb{T}^3\right)} \lesssim k^{-1}\|\nabla^2 a\|_{L^{\infty}\left(\mathbb{T}^3\right)}\|f\|_{L^p\left(\mathbb{T}^3\right)},
\end{aligned}
\end{equation}
	holds for any smooth function $f \in L^p\left(\mathbb{T}^3\right)$.
\end{lemma}
 
 Applying  Lemmas \ref{Lemma spatial building blocks}, \ref{Lemma temporal building blocks}, \ref{Lemma amplitudes 1} and \ref{Lemma commutator estimate1} we  obtain
\begin{align}\label{4-37-2}
	\|\Ru_{osc.2}\|_{L_{(\frac{\varsigma_{q}}{2}, T]}^{1} L_{x}^{\rho}}
		&\lesssim
		\sum_{k\in\Lambda}\|\Theta_{q+1}^{2}\|_{C_{t}}\|g_{(k)}^{2}\|_{L_t^{1}}
		\||\nabla|^{-1}\mathbb{P}_{\neq 0}(\mathbb{P}_{\geq \frac{\laq r_{\perp}}{2} }(W_{(k)}\otimes W_{(k)}))\nabla(a_{(k)}^{2})\|_{C_{t} L_{x}^{\rho}}\notag\\
		&\lesssim
		\sum_{k\in\Lambda}\|a_{(k)}^{2}\|_{C_{t,x}^{3}}\(\laq r_{\perp}\)^{-1}\|\psi_{(k_{1})}^{2}\|_{C_{t}L_{x}^{\rho}}\|\phi_{(k)}^{2}\|_{L_{x}^{\rho}}\notag\\
		&\lesssim
		(1+\|\Ru_{q}\|_{C_{t}L_{x}^{1}}^{7})\ell^{-32}r_{\perp}^{\frac{2}{\rho}-3}r_{\parallel}^{\frac{1}{\rho}-1}\lambda_{q+1}^{-1}.
\end{align}

 By Lemmas \ref{Lemma spatial building blocks}, \ref{Lemma temporal building blocks}, \ref{Lemma amplitudes 1}  and (\ref{4-10}), we also have
\begin{align}\label{4-37-3}
	\|\Ru_{osc.3}\|_{L_{(\frac{\varsigma_{q}}{2}, T]}^{1} L_{x}^{\rho}}
	&\lesssim
	\mu^{-1}\|\Theta_{q+1}^{2}\|_{C_{t}}\sum_{k\in\Lambda}\| \mathcal{R}\mathbb{P}_{\neq 0}(\p_{t}(a_{(k)}^{2}g_{(k)}^{2})\psi_{(k_{1})}^{2}\phi_{(k)}^{2})\|_{C_{t} L_{x}^{\rho}}\notag\\
	&\lesssim
	\mu^{-1}\sum_{k\in\Lambda}\|a_{(k)}^{2}\|_{C_{t,x}^{1}}\|g_{(k)}^{2}\|_{W_{t}^{1,1}}\|\psi_{(k_{1})}^{2}\|_{C_{t} L_{x}^{\rho}}\|\phi_{(k)}^{2})\|_{L_{x}^{\rho}}\notag\\
	&\lesssim
	(1+\|\Ru_{q}\|_{C_{t}L_{x}^{1}}^{5})\ell^{-20}\mu^{-1}r_{\perp}^{\frac{2}{\rho}-2}r_{\parallel}^{\frac{1}{\rho}-1}\sigma\tau.
\end{align}

Moreover, for the low frequence error, by Lemmas \ref{Lemma temporal building blocks}, \ref{Lemma amplitudes 1}  and \eqref{4-11}, 
\begin{align}\label{4-37-4}
		\|\Ru_{osc.4}\|_{L_{(\frac{\varsigma_{q}}{2}, T]}^{1} L_{x}^{\rho}}
		&\lesssim
		\sigma^{-1}\|\Theta_{q+1}^{2}\|_{C_{t}}\sum_{k\in\Lambda}\|h_{(k)}\|_{L_t^{1}}
		\|a_{(k)}^{2}\|_{C_{t,x}^{2}}\notag\\
		&\lesssim
	(1+\|\Ru_{q}\|_{C_{t}L_{x}^{1}}^{6})\ell^{-26}\sigma^{-1}.
\end{align}

Thus, combining \eqref{4-37-1}, \eqref{4-37-2}-\eqref{4-37-4} altogether, we come to
\begin{align*}
	\|\Ru_{osc}\|_{L_{(\frac{\varsigma_{q}}{2}, T]}^{1}L_x^{1}}
	%&\lesssim\varsigma_{q}\|\Ru_{q}\|_{C_{t}L_{x}^{1}}+(1+\|\Ru_{q}\|_{C_{t}L_{x}^{1}}^{5})\ell^{-20}\sigma^{-1}\varsigma_{q}^{-1}+(1+\|\Ru_{q}\|_{C_{t}L_{x}^{1}}^{7})\ell^{-32}r_{\perp}^{\frac{2}{\rho}-3}r_{\parallel}^{\frac{1}{\rho}-1}\lambda_{q+1}^{-1}\notag\\
	%&\quad+	(1+\|\Ru_{q}\|_{C_{t}L_{x}^{1}}^{5})\ell^{-20}\mu^{-1}r_{\perp}^{\frac{2}{\rho}-2}r_{\parallel}^{\frac{1}{\rho}-1}\sigma\tau+(1+\|\Ru_{q}\|_{C_{t}L_{x}^{1}}^{6})\ell^{-26}\sigma^{-1}\notag\\
	&\lesssim
	\varsigma_{q}\|\Ru_{q}\|_{C_{t}L_{x}^{1}}+(1+\|\Ru_{q}\|_{C_{t}L_{x}^{1}}^{7})\ell^{-32}r_{\perp}^{\frac{2}{\rho}-3}r_{\parallel}^{\frac{1}{\rho}-1}\lambda_{q+1}^{-1},
\end{align*}
which, via (\ref{2-7}), (\ref{2-15}) and (\ref{3-1}), yields that
\begin{align}\label{4-37}
	\|\Ru_{osc}\|_{\Lo^{r}_{\om}L_{(\frac{\varsigma_{q}}{2}, T]}^{1}L_{x}^{1}}
	&\lesssim
\varsigma_{q}\|\Ru_{q}\|_{\Lo^{r}_{\om}C_{t}L_{x}^{1}}
+(1+\|\Ru_{q}\|_{\Lo^{7r}_{\om}C_{t}L_{x}^{1}}^{7})\ell^{-32}r_{\perp}^{\frac{2}{\rho}-3}r_{\parallel}^{\frac{1}{\rho}-1}\lambda_{q+1}^{-1}
	%&\lesssim\varsigma_{q}\big(\lambda_{q}^{(4\a+12)}(8\cdot rL^{2}50^{q})^{50^{q}}\big)+\big(\lambda_{q}^{(4\a+12)}(8\cdot 7rL^{2}50^{q})^{50^{q}}\big)^{7}\ell^{-32}\lambda_{q+1}^{-\varepsilon}\notag\\
	%&\lesssim \lambda_{q}^{-19}+\lambda_{q+1}^{-\frac{1}{2}\varepsilon}
	\lesssim
	\dqqq.
\end{align}

\subsubsection{Corrector error}
%\paragraph{\bf Estimate of corrector errors}
By (\ref{3-49}), (\ref{3-50}),  (\ref{4-3}) and Lemma \ref{Lemma Estimates of perturbations},
\begin{align*}
		\|\Ru_{corr}\|_{L_{(\frac{\varsigma_{q}}{2}, T]}^{1}L_x^{1}}
		&\lesssim
	\|\omw_{q+1}^{(c)+(t)+(o)}\|_{L_t^{2}L_x^{2}}	(\|w_{q+1}\|_{L_t^{2}L_x^{2}}+\|\omw_{q+1}^{(p)}\|_{L_t^{2}L_x^{2}})
		\notag\\
		%&\lesssim	(1+\|\Ru_{q}\|_{C_{t}L_{x}^{1}}^{5})(\ell^{-7}r_{\perp}r_{\parallel}^{-1}	+\ell^{-14}\mu^{-1}r_{\perp}^{-1}r_{\parallel}^{-\frac{1}{2}}\tau^{\frac{1}{2}}+\ell^{-20}\sigma^{-1})\notag\\
		%&\quad\times	(1+\|\Ru_{q}\|_{C_{t}L_{x}^{1}}^{5})(\ell^{-7}	+\ell^{-7}r_{\perp}r_{\parallel}^{-1}	+\ell^{-14}\mu^{-1}r_{\perp}^{-1}r_{\parallel}^{-\frac{1}{2}}\tau^{\frac{1}{2}}	+\ell^{-20}\sigma^{-1})\notag\\
		&\lesssim
		(1+\|\Ru_{q}\|_{C_{t}L_{x}^{1}}^{10})\ell^{-27}\sigma^{-1}.
\end{align*}
Taking into account  (\ref{2-7}), (\ref{2-15}) and (\ref{3-1})  we get
\begin{align}\label{4-38}
		\|\Ru_{corr}\|_{\Lo^{r}_{\om}L_{(\frac{\varsigma_{q}}{2}, T]}^{1}L_{x}^{1}}
		%&\lesssim	(1+\|\Ru_{q}\|_{\Lo^{10r}_{\om}C_{t}L_{x}^{1}}^{10})\ell^{-27}\sigma^{-1}\notag\\
		&\lesssim
	\big(\lambda_{q}^{(4\a+12)}\(8\cdot 10rL^{2}50^{q}\)^{50^{q}}\big)^{10}\ell^{-27}\laq^{-2\varepsilon}
		%&\lesssim\laq^{-\varepsilon}
		\lesssim
		\dqqq.
\end{align}

\subsubsection{ Commutator errors}
By the standard mollification estimates, (\ref{2-27}) and  (\ref{4-31}), it holds that
\begin{align*}%\label{4-39-0}
		\|\Ru_{com1}\|_{L_{(\frac{\varsigma_{q}}{2}, T]}^{1}L_x^{1}}
		&\lesssim
			\ell^{\frac{1}{2}-\delta}(\|v_{q}\|_{C_{t,x}^{1}}
			+\|z\|_{C_{(\frac{\varsigma_{q}}{2}-\ell, T]}^{\frac{1}{2}-\delta}L_{x}^{2}}
			+\|z\|_{C_{(\frac{\varsigma_{q}}{2}, T]} H_{x}^{1-\delta}})(\|v_{q}\|_{C_{t,x}}+\|z\|_{C_{t} L_{x}^{2}})\notag\\
		&\lesssim
		\ell^{\frac{1}{2}-\delta}	\big(\|v_{q}\|_{C_{t,x}^{1}}+	(1+\varsigma_{q}^{-(\frac{1}{2}-\delta)})M
		+\|Z\|_{C_{t}^{\frac{1}{2}-\delta} L_{x}^{2}}+\|Z\|_{C_{t} H_{x}^{1-\delta}}\big)
		(\|v_{q}\|_{C_{t,x}}+\|z\|_{C_{t} L_{x}^{2}}).
\end{align*}
Thus,  by (\ref{stochastic evolution}), (\ref{2-7}), (\ref{2-14}) and (\ref{3-1}), we get
\begin{align}\label{4-39}
\|\Ru_{com1}\|_{\Lo^{r}_{\om}L_{(\frac{\varsigma_{q}}{2}, T]}^{1}L_{x}^{1}}
	&\lesssim
	\ell^{\frac{1}{2}-\delta}\big(\|v_{q}\|_{\Lo^{2r}_{\om}C_{t,x}^{1}}+(1+\varsigma_{q}^{-(\frac{1}{2}-\delta)})M
	+\|Z\|_{\Lo^{2r}_{\om}C_{t}^{\frac{1}{2}-\delta}L_{x}^{2}}+\|Z\|_{\Lo^{2r}_{\om}C_{t}H_{x}^{1-\delta}}\big)\notag\\
	&\quad
	\times(\|v_{q}\|_{\Lo^{2r}_{\om}C_{t,x}}+\|z\|_{\Lo^{2r}_{\om}C_{t} L_{x}^{2}})\notag\\
	%&\lesssim \ell^{\frac{1}{2}-\delta}\big( \la^{(4\a+5)} (50^{q-1}\cdot8\cdot   22rL^{2})^{11\cdot50^{q-1}}+\varsigma_{q}^{-(\frac{1}{2}-\delta)}M+\(2r-1\)^{\frac{1}{2}}L\big)\notag\\
	%&\quad\times\big(\la^{(2\a+5)}(8\cdot20rL^{2}50^{q-1})^{10\cdot50^{q-1}}+M+(2r-1)^{\frac{1}{2}}L\big)\notag\\
	%&\lesssim %\lambda_{q}^{-39}\big(\lambda_{q}^{4\a+6}+\lambda_{q}^{20}\big)\lambda_{q}^{2\a+6}
%	\lesssim
%	\lambda_{q}^{-9}
	&\lesssim
	\dqqq.		
\end{align}

Concerning the second commutator error $ \Ru_{com2} $,
by (\ref{3-55}), (\ref{3-80}), (\ref{4-4}),  (\ref{4-31}) and the mollification estimate,
\begin{align*}
		\|\Ru_{com2}\|_{L_{(\frac{\varsigma_{q}}{2}, T]}^{1}L_x^{1}}
		&\lesssim
		(\|v_{q+1}\|_{L_{(\frac{\varsigma_{q}}{2}, T]}^{1}L_{x}^{2}}+\|z_{\ell}\|_{L_{(\frac{\varsigma_{q}}{2}, T]}^{1}L_{x}^{2}}+\|z_{q+1}\|_{L_{(\frac{\varsigma_{q}}{2}, T]}^{1}L_{x}^{2}})\notag\\
		&\quad\times
		(\|z_{\ell}^{u}- z^{u}\|_{C_{(\frac{\varsigma_{q}}{2}, T]} L_{x}^{2}}
		+\|Z_{\ell}- Z_{q+1}\|_{C_{(\frac{\varsigma_{q}}{2}, T]} L_{x}^{2}})\notag\\
		&\lesssim
		(\|v_{q}\|_{C_{t,x}}+\|w_{q+1}\|_{L_{t}^{1}L_{x}^{2}}+\|z\|_{C_{(\frac{\varsigma_{q}}{2}-\ell, T]}L_{x}^{2}})\notag\\
	&\quad\times \big(\ell^{\frac{1}{2}}(\|z^{u}\|_{C_{(\frac{\varsigma_{q}}{2}-\ell, T]}^{\frac{1}{2}} L_{x}^{2}}
	+\|z^{u}\|_{C_{t} H_{x}^{\frac{1}{2}}})+\ell^{\frac{1}{2}-\delta}\|Z\|_{C_{t}^{\frac{1}{2}-\delta} L_{x}^{2}}+\la^{-15(1-\delta)}\|Z\|_{C_{t} H_{x}^{1-\delta}}\big)\notag\\
		&\lesssim
		\big(\|v_{q}\|_{C_{t,x}}
		+(1+\|\Ru_{q}\|_{C_{t}L_{x}^{1}}^{5})\ell^{-20}\sigma^{-1}+\|z\|_{C_{(\frac{\varsigma_{q}}{2}-\ell, T]}L_{x}^{2}}\big)\notag\\
		&\quad\times \big(\ell^{\frac{1}{2}}(1+\varsigma_{q}^{-\frac{1}{2}})M+\la^{-15(1-\delta)}(\|Z\|_{C_{t}^{\frac{1}{2}-\delta} L_{x}^{2}}+\|Z\|_{C_{t} H_{x}^{1-\delta}})\big),
\end{align*}
which, via (\ref{stochastic evolution}), (\ref{2-7}), (\ref{2-14}), (\ref{2-15}) and (\ref{3-1}), yields that
\begin{align}\label{4-40}
	&\quad\|\Ru_{com2}\|_{\Lo^{r}_{\om}L_{(\frac{\varsigma_{q}}{2}, T]}^{1}L_{x}^{1}}\notag\\
		&\lesssim
		\big(\|v_{q}\|_{\Lo^{2r}_{\om}C_{t,x}}+(1+\|\Ru_{q}\|_{\Lo^{10r}_{\om}C_{t}L_{x}^{1}}^{5})\ell^{-20}\sigma^{-1}+\|z\|_{\Lo^{2r}_{\om}C_{(\frac{\varsigma_{q}}{2}-\ell, T]} L_{x}^{2}}\big)\notag\\
		&\quad\times
		\big(\ell^{\frac{1}{2}}(1+\varsigma_{q}^{-\frac{1}{2}})M+\la^{-15(1-\delta)}(\|Z\|_{\Lo^{2r}_{\om}C_{t}^{\frac{1}{2}-\delta}L_{x}^{2}}+\|Z\|_{\Lo^{2r}_{\om}C_{t}H_{x}^{1-\delta}})\big)\notag\\
		%&\lesssim\( \la^{(2\a+5)}(8\cdot   20rL^{2}50^{q-1})^{10\cdot50^{q-1}}	+\big(\lambda_{q}^{(4\a+12)}(8\cdot10rL^{2}50^{q})^{50^{q}}\big)^{5}\ell^{-20}\laq^{-2\varepsilon}+(2r-1)^{\frac{1}{2}}L+M\)\notag\\
		%&\quad\times\big(\ell^{\frac{1}{2}}\varsigma_{q}^{-\frac{1}{2}}M+\la^{-15(1-\delta)}(2r-1)^{\frac{1}{2}}L\big)\notag\\
		%&\lesssim
		%\lambda_{q}^{-4}
		&\lesssim
		\dqqq.
\end{align}

Therefore,  combining (\ref{4-36}) and \eqref{4-37}-\eqref{4-40} altogether,  we conclude that
	\begin{align*}%\label{4-41}
		\|\Ru_{q+1}\|_{\Lo^{r}_{\om}L_{(\frac{\varsigma_{q}}{2}, T]}^{1}L_{x}^{1}}\lesssim \dqqq.
\end{align*}
This verifies the decay estimate \eqref{2-16} at level $q+1$ in the temporal regime $(\frac{\varsigma_{q}}{2}, T]$ away from the initial time.

\subsection{Verification of inductive estimate \eqref{2-16}: near the initial time} 
We continue to verify the inductive estimate (\ref{2-16})  
in the regime $[0, \frac{\varsigma_{q}}{2}]$ near the initial time. 

In this case,  $\Theta_{q+1}(t)=0$ and so $w_{q+1}=0$. Then, by  (\ref{4-1}), we get
\begin{align}\label{4-42}
	\Ru_{q+1}&=\mathcal{R}\(z_{\ell}- z_{q+1}\)+\Ru_{\ell}+	(v_{\ell}+z_{\ell})\mathring{\otimes}(v_{\ell}+z_{\ell})-\left((v_{q}+z_{q})\mathring{\otimes}(v_{q}+z_{q})\right)*_x \varrho_{\ell} *_t \vartheta_{\ell}\notag\\
	&\quad+v_{\ell}\mathring{\otimes}z_{q+1}
	+z_{q+1}\mathring{\otimes} v_{\ell}
	-v_{\ell}\mathring{\otimes}z_{\ell}
	-z_{\ell}\mathring{\otimes} v_{\ell}
	+z_{q+1}\mathring{\otimes} z_{q+1}
	-z_{\ell}\mathring{\otimes}z_{\ell}.
\end{align}
Then, by (\ref{4-42}) and H\"{o}lder's inequality, we obtain
\begin{align}\label{4-42-1}
\|\Ru_{q+1}\|_{ L_{[0, \frac{\varsigma_{q}}{2}]}^{1}L_{x}^{1}}
&\lesssim \|z\|_{L_{[0, \frac{\varsigma_{q}}{2}]}^{1}L_{x}^{2}}
+\|z_{\ell}\|_{L_{[0, \frac{\varsigma_{q}}{2}]}^{1}L_{x}^{2}}
+\|\Ru_{\ell}\|_{ L_{[0,\frac{\varsigma_{q}}{2}]}^{1}L_{x}^{1}}\notag\\
&\quad+(\|v_{q}\|_{L_{[0, \frac{\varsigma_{q}}{2}]}^{1}L_{x}^{2}}
+\|z\|_{L_{[0, \frac{\varsigma_{q}}{2}]}^{1}L_{x}^{2}})
(\|v_{q}\|_{C_{[0, \frac{\varsigma_{q}}{2}]}L_{x}^{2}}
+\|z\|_{C_{[0, \frac{\varsigma_{q}}{2}]}L_{x}^{2}})
\notag\\
&\quad+
(\|v_{q}\|_{L_{[0, \frac{\varsigma_{q}}{2}]}^{1}L_{x}^{2}}+\|z\|_{L_{[0, \frac{\varsigma_{q}}{2}]}^{1}L_{x}^{2}})\|z\|_{C_{[0, \frac{\varsigma_{q}}{2}]}L_{x}^{2}}\notag\\
	&\lesssim
	\varsigma_{q}(\|z\|_{C_{[0, \frac{\varsigma_{q}}{2}]}L_{x}^{2}}+\|\Ru_{q}\|_{ C_{[0, \frac{\varsigma_{q}}{2}]}L_{x}^{1}}
	+\|v_{q}\|_{C_{[0, \frac{\varsigma_{q}}{2}],x}}^{2}+\|z\|_{C_{[0, \frac{\varsigma_{q}}{2}]}L_{x}^{2}}^{2}).
\end{align}
Thus, by (\ref{stochastic evolution}), (\ref{2-7}), (\ref{2-14}), (\ref{2-15}) and (\ref{3-1}), we have
\begin{align}\label{4-42-2}
	\|\Ru_{q+1}\|_{\Lo^{r}_{\om} L_{[0, \frac{\varsigma_{q}}{2}]}^{1}L_{x}^{1}}
	&\lesssim
	\varsigma_{q}(\|z\|_{\Lo^{2r}_{\om}C_{[0, \frac{\varsigma_{q}}{2}]}L_{x}^{2}}+\|\Ru_{q}\|_{ \Lo^{r}_{\om}C_{[0, \frac{\varsigma_{q}}{2}]}L_{x}^{1}}
	+\|v_{q}\|_{\Lo^{2r}_{\om}C_{[0, \frac{\varsigma_{q}}{2}],x}}^{2}+\|z\|_{\Lo^{2r}_{\om}C_{[0, \frac{\varsigma_{q}}{2}]}L_{x}^{2}}^{2})\notag\\
	%&\lesssim\la^{-40}\(M+(r-1)^{\frac{1}{2}}L+\lambda_{q}^{(4\a+12)}(8rL^{2}50^{q})^{50^{q}}+\la^{4\a+10}(8\cdot20rL^{2}50^{q-1})^{20\cdot50^{q-1}}+M^{2}+(2r-1)L^{2}\)\notag\\
	%&\lesssim\lambda_{q}^{-19}
	&\lesssim
	\dqqq.
\end{align}
Therefore, the inductive estimate (\ref{2-16}) at level $q+1$ is verified.

Consequently, we conclude from Sections \ref{Sec-S1} and \ref{Sec-Reynolds} that the main iteration in Proposition \ref{Proposition Main iteration}
 holds in the supercritical regime $\mathcal{S}_{1}$. Thus, we infer from Section \ref{Sec-Main-Iteration} that Theorems \ref{Thm-Nonuniq-Hyper} and \ref{Theorem Vanishing noise} hold in the $\mathcal{S}_{1}$ regime.

\section{Velocity perturbations in the supercritical regime \texorpdfstring{$\mathcal{S}_{2}$}{S2}} \label{Sec-S2}

From this section, we   treat  
the main iteration in the supercritical regime $\mathcal{S}_{2}$. Let us start with the building blocks of the velocity perturbations. Unlike in Section \ref{Sec-S1}, the spatial building blocks are the concertrated Mikado flows introduced in {\citep{DS}}.

\subsection{Mikado flows}  
We use  the spatial-temporal building blocks  indexed by four parameters $r_{\perp}, \lambda, \tau$ and $\sigma$ :
\begin{equation}\label{5-1}
	\begin{aligned}
		r_{\perp}:=\lambda_{q+1}^{- \alpha+1-8 \varepsilon},\quad \lambda:=\lambda_{q+1}, \quad
		\tau:=\lambda_{q+1}^{2 \alpha},\quad \sigma:=\lambda_{q+1}^{2 \varepsilon},
	\end{aligned}
\end{equation}
where $\varepsilon$ satisfies (\ref{2-9}). The corresponding Mikado flows is defined by
	\begin{align}\label{5-2}
	W_{(k)}:=\phi_{r_{\perp}}(\lambda r_{\perp} N_{\Lambda} k \cdot x, \lambda r_{\perp} N_{\Lambda} k_{2} \cdot x) k_1, \quad k \in \Lambda,
\end{align}
where   $\phi_{r_{\perp}}$, $r_{\perp}$, $N_{\Lambda}$ and $(k,k_{1},k_{2})$ are same as in Section \ref{Sec-S1}. Additionally, we also use  the temporal building blocks $g_{(k)}$, $h_{(k)}$ 
as in (\ref{3-19}). 

For the sake of simplicity, we set
\begin{align}
	\phi_{(k)}(x):=\phi_{r_{\perp}}\big(\lambda r_{\perp} N_{\Lambda} k \cdot x,\lambda r_{\perp} N_{\Lambda} k_{2} \cdot x \big),\label{5-3-0}\\ \Phi_{(k)}(x):=\Phi_{r_{\perp}}\big(\lambda r_{\perp} N_{\Lambda} k \cdot x, \lambda r_{\perp} N_{\Lambda} k_{2} \cdot x \big),\label{5-3}
\end{align}
and rewrite
\begin{align}\label{5-4}
	W_{(k)}:=\phi_{(k)}(x) k_1, \quad k \in \Lambda.
\end{align}
{Then, $W_{(k)}=\curl\curl W_{(k)}^{c}$ with}
\begin{align}\label{5-5} 
	W_{(k)}^{c} & :=\frac{1}{\lambda^2 N_{\Lambda}^2} \Phi_{(k)} k_1, \quad k \in \Lambda.
\end{align}

The following  result  gives the analytic estimates of our building blocks.

\begin{lemma} [Estimate of Mikado flows, \textcolor{blue}{\citep{BMS}}] \label{Lemma spatial building blocks2}  

For any $p \in[1, \infty]$ and $N \in \bbn$, one has
	\begin{align}\label{5-6}
		\|\nabla^N \phi_{(k)}\|_{L_x^p}+\|\nabla^N \Phi_{(k)}\|_{L_x^p} \lesssim r_{\perp}^{\frac{2}{p}-1} \lambda^N.
	\end{align}
	In particular,
	\begin{align}\label{5-7}
		&\|\nabla^N W_{(k)}\|_{ L_x^p}+\lambda^2\|\nabla^N W_{(k)}^c\|_{ L_x^p} \lesssim r_{\perp}^{\frac{2}{p}-1} \lambda^N, \quad k \in \Lambda.
	\end{align} 
	The implicit constants above are deterministic and  independent of the parameters $r_{\perp}$ and $\lambda$.
\end{lemma}

\subsection{Amplitudes} We define the amplitudes of the velocity perturbations by 
\begin{align}\label{5-8}
	a_{(k)}(t,x):=\varrho^{1/2}(t,x)\gamma_{(k)}\big(\mathrm{Id}-\frac{\Ru_{\ell}(t,x)}{\varrho(t,x)}\big), \quad k\in \Lambda,
\end{align}
where $\varrho$ and  $\gamma_{(k)}$ are as in Section \ref{Sec-S1}. Note that, the amplitudes  obey the same estimates as in Lemma \ref{Lemma amplitudes 1}.

\subsection{Velocity perturbations}  The velocity perturbations in the $\mathcal{S}_{2}$ regime consist of three components: the principle part, the incompressibility corrector and the temporal corrector.
\paragraph{\bf Principal part.} 
We define the principle part  of the velocity perturbations by
\begin{align}\label{5-10}
	w_{q+1}^{(p)}:=\sum_{k \in \Lambda} a_{(k)} g_{(k)} W_{(k)}.
\end{align}
Using the Geometric Lemma \ref{Lemma First Geometric} one still has the algebraic identity (\ref{3-43-1}).

\paragraph{\bf Incompressibility corrector.} We define the corresponding incompressibility corrector  by
	\begin{align}\label{5-11}
		w_{q+1}^{(c)} & :=\sum_{k \in \Lambda} g_{(k)}\left(\nabla a_{(k)} \times \curl W_{(k)}^c+\curl(\nabla a_{(k)} \times W_{(k)}^c)\right) .
\end{align}
It holds that
	\begin{align}\label{5-12}
		w_{q+1}^{(p)}+w_{q+1}^{(c)}&=\sum_{k \in \Lambda} \operatorname{curlcurl}(a_{(k)} g_{(k)} W_{(k)}^c),
\end{align}
and so 
\begin{align}\label{5-13}
		\operatorname{div}(w_{q+1}^{(p)}+w_{q+1}^{(c)})=0.
\end{align}
\paragraph{\bf Temporal corrector to balance temporal oscillations.}   In the regime $\mathcal{S}_{2}$, we only use one type of the  temporal corrector  defined  by	\begin{align}\label{5-14}
		w_{q+1}^{(o)}:= & -\sigma^{-1} \sum_{k \in \Lambda} \mathbb{P}_H \mathbb{P}_{\neq 0}\left(h_{(k)} \fint_{\mathbb{T}^3} W_{(k)} \otimes W_{(k)} \mathrm{d} x \nabla(a_{(k)}^2)\right).
\end{align}
One  also has the algebraic identity (\ref{3-48-2}), which permits to balance the high temporal oscillations.
 
Define the cut-off perturbations by
\begin{align}\label{5-15}
	&\widetilde{w}_{q+1}^{(p)}:=\Theta_{q+1} w_{q+1}^{(p)}, \quad \widetilde{w}_{q+1}^{(c)}:=\Theta_{q+1} w_{q+1}^{(c)}, \quad \widetilde{w}_{q+1}^{(o)}:=\Theta_{q+1}^2 w_{q+1}^{(o)},
\end{align}
where $\Theta_{q+1}$ is the temporal cut-off function as in (\ref{3-49}).

Analogously to Subsection \ref{subsec3.3}, we define the velocity perturbations $w_{q+1}$ at level $q+1$ by
\begin{align}\label{5-16}
	w_{q+1}=\widetilde{w}_{q+1}^{(p)}+\widetilde{w}_{q+1}^{(c)}+\widetilde{w}_{q+1}^{(o)},
\end{align}
with
\begin{align}\label{5-17}
	\omw^{(*_{1})+(*_{2})}_{q+1}:=\omw^{(*_{1})}_{q+1}+\omw^{(*_{2})}_{q+1}, \quad\text{where} \quad *_{1},~*_{2}\in \{p, c, o \}. 
\end{align}

At last,  the new velocity  fields at level ${q+1}$ is defined by
\begin{align}\label{5-18}
		v_{q+1}:=v_{\ell}+w_{q+1}.
\end{align}

In the following Lemma, we give the  estimates of the velocity perturbations.
\begin{lemma}[Estimates of perturbations]\label{Lemma Estimates of perturbations 2}
	For any $\rho \in(1, \infty), \gamma \in[1, \infty]$ and every integer $0 \leq N \leq 5$, the following estimates hold :
	\begin{align}
		&\|\nabla^N w_{q+1}^{(p)}\|_{L_t^\gamma L_x^\rho} \lesssim  (1+\|\Ru_{q}\|_{C_{t}L_{x}^{1}}^{N+2})\ell^{-7}r_{\perp}^{\frac{2}{\rho}-1}\laq^{N}\tau^{\frac{1}{2}-\frac{1}{\gamma}},\label{5-19} \\
		&\|\nabla^N w_{q+1}^{(c)}\|_{L_t^\gamma L_x^\rho}
		\lesssim 	(1+\|\Ru_{q}\|_{C_{t}L_{x}^{1}}^{N+4})\ell^{-13}r_{\perp}^{\frac{2}{\rho}-1}\laq^{N-1}\tau^{\frac{1}{2}-\frac{1}{\gamma}},\label{5-20}\\
		&\|\nabla^N w_{q+1}^{(o)}\|_{L_t^\gamma L_x^\rho} \lesssim (1+\|\Ru_{q}\|_{C_{t}L_{x}^{1}}^{N+5})\ell^{-6N-20}\sigma^{-1}.\label{5-21}
	\end{align}
	 Moreover,
	\begin{align}
		&\|w_{q+1}^{(p)}\|_{C_{t,x}^{N}}
		\lesssim
		(1+\|\Ru_{q}\|_{C_{t}L_{x}^{1}}^{N+2})\lambda_{q+1}^{2\a(N+1)+1},\label{5-22}\\
		&\|w_{q+1}^{(c)}\|_{C_{t,x}^{N}}
		\lesssim
		(1+\|\Ru_{q}\|_{C_{t}L_{x}^{1}}^{N+4})\lambda_{q+1}^{2\a(N+1)},\label{5-23}\\	
		&\|w_{q+1}^{(o)}\|_{C_{t,x}^{N}}
		\lesssim
		(1+\|\Ru_{q}\|_{C_{t}L_{x}^{1}}^{N+6})\lambda_{q+1}^{2\a N+1}.\label{5-24}
	\end{align}
	The above implicit constants are deterministic and  independent of $q$.
\end{lemma}
\textit{Proof.} First, using Lemmas \ref{Lemma temporal building blocks},  \ref{Lemma amplitudes 1}, \ref{Lemma spatial building blocks2} and (\ref{5-10}) we get
\begin{align}\label{5-25}
	\|\nabla^N w_{q+1}^{(p)}\|_{L_t^\gamma L_x^\rho}
	&\lesssim \sum_{k \in \Lambda}   \sum_{N_{1}+N_{2}=N}\|a_{(k)}\|_{C_{t,x}^{N_{1}}}\|\nabla^{N_{2}} W_{(k)}\|_{L_x^\rho}\|g_{(k)}\|_{L_t^\gamma}\notag\\
	%&\lesssim \sum_{N_{1}+N_{2}=N}(1+\|\Ru_{q}\|_{C_{t}L_{x}^{1}}^{N_{1}+2})\ell^{-6N_{1}-7} r_{\perp}^{\frac{2}{\rho}-1}\laq^{N_{2}}\tau^{\frac{1}{2}-\frac{1}{\gamma}}\notag\\
	&\lesssim (1+\|\Ru_{q}\|_{C_{t}L_{x}^{1}}^{N+2})\ell^{-7} r_{\perp}^{\frac{2}{\rho}-1}\laq^{N}\tau^{\frac{1}{2}-\frac{1}{\gamma}}, 
\end{align}
where the last step was due to $\ell^{-6}<\lambda_{q+1}$. This verifies (\ref{5-19}).

Similarly,  we obtain
\begin{align}\label{5-26}
	\|\nabla^N w_{q+1}^{(c)}\|_{L_t^\gamma L_x^\rho} 
	&\lesssim
	\sum_{k \in \Lambda}   \big(\sum_{N_{1}+N_{2}=N}\| a_{(k)}\|_{C_{t,x}^{N_{1}+1}}
	\| \nabla^{N_{2}}W_{(k)}^{c}\|_{W_{x}^{1,\rho}}
	+\| a_{(k)}\|_{C_{t,x}^{N_{1}+2}}
	\|\nabla^{N_{2}}W_{(k)}^{c}\|_{L_{x}^{\rho}}\big)\|g_{(k)}\|_{L_t^\gamma}\notag\\
	%&\lesssim\big(\sum_{N_{1}+N_{2}=N}(1+\|\Ru_{q}\|_{C_{t}L_{x}^{1}}^{N_{1}+3})\ell^{-6(N_{1}+1)-7} r_{\perp}^{\frac{2}{\rho}-1}\laq^{N_{2}+1}\laq^{-2}\notag\\
	%&\hspace{2cm} +(1+\|\Ru_{q}\|_{C_{t}L_{x}^{1}}^{N_{1}+4})\ell^{-6(N_{1}+2)-7}r_{\perp}^{\frac{2}{\rho}-1}\laq^{N_{2}}\laq^{-2}\big)\tau^{\frac{1}{2}-\frac{1}{\gamma}}\notag\\
	%&\lesssim (1+\|\Ru_{q}\|_{C_{t}L_{x}^{1}}^{N+4})(\ell^{-13} r_{\perp}^{\frac{2}{\rho}-1}\laq^{N-1}+\ell^{-19} r_{\perp}^{\frac{2}{\rho}-1}\laq^{N-2})\tau^{\frac{1}{2}-\frac{1}{\gamma}}\notag\\
	&\lesssim
	(1+\|\Ru_{q}\|_{C_{t}L_{x}^{1}}^{N+4})\ell^{-13} r_{\perp}^{\frac{2}{\rho}-1}\laq^{N-1}\tau^{\frac{1}{2}-\frac{1}{\gamma}},
\end{align}
 which verifies (\ref{5-20}).

For the temporal correctors, using Lemmas \ref{Lemma temporal building blocks}, \ref{Lemma amplitudes 1}, \ref{Lemma spatial building blocks2} and (\ref{5-14}) we have
\begin{align}\label{5-27}
	\|\nabla^N w_{q+1}^{(o)}\|_{L_t^\gamma L_x^\rho}
	%&\lesssim \sigma^{-1} \sum_{k \in \Lambda} \|\nabla^N\mathbb{P}_H \mathbb{P}_{\neq 0}\big(h_{(k)} \fint_{\mathbb{T}^3} W_{(k)} \otimes W_{(k)} \mathrm{d} x (\nabla a_{(k)}^2)\big) \|_{L_t^\gamma L_x^\rho}\notag\\
	&\lesssim
	\sigma^{-1} \sum_{k \in \Lambda} \|h_{(k)}\|_{C_{t}}\|a_{(k)}^2\|_{C_{t,x}^{N+1}}
	\lesssim
	(1+\|\Ru_{q}\|_{C_{t}L_{x}^{1}}^{N+5})\ell^{-6N-20}\sigma^{-1}.
\end{align}
Thus, we obtain (\ref{5-21}).

Next, we continue to verify the $ C_{t, x}^{N}$-estimates. Applying Lemmas \ref{Lemma temporal building blocks}, \ref{Lemma amplitudes 1}, \ref{Lemma spatial building blocks2}, \eqref{2-22}, 
 (\ref{5-1}) and (\ref{5-10}) we get
\begin{align}\label{5-28}
	\| w_{q+1}^{(p)}\|_{C_{t, x}^{N}}
	&\lesssim \sum_{k \in \Lambda}\|a_{(k)}\|_{C_{t, x}^{N}}	\sum_{N_{1}+N_{2}\leq N}\|\p_{t}^{N_{1}}g_{(k)}\|_{L_t^{\infty}}  \|\nabla^{N_{2}}W_{(k)}\|_{L_{x}^{\infty}}\notag\\
	%&\lesssim	(1+\|\Ru_{q}\|_{C_{t}L_{x}^{1}}^{N+2})\ell^{-6N-7}\sum_{N_{1}+N_{2}\leq N}\sigma^{N_{1}}\tau^{N_{1}+\frac{1}{2}}  r_{\perp}^{-1}\laq^{N_{2}}\notag\\
	&\lesssim (1+\|\Ru_{q}\|_{C_{t}L_{x}^{1}}^{N+2})\ell^{-6N-7}\sigma^{N}\tau^{N+\frac{1}{2}} r_{\perp}^{-1}\notag\\
	&\lesssim
	(1+\|\Ru_{q}\|_{C_{t}L_{x}^{1}}^{N+2})\lambda^{2\a(N+1)+1}.
\end{align}

 By Lemmas  \ref{Lemma temporal building blocks}, \ref{Lemma amplitudes 1},  \ref{Lemma spatial building blocks2},  \eqref{2-22}, (\ref{5-1}) and (\ref{5-11}), we obtain
\begin{align}\label{5-29}
	\| w_{q+1}^{(c)}\|_{C_{t, x}^{N}}
	&\lesssim
	\sum_{k \in \Lambda}
	\sum_{N_{1}+N_{2}\leq N}\|\p_{t}^{N_{1}}g_{(k)}\|_{L_t^{\infty}}(
	\| a_{(k)}\|_{C_{t, x}^{N+1}}\|\nabla^{N_{2}+1} W_{(k)}^{c}\|_{L_{x}^{\infty}}
	+\| a_{(k)}\|_{C_{t, x}^{N+2}}\|\nabla^{N_{2}}W_{(k)}^{c}\|_{L_{x}^{\infty}})\notag\\
	%&\lesssim\sum_{N_{1}+N_{2}\leq N}\sigma^{N_{1}}\tau^{N_{1}+\frac{1}{2}}\big( (1+\|\Ru_{q}\|_{C_{t}L_{x}^{1}}^{N+3})\ell^{-6(N+1)-7}r_{\perp}^{-1}\laq^{N_{2}+1}\lambda^{-2}\notag\\
	%&\hspace{4cm}+(1+\|\Ru_{q}\|_{C_{t}L_{x}^{1}}^{N+4})\ell^{-6(N+2)-7}r_{\perp}^{-1}\laq^{N_{2}}\lambda^{-2}\big)\notag\\
	&\lesssim (1+\|\Ru_{q}\|_{C_{t}L_{x}^{1}}^{N+4})\ell^{-6N-13} \sigma^{N}\tau^{N+\frac{1}{2}}r_{\perp}^{-1}\lambda_{q+1}^{-1}	\notag \\
	&\lesssim
	(1+\|\Ru_{q}\|_{C_{t}L_{x}^{1}}^{N+4})\lambda_{q+1}^{2\a(N+1)}.
\end{align}

Regarding the temporal corrector, one has similar estimate as in (\ref{3-71}):
\begin{align}\label{5-30}
\|w_{q+1}^{(o)}\|_{C_{t, x}^{N}}
&\lesssim		
\sigma^{-1} \sum_{k \in \Lambda} \sum_{N_{1}+N_{2}\leq N} \|h_{(k)}\|_{C_{t}^{N_{1}}}\|\nabla (a_{(k)}^2)\|_{C_{t,x}^{N_{2}+1}},
\end{align}
which, via Lemmas  \ref{Lemma temporal building blocks}, \ref{Lemma amplitudes 1} and (\ref{5-1}), yields that
\begin{align}\label{5-31}
	\|w_{q+1}^{(o)}\|_{C_{t, x}^{N}}
	%&\lesssim (1+\|\Ru_{q}\|_{C_{t}L_{x}^{1}}^{N+6})\ell^{-26}\sigma^{N-1}\tau^{N}\notag\\
	&\lesssim
	(1+\|\Ru_{q}\|_{C_{t}L_{x}^{1}}^{N+6})\lambda_{q+1}^{2\a N+1}.
\end{align}
Therefore, the proof of Lemma \ref{Lemma Estimates of perturbations 2} is complete.\hfill$\square$

\subsection{Verification of inductive estimates for velocity fields} 
We shall verify the inductive estimates  (\ref {2-14}) and (\ref {2-17})-(\ref {2-19}) at level $q+1$ for the velocity fields.

By (\ref{5-16}), (\ref{5-22})-(\ref{5-24}), for $0\leq N\leq 4$, we obtain
\begin{align}\label{5-32}
	\|w_{q+1}\|_{C_{t,x}^{N}}	
	&\lesssim
	\|\Theta_{q+1}\|_{C_{t}^{N}}\|w_{q+1}^{(p)+(c)}\|_{C_{t,x}^{N}}	+	\|\Theta_{q+1}\|^{2}_{C_{t}^{N}}\|w_{q+1}^{(o)}\|_{C_{t,x}^{N}}	\notag\\
%	&\lesssim(1+\|\Ru_{q}\|_{C_{t}L_{x}^{1}}^{N+2})\lambda_{q+1}^{2\a(N+1)+1}\varsigma_{q}^{-N}+(1+\|\Ru_{q}\|_{C_{t}L_{x}^{1}}^{N+4})\lambda_{q+1}^{2\a(N+1)+1}\varsigma_{q}^{-N}\notag\\
	%&\quad+(1+\|\Ru_{q}\|_{C_{t}L_{x}^{1}}^{N+6})\lambda_{q+1}^{2\a N+1}\varsigma_{q}^{-N}\notag\\
	&\lesssim
	(1+\|\Ru_{q}\|_{C_{t}L_{x}^{1}}^{N+6})\lambda_{q+1}^{2\a(N+1)+2}.
\end{align}
Thus,  by (\ref{2-14}) and (\ref{5-32}), we have
\begin{align}\label{5-33}
	\|v_{q+1}\|_{\Lo^{m}_{\om}C_{t,x}^{N}}
	&\lesssim
	\|v_{\ell}\|_{\Lo^{m}_{\om}C_{t,x}^{N}}+\|w_{q+1}\|_{\Lo^{m}_{\om}C_{t,x}^{N}}
	%&\lesssim \|v_{q}\|_{\Lo^{m}_{\om}C_{t,x}^{N}}+(1+\|\Ru_{q}\|_{\Lo^{(N+6)m}_{\om}C_{t}L_{x}^{1}}^{N+6})\lambda_{q+1}^{2\a(N+1)+2}	\notag\\
	%&\lesssim \la^{{2\a\(N+1\)+5}}\( 8(N+10)mL^{2}50^{q-1}\)^{\(N+10\)50^{q-1}}\notag\\
	%&\quad+\(\lambda_{q}^{(4\a+12)}\( 8(N+6)mL^{2}50^{q}\)^{50^{q}}\)^{N+6}\lambda_{q+1}^{2\a(N+1)+2}\notag\\
	\leq
	\(8(N+10)mL^{2}50^{q}\)^{\(N+10\)50^{q}}\laq^{2\a(N+1)+5},
\end{align}
where the last step was due to \eqref{2-7}. 
This verifies  (\ref{2-14}) at level $q+1$.

Regarding the  decay estimate (\ref{2-17}), we  apply the $L^{p}-$decorrelation in Lemma \ref{Lemma Decorrelation1} to derive
\begin{align}\label{5-34}
	\|w_{q+1}^{(p)}\|_{L_{t}^{2}L_{x}^{2}}
	&\lesssim
	\sum_{k \in \Lambda }
	(\|a_{(k)}\|_{L_{t}^2L_{x}^2}\|g_{(k)}\|_{L_t^2}\|\phi_{(k)}\|_{L_x^2}+\sigma^{-\frac{1}{2}}\|a_{(k)}\|_{C_{t, x}^1}\|g_{(k)}\|_{L_t^2}\|\phi_{(k)}\|_{L_x^2}).
\end{align}
By  Lemmas \ref{Lemma amplitudes 1} and  \ref{Lemma spatial building blocks2}, one gets 
\begin{align}\label{5-35}
	\|w_{q+1}^{(p)}\|_{L_{t}^{2}L_{x}^{2}}
	\lesssim
	\dqq^{\frac{1}{2}}+\|\Ru_{q}\|_{C_{t}L_{x}^{1}}^{\frac{1}{2}}+\sigma^{-\frac{1}{2}}(1+\|\Ru_{q}\|_{C_{t}L_{x}^{1}}^{3})\ell^{-13}.
\end{align}
Taking  into account  (\ref{5-1}) and (\ref{5-19})-(\ref{5-21}) we have
\begin{align}\label{5-36}
	\|w_{q+1}\|_{L_{t}^{2}L_{x}^{2}}
		&\lesssim
	\|\Theta_{q+1}\|_{C_{t}}\|w_{q+1}^{(p)+(c)}\|_{L_{t}^{2}L_{x}^{2}}+	\|\Theta_{q+1}\|^{2}_{C_{t}}\|w_{q+1}^{(o)}\|_{L_{t}^{2}L_{x}^{2}}\notag\\
	%&\lesssim\dqq^{\frac{1}{2}}+\|\Ru_{q}\|_{C_{t}L_{x}^{1}}^{\frac{1}{2}}+\sigma^{-\frac{1}{2}}(1+\|\Ru_{q}\|_{C_{t}L_{x}^{1}}^{3})\ell^{-13}+(1+\|\Ru_{q}\|_{C_{t}L_{x}^{1}}^{4})\ell^{-13}\lambda_{q+1}^{-1}\notag\\
	%&\quad+(1+\|\Ru_{q}\|_{C_{t}L_{x}^{1}}^{5})\ell^{-20}\sigma^{-1}	\notag\\
	&\lesssim
	\dqq^{\frac{1}{2}}+\|\Ru_{q}\|_{C_{t}L_{x}^{1}}^{\frac{1}{2}}++(1+\|\Ru_{q}\|_{C_{t}L_{x}^{1}}^{5})\ell^{-14}\sigma^{-1}.
\end{align}
Then, by   (\ref{2-7}), (\ref{2-15}), (\ref{2-16}), (\ref{2-22}) and  (\ref{3-49}), we obtain
\begin{align}\label{5-37}
	\|w_{q+1}\|_{\Lo^{2r}_{\om}L_{t}^{2}L_{x}^{2}}
	&\lesssim
	\dqq^{\frac{1}{2}}+\|\Ru_{q}\|_{\Lo^{r}_{\om}C_{t}L_{x}^{1}}^{\frac{1}{2}}+(1+\|\Ru_{q}\|_{\Lo^{10r}_{\om}C_{t}L_{x}^{1}}^{5})\ell^{-14}\sigma^{-1}
	%&\lesssim \dqq^{\frac{1}{2}}+\big(\lambda_{q}^{(4\a+12)}(8\cdot10rL^{2}50^{q})^{50^{q}}\big)^{5}\ell^{-14}\laq^{-2\varepsilon}\notag\\
	%&\lesssim\dqq^{\frac{1}{2}}+\lambda_{q+1}^{-\varepsilon}
	\lesssim
	\dqq^{\frac{1}{2}}.
\end{align}
Thus, by (\ref{2-14}), we obtain
\begin{align}\label{5-38}
	\|v_{q+1}-v_{q}\|_{\Lo^{2r}_{\om}L_{t}^{2}L_{x}^{2}}
	&\lesssim
	\|v_\ell-v_{q}\|_{\Lo^{2r}_{\om}L_{t}^{2}L_{x}^{2}}+	\|w_{q+1}\|_{\Lo^{2r}_{\om}L_{t}^{2}L_{x}^{2}}
 \lesssim \ell\|v_{q}\|_{\Lo^{2r}_{\om}C_{t,x}^{1}}+\delta_{q+1}^{\frac{1}{2}}
	%&\lesssim\ell\lambda^{4\a+5}(8\cdot22rL^{2}50^{q-1})^{50^{q-1}\cdot11}+\delta_{q+1}^{\frac{1}{2}}
	\lesssim\delta_{q+1}^{\frac{1}{2}},
\end{align}
which yields (\ref{2-17}) at level $q+1$.

Concerning the $\Lo^{r}_{\om}L_{t}^{1}L_{x}^{2}$-estimate (\ref{2-18}) at $q+1$,  by  (\ref{3-49}) and (\ref{5-19})-(\ref{5-21}) in Lemma \ref{Lemma Estimates of perturbations 2},  we get
\begin{align}\label{5-39}
	\|w_{q+1}\|_{L_{t}^{1}L_{x}^{2}}
 	&\lesssim
	\|\Theta_{q+1}\|_{C_{t}}\|w_{q+1}^{(p)+(c)}\|_{L_{t}^{1}L_{x}^{2}}+	\|\Theta_{q+1}\|^{2}_{C_{t}}\|w_{q+1}^{(o)}\|_{L_{t}^{1}L_{x}^{2}}\notag\\
	%&\lesssim (1+\|\Ru_{q}\|_{C_{t}L_{x}^{1}}^{2})\ell^{-7}\tau^{-\frac{1}{2}}+(1+\|\Ru_{q}\|_{C_{t}L_{x}^{1}}^{4})\ell^{-7}r_{\perp}\lambda_{q+1}^{-1}\tau^{-\frac{1}{2}}+(1+\|\Ru_{q}\|_{C_{t}L_{x}^{1}}^{5})\ell^{-20}\sigma^{-1}\notag\\
	&\lesssim
	(1+\|\Ru_{q}\|_{C_{t}L_{x}^{1}}^{5})\ell^{-20}\sigma^{-1},
\end{align}
which along with (\ref{2-7}), (\ref{2-15}), (\ref{2-16})  and (\ref{2-22})  implies
\begin{align}\label{5-40}
	\|v_{q+1}-v_q\|_{\Lo^{r}_{\om}L_{t}^{1}L_{x}^{2}}
	& \lesssim
	\|v_\ell-v_q\|_{\Lo^{r}_{\om}L_{t}^{1}L_{x}^{2}}+\|w_{q+1}\|_{\Lo^{r}_{\om}L_{t}^{1}L_{x}^{2}}\notag \\
	& \lesssim
	\ell\|v_q\|_{\Lo^{r}_{\om}C_{t,x}^{1}}+	(1+\|\Ru_{q}\|_{\Lo^{5r}_{\om}C_{t}L_{x}^{1}}^{5}\ell^{-20}\sigma^{-1} 
	%& \lesssim\ell\la^{4\a+5}(8\cdot 11rL^{2}50^{q-1})^{11\cdot50^{q-1}} +\big(\lambda_{q}^{(4\a+12)}(8\cdot5rL^{2}50^{q})^{50^{q}}\big)^{5}\ell^{-20}\laq^{-2\varepsilon}
	\leq
	\dqqq^{\frac{1}{2}}.
\end{align}
This verifies  estimate (\ref{2-18}) at level $q+1$.

At last, for estimate (\ref{2-19}), by Lemma \ref{Lemma Estimates of perturbations 2} and the  interpolation, we get
\begin{align}\label{5-41}
	\|w_{q+1}\|_{L_{t}^{\ga}W_{x}^{s,p}}
		&\lesssim
	\|\Theta_{q+1}\|_{C_{t}}\|w_{q+1}^{(p)+(c)}\|_{L_{t}^{\ga}W_{x}^{s,p}}+	\|\Theta_{q+1}\|^{2}_{C_{t}}\|w_{q+1}^{(o)}\|_{L_{t}^{\ga}W_{x}^{s,p}}\notag\\
	%&\lesssim (1+\|\Ru_{q}\|_{C_{t}L_{x}^{1}}^{s+2})\ell^{-7}r_{\perp}^{\frac{2}{p}-1}\laq^{s}\tau^{\frac{1}{2}-\frac{1}{\ga}}+(1+\|\Ru_{q}\|_{C_{t}L_{x}^{1}}^{s+4})\ell^{-13}r_{\perp}^{\frac{2}{p}-1}\laq^{s-1}\tau^{\frac{1}{2}-\frac{1}{\ga}}\notag\\
	%&\quad+(1+\|\Ru_{q}\|_{C_{t}L_{x}^{1}}^{s+6})\ell^{-6s-20}\sigma^{-1}\notag\\
	&\lesssim
	(1+\|\Ru_{q}\|_{C_{t}L_{x}^{1}}^{s+6})\ell^{-6s-20}\sigma^{-1},
\end{align}
where the last step was due to  (\ref{2-9})  and the fact that
\begin{align}\label{5-42}
	\ell^{-7}r_{\perp}^{\frac{2}{p}-1}\laq^{s}\tau^{\frac{1}{2}-\frac{1}{\ga}}
	&=\lambda_{q}^{7\times80}\lambda_{q+1}^{s-\frac{2\a}{\gamma}+(8-\frac{16}{p})\varepsilon+2\a-1}
	\leq \lambda_{q+1}^{-10\varepsilon}\leq \ell^{-6s-20}\sigma^{-1}.
\end{align}
Then, using the embedding $C_{x}^{3}(\T^{3}) \hookrightarrow W_{x}^{s,p}(\T^{3})$ for  $(s, \gamma, p) \in \mathcal{S}_{2}$, (\ref{2-7}), (\ref{2-15}), (\ref{2-16})  and (\ref{2-22}), we get
\begin{align}\label{5-43}
	\|v_{q+1}-v_q\|_{\Lo^{r}_{\om}L_{t}^{\ga}W_{x}^{s,p}}
	& \lesssim
	\|v_\ell-v_q\|_{\Lo^{r}_{\om}L_{t}^{\ga}W_{x}^{s,p}}+\|w_{q+1}\|_{\Lo^{r}_{\om}L_{t}^{\ga}W_{x}^{s,p}} \notag\\
	& \lesssim
	\|v_\ell-v_q\|_{\Lo^{r}_{\om}L_{t}^{\ga}C_{x}^{3}}+ (1+	\|\Ru_{q}\|_{\Lo^{(s+6)r}_{\om}C_{t}L_{x}^{1}}^{s+6})\ell^{-6s-20}\sigma^{-1} \notag\\
	&\lesssim
	\ell\|v_q\|_{\Lo^{r}_{\om}C_{t,x}^{4}}+	 (1+	\|\Ru_{q}\|_{\Lo^{9r}_{\om}C_{t}L_{x}^{1}}^{9})\ell^{-6s-20}\sigma^{-1}\notag\\
	%&\lesssim\ell\la^{10\a+5}(8\cdot14rL^{2}50^{q-1})^{14\cdot50^{q-1}}+\big(\lambda_{q}^{(4\a+12)}(8\cdot9rL^{2}50^{q})^{50^{q}}\big)^{9}\ell^{-38}\laq^{-2\varepsilon} \notag\\
	%\lesssim\lambda_{q}^{-54}+\lambda_{q+1}^{-\varepsilon}
	&\lesssim
	\dqqq^{\frac{1}{2}}.
\end{align}
Therefore,  estimate (\ref{2-19}) is verified at level $q+1$.

\section{Reynolds stresses in the regime \texorpdfstring{$\mathcal{S}_{2}$}{S2}}\label{Sec-Reynolds 2}

We proceed to treat the Reynolds stress 
 at level $q+1$ in the supercritical regime $\mathcal{S}_{2}$.

\subsection{Reynolds stress} 
We derive  from system (\ref{2-12}) at level $q+1$ that
\begin{align}\label{6-1}
	\div\Ru_{q+1}-\nabla P_{q+1}
	&=\underbrace{\p_{t}\omw^{(p)+(c)}_{q+1}+\nu(-\Delta)^{\a}w_{q+1}-(z_{q+1}-z_{\ell})+\div\big((v_{\ell}+z_{\ell})\otimes w_{q+1}+w_{q+1}\otimes ( v_{\ell}+z_{\ell})\big)}_{\div\Ru_{lin}+\nabla P_{lin}}\notag\\  %lin
	&\quad+
	\underbrace{\div(\omw^{(p)}_{q+1}\otimes                       \omw^{(c)+(o)}_{q+1}
	+\omw^{(c)+(o)}_{q+1}\otimes w_{q+1})}_{\div\Ru_{corr}+\nabla P_{corr}}+\div(\Ru_{com1})\notag\\ %corr
	&\quad
	+\underbrace{\div( v_{q+1}\otimes z_{q+1}+z_{q+1}\otimes v_{q+1}
	-v_{q+1}\otimes z_{\ell}-z_{\ell}\otimes v_{q+1}+z_{q+1}\otimes z_{q+1}
	-z_{\ell}\otimes z_{\ell})}_{\div\Ru_{com2}+\nabla P_{com2}}\notag\\
	&\quad+\underbrace{\div(\omw^{(p)}_{q+1}\otimes\omw^{(p)}_{q+1}+\Ru_{\ell})+\p_{t}\omw^{(o)}_{q+1}}_{\div\Ru_{osc}+\nabla P_{osc}}-\nabla P_{\ell},
\end{align}
where $\Ru_{com1}$ is given by \eqref{2-27}.
Using the inverse divergence operator  $\mathcal{R}$ we define
\begin{align}
	&\Ru_{lin}
	:=\mathcal{R}\p_{t}\omw_{q+1}^{(p)+(c)}+\mathcal{R}\(\nu(-\Delta)^{\a}w_{q+1}\)+\mathcal{R}(z_{\ell}-z_{q+1})+(v_{\ell}+z_{\ell})\mathring{\otimes}w_{q+1}+w_{q+1}\mathring{\otimes}( v_{\ell}+z_{\ell}),\label{6-2}\\
	&\Ru_{corr}:=
	\omw_{q+1}^{(p)}\mathring{\otimes} \omw_{q+1}^{(c)+(o)}	+\omw_{q+1}^{(c)+(o)}\mathring{\otimes}w_{q+1},\label{6-3}\\
	&\Ru_{com2}
	:= v_{q+1}\mathring{\otimes}z_{q+1}+z_{q+1}\mathring{\otimes} v_{q+1}-v_{q+1}\mathring{\otimes}z_{\ell}
	-z_{\ell}\mathring{\otimes} v_{q+1}
	+z_{q+1}\mathring{\otimes} z_{q+1} -z_{\ell}\mathring{\otimes}z_{\ell}.\label{6-4}
\end{align}
While for the remaining oscillation error, by (\ref{3-43-1}), 
\begin{align}\label{6-5}
	\div(\omw^{(p)}_{q+1}\otimes\omw^{(p)}_{q+1}+\Ru_{\ell})
	&=(1-\Theta_{q+1}^{2})\div \Ru_{\ell}+\Theta_{q+1}^{2}\div\big(\sum_{k\in\Lambda} a_{(k)}^{2}(g_{(k)}^{2}-1) \fint_{\T^3}W_{(k)}\otimes W_{(k)} \mathrm{d}x\big)\notag\\
	&\quad+\Theta_{q+1}^{2}\div(\varrho \mathrm{Id})+\Theta_{q+1}^{2}\div\big(\sum_{k\in\Lambda}a_{(k)}^{2}g_{(k)}^{2}\mathbb{P}_{\neq 0}(W_{(k)}\otimes W_{(k)})\big).
\end{align}
By  (\ref{3-48-2}), \eqref{5-14} and (\ref{6-5}), we get
\begin{align}\label{6-6}
	&\quad\div(\omw^{(p)}_{q+1}\otimes\omw^{(p)}_{q+1}+\Ru_{\ell})+\p_{t}\omw^{(o)}_{q+1}\notag\\
	&=(1-\Theta_{q+1}^{2})\div \Ru_{\ell}+\p_{t}\Theta_{q+1}^{2}w^{(o)}_{q+1}+\Theta_{q+1}^{2}\div(\varrho \mathrm{Id})+\Theta_{q+1}^{2}\sum_{k \in \Lambda}\mathbb{P}_{\neq 0}\big(g_{(k)}^{2}\mathbb{P}_{\neq 0}(W_{(k)}\otimes W_{(k)})\nabla (a_{(k)}^{2})\big)	\notag\\
	&\quad
	-\Theta_{q+1}^{2}\sigma^{-1} \sum_{k \in \Lambda} \mathbb{P}_{\neq 0}\big(h_{(k)}\fint_{\T^3}W_{(k)}\otimes W_{(k)}\mathrm{d}x\p_{t}\nabla (a_{(k)}^{2})\big)\notag\\
	&\quad+(\nabla\Delta^{-1}\div)\sigma^{-1}\mathbb{P}_{\neq 0}\p_{t}\big(h_{(k)} \fint_{\mathbb{T}^3} W_{(k)} \otimes W_{(k)} \mathrm{d} x \nabla(a_{(k)}^2)\big).
\end{align}

Therefore, we can choose the oscillation error by
\begin{align} \label{6-7}
	\Ru_{osc}:=\Ru_{osc.1}+\Ru_{osc.2}+\Ru_{osc.3},  
\end{align}
where
\begin{align}
	&\Ru_{osc.1}
	:=(1-\Theta_{q+1}^{2}) \Ru_{\ell}+\mathcal{R}(\p_{t}\Theta_{q+1}^{2}w^{(o)}_{q+1}),\label{6-8}\\
	&\Ru_{osc.2}:=\Theta_{q+1}^{2}\sum_{k \in \Lambda}\mathcal{R}\mathbb{P}_{\neq 0}\big(g_{(k)}^{2}\mathbb{P}_{\neq 0}(W_{(k)}\otimes W_{(k)})\nabla (a_{(k)}^{2})\big),\label{6-9}\\
	&\Ru_{osc.3}
	:=-\Theta_{q+1}^{2}\sigma^{-1} \sum_{k \in \Lambda} \mathcal{R}\mathbb{P}_{\neq 0}\big(h_{(k)}\fint_{\T^3}W_{(k)}\otimes W_{(k)}\mathrm{d}x\p_{t}\nabla (a_{(k)}^{2})\big).    \label{6-10}
\end{align}

Finally, we obtain the decomposition of the Reynolds stress at level $ q + 1 $:
$$\Ru_{q+1}:=\Ru_{lin}+\Ru_{corr}+\Ru_{com1}+\Ru_{com2}+\Ru_{osc}.$$
\subsection{Verification of growth estimate}
We aim to verify the inductive growth estimate in \eqref{2-15}.
To this end, let us estimate each term in the above choice of the Reynolds stresses  $ \Ru_{q+1}$.

\medskip 
\paragraph{\bf Linear error.}
By Lemma \ref{Lemma Estimates of perturbations 2}, (\ref{3-49})  and  (\ref{3-50}),
\begin{align}\label{6-12}
	\|\mathcal{R}\p_{t}\omw_{q+1}^{(p)+(c)}\|_{C_{t}L_x^{1}}
	&\lesssim
	\| \Theta_{q+1}\|_{C_{t}^{1}} \|w_{q+1}^{(p)+(c)} \|_{C_{t,x}^{1}}
	%&\lesssim\varsigma_{q}^{-1}\big((1+\|\Ru_{q}\|_{C_{t}L_{x}^{1}}^{3})\lambda_{q+1}^{4\a+1}+(1+\|\Ru_{q}\|_{C_{t}L_{x}^{1}}^{5})\lambda_{q+1}^{4\a+1}\big)\notag\\
	\lesssim(1+\|\Ru_{q}\|_{C_{t}L_{x}^{1}}^{5})\lambda_{q+1}^{4\a+1}\varsigma_{q}^{-1}.
\end{align}
By  (\ref{3-49}), (\ref{5-16}), Lemma \ref{Lemma Estimates of perturbations 2}
and the Sobolev embedding $H^{3}_{x}(\T^{3})\hookrightarrow H^{2\a-1}_{x}(\T^{3})$
with $\a\in[1,2)$, we obtain
\begin{align}\label{6-13}
	\|\mathcal{R}(\nu(-\Delta)^{\a}w_{q+1})\|_{C_{t}L_x^{1}}
	&\lesssim
	\|w_{q+1}\|_{C_{t}H_x^{3}}
	%&\lesssim(1+\|\Ru_{q}\|_{C_{t}L_{x}^{1}}^{5})\ell^{-7}\laq^{3}\tau^{\frac{1}{2}}+(1+\|\Ru_{q}\|_{C_{t}L_{x}^{1}}^{7})\ell^{-7}\laq^{2}\tau^{\frac{1}{2}}\notag\\
	%&\quad+(1+\|\Ru_{q}\|_{C_{t}L_{x}^{1}}^{8})\ell^{-38}\sigma^{-1}\notag\\
	\lesssim	(1+\|\Ru_{q}\|_{C_{t}L_{x}^{1}}^{8})\ell^{-7}\laq^{3}\tau^{\frac{1}{2}}.
\end{align}
Moreover, for the remaining linear errors in  \eqref{6-2}, by Lemma \ref{Lemma Estimates of perturbations 2} and H\"{o}lder's inequality,
\begin{align}\label{6-14}
	&\quad\|\mathcal{R}(z_{\ell}-z_{q+1})\|_{C_{t}L_x^{1}}
	+\|(v_{\ell}+z_{\ell})\mathring{\otimes}w_{q+1}+w_{q+1}\mathring{\otimes}( v_{\ell}+z_{\ell})\|_{C_{t}L_x^{1}}\notag\\
	&\lesssim\|z\|_{C_{t}L_x^{2}}+(\|v_{q}\|_{C_{t,x}}+\|z\|_{C_{t} L_{x}^{2}})\|w_{q+1}\|_{C_{t,x}}\notag\\
	%&\lesssim\|z\|_{C_{t}L_x^{2}}+(\|v_{q}\|_{C_{t,x} }+\|z\|_{C_{t} L_{x}^{2}})(1+\|\Ru_{q}\|_{C_{t}L_{x}^{1}}^{6})(\lambda_{q+1}^{2\a+1}+\lambda_{q+1}^{2\a}+\lambda_{q+1}^{3})\notag\\
	&\lesssim
	(\|v_{q}\|_{C_{t,x} }+\|z\|_{C_{t} L_{x}^{2}})(1+\|\Ru_{q}\|_{C_{t}L_{x}^{1}}^{6})\lambda_{q+1}^{2\a+1}.
\end{align}
Thus, we conclude from the above estimates (\ref{6-12})-(\ref{6-14}) that
\begin{align}\label{6-15}
	\|\Ru_{lin}\|_{C_{t}L_x^{1}}
	%&\lesssim(1+\|\Ru_{q}\|_{C_{t}L_{x}^{1}}^{5})\lambda_{q+1}^{4\a+1}\varsigma_{q}^{-1}+(1+\|\Ru_{q}\|_{C_{t}L_{x}^{1}}^{8})\ell^{-7}\laq^{3}\tau^{\frac{1}{2}}\notag\\
	%&\quad	+(\|v_{q}\|_{C_{t,x} }+\|z\|_{C_{t} L_{x}^{2}})(1+\|\Ru_{q}\|_{C_{t}L_{x}^{1}}^{6})\lambda_{q+1}^{2\a+1}\notag\\
	&\lesssim
	(1+\|\Ru_{q}\|_{C_{t}L_{x}^{1}}^{8})\lambda_{q+1}^{4\a+2}+(\|v_{q}\|_{C_{t,x} }+\|z\|_{C_{t} L_{x}^{2}})(1+\|\Ru_{q}\|_{C_{t}L_{x}^{1}}^{6})\lambda_{q+1}^{2\a+1}.  
\end{align}

\medskip 
\paragraph{\bf Oscillation error.} 
By (\ref{3-49}),  (\ref{6-8}) and Lemma \ref{Lemma Estimates of perturbations 2}, we  get
\begin{align}\label{6-16}
	\| \Ru_{osc.1}\|_{C_{t}L_{x}^{1}}
	&\lesssim
	\|1-\Theta_{q+1}^{2}\|_{C_{t}}\|\Ru_{q}\|_{C_{t}L_{x}^{1}}+\|\Theta_{q+1}^{2}\|_{C_{t}^{1}}\|w^{(o)}_{q+1}\|_{C_{t,x}}
	\notag\\
	%&\lesssim	\|\Ru_{q}\|_{C_{t}L_{x}^{1}}+\varsigma_{q}^{-1}(1+\|\Ru_{q}\|_{C_{t}L_{x}^{1}}^{6})\lambda_{q+1}\notag\\
	&\lesssim
	(1+\|\Ru_{q}\|_{C_{t}L_{x}^{1}}^{6})\lambda_{q+1}\varsigma_{q}^{-1}.
\end{align}

By (\ref{3-49}), (\ref{6-9}) and Lemmas \ref{Lemma temporal building blocks}, \ref{Lemma amplitudes 1} 
and \ref{Lemma spatial building blocks2}, 
\begin{align}\label{6-17}
	\| \Ru_{osc.2}\|_{C_{t}L_{x}^{1}}
	&\lesssim
	\|\Theta_{q+1}^{2}\|_{C_{t}}
	\|\mathcal{R}\mathbb{P}_{\neq 0}(g_{(k)}^{2}\mathbb{P}_{\neq 0}(W_{(k)}\otimes W_{(k)})\nabla (a_{(k)}^{2}))\|_{C_{t}L_{x}^{2}}\notag\\
	&\lesssim
	\sum_{k\in\Lambda}\|a_{(k)}^{2}\|_{C_{t,x}^{1}}\|W_{(k)}\otimes W_{(k)}\|_{C_{t}L_{x}^{2}}
	\|g_{(k)}^{2}\|_{L_{t}^{\infty}}\notag\\
	&\lesssim
	(1+\|\Ru_{q}\|_{C_{t}L_{x}^{1}}^{5})\ell^{-20}r_{\perp}^{-1}\tau.
\end{align}

Regarding the last component, by (\ref{3-49}), (\ref{3-21}), (\ref{6-10})  and Lemma \ref{Lemma amplitudes 1},
\begin{align}  \label{6-19}
\|\Ru_{osc.3}\|_{C_{t}L_{x}^{1}}
	\lesssim
\sigma^{-1}\sum_{k\in\Lambda}\|h_{(k)}\|_{C_{t}}
\|a_{(k)}^{2}\|_{C_{t,x}^{2}}   
	 \lesssim
(1+\|\Ru_{q}\|_{C_{t}L_{x}^{1}}^{6})\ell^{-26}\sigma^{-1}.
\end{align}

Thus, combining {\eqref{6-16}-\eqref{6-19}} altogether,  we conclude that
\begin{align}  \label{6-20}
	\|\Ru_{osc}\|_{C_{t}L_x^{1}}
	%&\lesssim	(1+\|\Ru_{q}\|_{C_{t}L_{x}^{1}}^{6})\lambda_{q+1}\varsigma_{q}^{-1}	+(1+\|\Ru_{q}\|_{C_{t}L_{x}^{1}}^{5})\ell^{-20}r_{\perp}^{-1}\tau	+(1+\|\Ru_{q}\|_{C_{t}L_{x}^{1}}^{6})\ell^{-26}\sigma^{-1}\notag\\
	&\lesssim
	(1+\|\Ru_{q}\|_{C_{t}L_{x}^{1}}^{6})(\lambda_{q+1}\varsigma_{q}^{-1}+\ell^{-20}\lambda_{q+1}^{2\a-1+8\varepsilon}+\ell^{-26}\lambda_{q+1}^{-2\varepsilon})\notag\\
	&\lesssim
	(1+\|\Ru_{q}\|_{C_{t}L_{x}^{1}}^{6})\lambda_{q+1}^{2\a},
\end{align}
where   the last  step was due to (\ref{2-6}), (\ref{2-22}) and (\ref{5-1}).
\medskip 
\paragraph{\bf Corrector error.}
By (\ref{3-49}), Lemma \ref{Lemma Estimates of perturbations 2}  and (\ref{6-3}),  we have
\begin{align}\label{6-21}
	\|\Ru_{corr}\|_{C_{t}L_{x}^{1}}
	&\lesssim
	\|\omw_{q+1}^{(c)+(o)} \|_{C_{t,x}}(\|\omw_{q+1}^{(p)}\|_{C_{t,x}}+\|w_{q+1}\|_{C_{t,x}})
	%&\lesssim(1+\|\Ru_{q}\|_{C_{t}L_{x}^{1}}^{6})(\lambda_{q+1}^{2\a}+\lambda_{q+1})\times(1+\|\Ru_{q}\|_{C_{t}L_{x}^{1}}^{6})(\lambda_{q+1}^{2\a+1}+\lambda_{q+1}^{2\a}+\lambda_{q+1})
 \notag\\
	&\lesssim	(1+\|\Ru_{q}\|_{C_{t}L_{x}^{1}}^{12})\lambda_{q+1}^{4\a+1}.
\end{align}

\medskip 
\paragraph{\bf Commutator errors.} At last, for the commutator terms, by
(\ref{2-27}) and (\ref{6-4}), we have
\begin{align}\label{6-22}
	\|\Ru_{com1}\|_{C_{t}L_x^{1}}
	&\lesssim \|v_{q}\|_{C_{t,x}}^{2}+\|z\|_{C_{t}L_{x}^{2}}^{2},
\end{align}
and
\begin{align}\label{6-23}
	\|\Ru_{com2}\|_{C_{t}L_x^{1}}
	&\lesssim
	\|v_{q+1}\|_{C_{t,x}}^{2}+\|z\|_{C_{t} L_{x}^{2}}^{2}.
\end{align}

\medskip 
\paragraph{\bf Verification of inductive estimate  (\ref{2-15}).} 
Now, summing up all the above estimates (\ref{6-15}),  and  (\ref{6-20})-(\ref{6-23}) we conclude that
\begin{align}\label{6-24}
	\|\Ru_{q+1}\|_{C_{t}L_x^{1}}
	&\lesssim
	(1+\|\Ru_{q}\|_{C_{t}L_{x}^{1}}^{8})\lambda_{q+1}^{4\a+2}+(\|v_{q}\|_{C_{t,x} }+\|z\|_{C_{t} L_{x}^{2}})(1+\|\Ru_{q}\|_{C_{t}L_{x}^{1}}^{6})\lambda_{q+1}^{2\a+1}+(1+\|\Ru_{q}\|_{C_{t}L_{x}^{1}}^{6})\lambda_{q+1}^{2\a}\notag\\           %lin, osc
	&\quad+(1+\|\Ru_{q}\|_{C_{t}L_{x}^{1}}^{12})\lambda_{q+1}^{4\a+1}
	+\|v_{q}\|_{C_{t,x}}^{2}+\|z\|_{C_{t}L_{x}^{2}}^{2}  % corr com1
	+\|v_{q+1}\|_{C_{t,x}}^{2} \notag\\%com2
	&\lesssim
	(\|v_{q}\|_{C_{t,x} }+\|z\|_{C_{t} L_{x}^{2}})(1+\|\Ru_{q}\|_{C_{t}L_{x}^{1}}^{6})\lambda_{q+1}^{2\a+1}+(1+\|\Ru_{q}\|_{C_{t}L_{x}^{1}}^{12})\lambda_{q+1}^{4\a+2}\notag\\           %lin, osc
	&\quad
	+\|v_{q}\|_{C_{t,x}}^{2}+\|z\|_{C_{t}L_{x}^{2}}^{2}  % corr com1
	+\|v_{q+1}\|_{C_{t,x}}^{2}.
\end{align}
Taking the $m$-th moment in (\ref{6-24}) and using (\ref{stochastic evolution}),  (\ref{2-14}), (\ref{2-15}), (\ref{5-33}) and H\"{o}lder's inequality, we obtain
\begin{align*}
	\|\Ru_{q+1}\|_{\Lo^{m}_{\om}C_{t}L_x^{1}}
	&\lesssim
	(\|v_{q}\|_{\Lo^{2m}_{\om}C_{t,x}}+\|z\|_{\Lo^{2m}_{\om}C_{t}L_{x}^{2}})(1+\|\Ru_{q}\|_{\Lo^{12m}_{\om}C_{t}L_{x}^{1}}^{6})\lambda_{q+1}^{2\a+1}+	(1+\|\Ru_{q}\|_{\Lo^{12m}_{\om}C_{t}L_{x}^{1}}^{12})\lambda_{q+1}^{4\a+2}\notag\\
	&\quad+\|v_{q}\|_{\Lo^{2m}_{\om}C_{t,x}}^{2}
	%linÏî osc corr com1
	+\|z\|_{\Lo^{2m}(\om;C_{t}L_{x}^{2})}^{2}  %com1Ïî
	+\|v_{q+1}\|_{\Lo^{2m}(\om;C_{t,x})}^{2} %com2
    \notag\\
	%&\lesssim\big(\la^{(2\a+5)}(8\cdot 20mL^{2}50^{q-1})^{10\cdot50^{q-1}}+M+(2m-1)^{\frac{1}{2}}L\big) \big(\lambda_{q}^{(4\a+12)}(8\cdot 12mL^{2}50^{q})^{50^{q}}\big)^{6}\lambda_{q+1}^{2\a+1}\notag\\
	%&\quad+\big(\lambda_{q}^{(4\a+12)}(8\cdot 12mL^{2}50^{q})^{50^{q}}\big)^{12}\lambda_{q+1}^{4\a+2}+\big(\la^{(2\a+5)}(8\cdot 20mL^{2}50^{q-1})^{10\cdot50^{q-1}}\big)^{2}\notag\\
	%&\quad+M^{2}+(2m-1)L^{2}+\big(\laq^{(2\a+5)}(8\cdot 20mL^{2}50^{q})^{10\cdot50^{q}}\big)^{2}\notag\\
	&\leq
	\lambda_{q+1}^{(4\a+12)}(8mL^{2}50^{q+1})^{50^{q+1}},
\end{align*}
which verifies the inductive estimate (\ref{2-15}) at level $q+1$.

\subsection{Verification of decay estimate \eqref{2-16}: away from the initial time}
Below we  prove the decay estimate \eqref{2-14} at level $q+1$ in the difficult regime $ (\frac{\varsigma_{q}}{2}, T]$ away from the initial.

Let
\begin{equation}\label{6-26}
	\begin{aligned}
		\rho:=\frac{2\a-1+16\varepsilon}{2\a-1+14\varepsilon} \in(1,2)
\end{aligned}\end{equation}
with $\varepsilon$ as in \eqref{2-9}. Then, one has
\begin{equation}\label{6-27}
	\begin{aligned}
		(1-\a-8\varepsilon)(\frac{1}{\rho}-1)= 1-\a-6\varepsilon,
\end{aligned}\end{equation}
and
\begin{equation}\label{6-28}
	\begin{aligned}
		r_{\perp}^{\frac{2}{\rho}-1}=\lambda_{q+1}^{1-\a-6\varepsilon}.
\end{aligned}\end{equation}

We shall treat each component of the Reynolds stress in Subsection \ref{subsec 6.3.1}-\ref{subsec 6.3.4} 
below.

\subsubsection{Linear error:}\label{subsec 6.3.1}  

\medskip 
\paragraph{\bf Control of $\p_{t}\omw_{q+1}^{(p)+(c)}$.}
Let us first consider the linear error and estimate
\begin{align*}
	\|\mathcal{R}\p_{t}\omw_{q+1}^{(p)+(c)}\|_{L_{(\frac{\varsigma_{q}}{2}, T]}^{1}L_x^{\rho}} 
	&\leq
	\|\mathcal{R}(\p_t \Theta_{q+1}) w_{q+1}^{(p)+(c)}\|_{L_{t}^{1}L_x^{\rho}}
	+\|\mathcal{R} \Theta_{q+1} \partial_t w_{q+1}^{(p)+(c)} \|_{L_{t}^{1}L_x^{\rho}} \notag\\
	&=:K_{1}+K_{2}.
\end{align*}
By ({\ref{3-49}}), (\ref{5-19}) and (\ref{5-20}),
\begin{align*}
	 K_{1}&\lesssim
	\|\Theta_{q+1}\|_{C_{t}^{1}} \|w_{q+1}^{(p)+(c)}\|_{L_{t}^{1}L_x^{\rho}}
	%&\lesssim\varsigma_{q}^{-1}(1+\|\Ru_{q}\|_{C_{t}L_{x}^{1}}^{2})\ell^{-7}r_{\perp}^{\frac{2}{\rho}-1}\tau^{-\frac{1}{2}}+\varsigma_{q}^{-1}(1+\|\Ru_{q}\|_{C_{t}L_{x}^{1}}^{4})\ell^{-13}r_{\perp}^{\frac{2}{\rho}-1}\lambda_{q+1}^{-1}\tau^{-\frac{1}{2}}\notag\\
	\lesssim(1+\|\Ru_{q}\|_{C_{t}L_{x}^{1}}^{4})\ell^{-7}\varsigma_{q}^{-1}r_{\perp}^{\frac{2}{\rho}-1}\tau^{-\frac{1}{2}}.
\end{align*}

Moreover, using Lemmas \textcolor{orange}{ \ref{Lemma temporal building blocks}, \ref{Lemma amplitudes 1}, \ref{Lemma spatial building blocks2}}, (\ref{5-12}) and ({\ref{3-49}}) we obtain
\begin{align*}
	\quad K_{2}
	&\lesssim
	\sum_{k \in\Lambda}
	\|\mathcal{R}\curl\curl\p_{t}(a_{(k)}g_{(k)}W_{(k)}^{c})\|_{L_{t}^{1}L_x^{\rho}}\notag\\
	&\lesssim
	\sum_{k \in \Lambda}
	(\|g_{(k)}\|_{L_{t}^{1}}\|a_{(k)}\|_{C_{t,x}^{2}}\|W_{(k)}^{c}\|_{W_{x}^{1,\rho}}
	+\|\p_{t}g_{(k)}\|_{L_{t}^{1}}\|a_{(k)}\|_{C_{t,x}^{1}}\|W_{(k)}^{c}\|_{W_{x}^{1,\rho}})\notag\\
	%&\lesssim\tau^{-\frac{1}{2}}(1+\|\Ru_{q}\|_{C_{t}L_{x}^{1}}^{4})\ell^{-19}r_{\perp}^{\frac{2}{\rho}-1}\lambda_{q+1}^{-1}+\sigma\tau^{\frac{1}{2}}(1+\|\Ru_{q}\|_{C_{t}L_{x}^{1}}^{3})\ell^{-13}r_{\perp}^{\frac{2}{\rho}-1}\lambda_{q+1}^{-1}\notag\\
	&\lesssim	(1+\|\Ru_{q}\|_{C_{t}L_{x}^{1}}^{4})\ell^{-13}r_{\perp}^{\frac{2}{\rho}-1}\lambda_{q+1}^{-1}\sigma\tau^{\frac{1}{2}},
\end{align*}
where the last is due to $\sigma\tau>\ell^{-6}$.

Thus, we get
\begin{align}\label{6-29}
	\|\mathcal{R}\p_{t}\omw_{q+1}^{(p)+(c)}\|_{L_{(\frac{\varsigma_{q}}{2}, T]}^{1}L_x^{\rho}}
	%&\lesssim(1+\|\Ru_{q}\|_{C_{t}L_{x}^{1}}^{4})\ell^{-7}\varsigma_{q}^{-1}r_{\perp}^{\frac{2}{\rho}-1}\tau^{-\frac{1}{2}}+(1+\|\Ru_{q}\|_{C_{t}L_{x}^{1}}^{4})\ell^{-13}r_{\perp}^{\frac{2}{\rho}-1}\lambda_{q+1}^{-1}\sigma\tau^{\frac{1}{2}}\notag\\
	&\lesssim
	(1+\|\Ru_{q}\|_{C_{t}L_{x}^{1}}^{4})\ell^{-13}r_{\perp}^{\frac{2}{\rho}-1}\lambda_{q+1}^{-1}\sigma\tau^{\frac{1}{2}}.
\end{align}

\medskip 
\paragraph{\bf  Control of hyper-viscosity. }
For the hyper-viscosity term,
using interpolation and Lemma \ref{Lemma Estimates of perturbations 2} we get
\begin{align*}
	\|\mathcal{R}\nu(-\Delta)^{\a}\omw_{q+1}^{(p)}\|_{L^{1}_{(\frac{\varsigma_{q}}{2}, T]}L_x^{\rho}}  
	&\lesssim \|\Theta_{q+1}\|_{C_{t}}\|w_{q+1}^{(p)}\|^{\frac{4-2\a}{3}}_{L_{t}^{1}L_x^{\rho}}\|w_{q+1}^{(p)}\|^{\frac{2\a-1}{3}}_{L_{t}^{1}W_x^{3,\rho}}\notag\\
	&\lesssim (1+\|\Ru_{q}\|_{C_{t}L_{x}^{1}}^{2\a+1})\ell^{-7}r_{\perp}^{\frac{2}{\rho}-1}\laq^{2\a-1}\tau^{-\frac{1}{2}}.
\end{align*}
Similarly, we have
\begin{align*}
	&\|\nu(-\Delta)^{\a}\omw_{q+1}^{(c)}\|_{L^{1}_{(\frac{\varsigma_{q}}{2}, T]}L_x^{\rho}}
	\lesssim (1+\|\Ru_{q}\|_{C_{t}L_{x}^{1}}^{2\a+3})\ell^{-13}r_{\perp}^{\frac{2}{\rho}-1}\laq^{2\a-2}\tau^{-\frac{1}{2}},\\
	&\|\mathcal{R}\nu(-\Delta)^{\a}\omw_{q+1}^{(o)}\|_{L^{1}_{(\frac{\varsigma_{q}}{2}, T]}L_x^{\rho}}
	\lesssim
	(1+\|\Ru_{q}\|_{C_{t}L_{x}^{1}}^{2\a+4})\ell^{-6(2\a-1)-20}\sigma^{-1}.
\end{align*}
Since $	\ell^{-7}r_{\perp}^{\frac{2}{\rho}-1}\laq^{2\a-1}\tau^{-\frac{1}{2}}=\ell^{-7}\lambda_{q+1}^{-6\varepsilon}<\ell^{-6(2\a-1)-20}\sigma^{-1}$, we thus get
\begin{align}\label{6-30}
	\|\mathcal{R}\nu(-\Delta)^{\a}w_{q+1}\|_{L^{1}_{(\frac{\varsigma_{q}}{2}, T]}L_x^{\rho}}
	&\lesssim (1+\|\Ru_{q}\|_{C_{t}L_{x}^{1}}^{2\a+4})\ell^{-6(2\a-1)-20}\sigma^{-1}.
\end{align}

\medskip 
\paragraph{\bf Control of noise term.}
The noise term can be estimated in the same manner as  \eqref{4-32}. Hence, we obtain
\begin{align}\label{6-31}
	\|\mathcal{R}\(z_{\ell}-z_{q+1}\)\|_{L_{(\frac{\varsigma_{q}}{2}, T]}^{1}L_x^{\rho}} 
	&\lesssim
	\ell^{\frac{1}{2}}(1+\varsigma_{q}^{-\frac{1}{2}})
	M+\la^{-15(1-\delta)}(\|Z\|_{C^{\frac{1}{2}-\delta}_{t} L_{x}^{2}}+\|Z\|_{C_{t} H_{x}^{1-\delta}}).
\end{align}

\medskip 
\paragraph{\bf Control of  remaining terms:}
For the remaining terms in  (\ref{6-2}), estimating as in (\ref{4-33}) and (\ref{4-35}), we have
\begin{align}\label{6-33}		
	\|v_{\ell}+z_{\ell})\mathring{\otimes}w_{q+1}+w_{q+1}\mathring{\otimes}( v_{\ell}+z_{\ell})\|_{L_{(\frac{\varsigma_{q}}{2}, T]}^{1}L_x^{\rho}}
	&\lesssim
	(\|v_{q}\|_{C_{(\frac{\varsigma_{q}}{2}-\ell, T]} L_{x}^{\infty}}+\|z_{q}\|_{C_{(\frac{\varsigma_{q}}{2}-\ell, T]}L_{x}^{\infty}})\|w_{q+1}\|_{L_{(\frac{\varsigma_{q}}{2}, T]}^{1}L_x^{\rho}} \notag\\
	&\lesssim
	(\|v_{q}\|_{C_{t,x}}
	+\varsigma_{q}^{-\frac{3}{4\a}}M
	+\|Z_{q}\|_{C_{t}L_{x}^{\infty} })
	(1+\|\Ru_{q}\|_{C_{t}L_{x}^{1}}^{5})\ell^{-20}\sigma^{-1}.
\end{align}

Now,  we obtain from  (\ref{6-29})-(\ref{6-33}) that
\begin{align*}
	\|\Ru_{lin}\|_{L_{(\frac{\varsigma_{q}}{2}, T]}^{1}L_x^{1}}
	&\lesssim
	(1+\|\Ru_{q}\|_{C_{t}L_{x}^{1}}^{4})\ell^{-13}r_{\perp}^{\frac{2}{\rho}-1}\lambda_{q+1}^{-1}\sigma\tau^{\frac{1}{2}}
	+(1+\|\Ru_{q}\|_{C_{t}L_{x}^{1}}^{8})\ell^{-38}\sigma^{-1} \\
	&\quad
	+\ell^{\frac{1}{2}}(1+\varsigma_{q}^{-\frac{1}{2}})
	M+\la^{-15(1-\delta)}(\|Z\|_{C^{\frac{1}{2}-\delta}_{t} L_{x}^{2}}+\|Z\|_{C_{t} H_{x}^{1-\delta}})\\
	&\quad+	(\|v_{q}\|_{C_{t,x}}+\varsigma_{q}^{-\frac{3}{4\a}}M+\|Z_{q}\|_{C_{t}L_{x}^{\infty} })(1+\|\Ru_{q}\|_{C_{t}L_{x}^{1}}^{5})\ell^{-20}\sigma^{-1},
\end{align*}
which along with (\ref{2-7}),  (\ref{2-14}), (\ref{2-15}), (\ref{2-22})  and (\ref{5-1})  yields
\begin{align}\label{6-36}
	\|\Ru_{lin}\|_{\Lo^{r}_{\om} L_{(\frac{\varsigma_{q}}{2}, T]}^{1}L_{x}^{1}}
	&\lesssim
	(1+\|\Ru_{q}\|_{\Lo^{4r}_{\om}C_{t}L_{x}^{1}}^{4})\ell^{-13}r_{\perp}^{\frac{2}{\rho}-1}\lambda_{q+1}^{-1}\sigma\tau^{\frac{1}{2}}
	+(1+\|\Ru_{q}\|_{\Lo^{8r}_{\om}C_{t}L_{x}^{1}}^{8})\ell^{-38}\sigma^{-1} \notag\\
	&\quad
	+\ell^{\frac{1}{2}}(1+\varsigma_{q}^{-\frac{1}{2}})M
	+\la^{-15(1-\delta)}(\|Z\|_{\Lo^{2r}_{\om}C_{t}^{1/2-\delta}L_{x}^{2} }+\|Z\|_{\Lo^{2r}_{\om}C_{t}H_{x}^{1-\delta}})\notag\\
	&\quad
	+(\|v_{q}\|_{\Lo^{2r}_{\om}C_{t,x}}+\varsigma_{q}^{-\frac{3}{4\a}}M+\|Z_{q}\|_{\Lo^{2r}_{\om}C_{t}L_{x}^{\infty}} )
	(1+\|\Ru_{q}\|_{\Lo^{10r}_{\om}C_{t}L_{x}^{1}}^{5})\ell^{-20}\sigma^{-1}\notag\\
	%&\lesssim\big(\lambda_{q}^{(4\a+12)}(8\cdot 4rL^{2}50^{q})^{50^{q}}\big)^{4}\ell^{-13}\lambda_{q+1}^{-4\varepsilon}+\big(\lambda_{q}^{(4\a+12)}(8\cdot 8rL^{2}50^{q})^{50^{q}}\big)^{8}\ell^{-38}\lambda_{q+1}^{-2\varepsilon}+\ell^{\frac{1}{2}}\varsigma_{q}^{-\frac{1}{2}}+\la^{-14}(2r-1)^{\frac{1}{2}}L\notag\\
%&\quad+\big(\lambda_{q}^{(2\a+5)}(8\cdot 20rL^{2}50^{q-1})^{10\cdot50^{q-1} }+\varsigma_{q}^{-\frac{3}{4\a}}+(2r-1)^{\frac{1}{2}}L\big)\big(\lambda_{q}^{(4\a+12)}(8\cdot 10rL^{2}50^{q})^{50^{q}}\big)^{5}\ell^{-20}\lambda_{q+1}^{-2\varepsilon}\notag\\
	%&\lesssim\lambda_{q}^{-14}+\lambda_{q+1}^{-\varepsilon}
	&\lesssim
	\dqqq.
\end{align}

\subsubsection{Oscillation error}
By (\ref{3-49}), (\ref{3-50}), (\ref{5-21})  and (\ref{6-8}),  we get
\begin{align}\label{6-37-1}
	\|\Ru_{osc.1}\|_{L_{(\frac{\varsigma_{q}}{2}, T]}^{1}L_{x}^{1}}
	&\lesssim
	\|1-\Theta_{q+1}^{2}\|_{L_{t}^{1}}\|\Ru_{q}\|_{C_{t}L_{x}^{1}}+\|\Theta_{q+1}^{2}\|_{C_{t}^{1}}\|w^{(o)}_{q+1}\|_{L_{t}^{1} L_{x}^{\rho}}\notag\\
	&\lesssim
	\varsigma_{q}\|\Ru_{q}\|_{C_{t}L_{x}^{1}}+(1+\|\Ru_{q}\|_{C_{t}L_{x}^{1}}^{5})\ell^{-20}\sigma^{-1}\varsigma_{q}^{-1}.
\end{align}
Moreover, by the decoupling Lemma \ref{Lemma Decorrelation1} and Lemmas \ref{Lemma temporal building blocks}, \ref{Lemma amplitudes 1}  and \ref{Lemma spatial building blocks2}, 
\begin{align}\label{6-37-2}
	\|\Ru_{osc.2}\|_{L_{(\frac{\varsigma_{q}}{2}, T]}^{1} L_{x}^{\rho}}
	&\lesssim
	\sum_{k\in\Lambda}\|\Theta_{q+1}^{2}\|_{C_{t}}\|g_{(k)}^{2}\|_{L_t^{1}}
	\||\nabla|^{-1}\mathbb{P}_{\neq 0}\big(\mathbb{P}_{\geq \frac{\laq r_{\perp}}{2} }(W_{(k)}\otimes W_{(k)})\big)\nabla(a_{(k)}^{2})\|_{C_{t} L_{x}^{\rho}}\notag\\
	&\lesssim
	\sum_{k\in\Lambda}\|a_{(k)}^{2}\|_{C_{t,x}^{3}}\(\laq r_{\perp}\)^{-1}\|\phi_{(k)}^{2}\|_{L_{x}^{\rho}}\notag\\
	&\lesssim
	(1+\|\Ru_{q}\|_{C_{t}L_{x}^{1}}^{7})\ell^{-32}r_{\perp}^{\frac{2}{\rho}-3}\lambda_{q+1}^{-1}.
\end{align}
By Lemmas \ref{Lemma temporal building blocks}, \ref{Lemma spatial building blocks2}, \ref{Lemma amplitudes 1}   and (\ref{6-10}), we also have
\begin{align}\label{6-37-3}
	\|\Ru_{osc.3}\|_{L_{(\frac{\varsigma_{q}}{2}, T]}^{1} L_{x}^{\rho}}
	&\lesssim
	\sigma^{-1}\|\Theta_{q+1}^{2}\|_{C_{t}}\sum_{k\in\Lambda}\|h_{(k)}\|_{L_t^{1}}
	\|a_{(k)}^{2}\|_{C_{t,x}^{2}}
	\lesssim
	(1+\|\Ru_{q}\|_{C_{t}L_{x}^{1}}^{6})\ell^{-26}\sigma^{-1}.
\end{align}
Thus, combining \eqref{6-37-1}-\eqref{6-37-3} altogether, we obtain that
\begin{align*}
	\|\Ru_{osc}\|_{L_{(\frac{\varsigma_{q}}{2}, T]}^{1}L_x^{1}}
	%&\lesssim\varsigma_{q}\|\Ru_{q}\|_{C_{t}L_{x}^{1}}+(1+\|\Ru_{q}\|_{C_{t}L_{x}^{1}}^{5})\ell^{-20}\sigma^{-1}\varsigma_{q}^{-1}	+(1+\|\Ru_{q}\|_{C_{t}L_{x}^{1}}^{7})\ell^{-32}r_{\perp}^{\frac{2}{\rho}-3}\lambda_{q+1}^{-1}\notag\\
%&\quad	+(1+\|\Ru_{q}\|_{C_{t}L_{x}^{1}}^{6})\ell^{-26}\sigma^{-1}\notag\\
	&\lesssim
	\varsigma_{q}\|\Ru_{q}\|_{C_{t}L_{x}^{1}}+(1+\|\Ru_{q}\|_{C_{t}L_{x}^{1}}^{7})\ell^{-26}\sigma^{-1}\notag,
\end{align*}
which, via (\ref{2-7}), (\ref{2-15}) and \eqref{5-1}, yields that
\begin{align}\label{6-37}
	\|\Ru_{osc}\|_{\Lo^{r}_{\om}L_{(\frac{\varsigma_{q}}{2}, T]}^{1}L_{x}^{1}}
	&\lesssim
	\varsigma_{q}\|\Ru_{q}\|_{\Lo^{r}_{\om}C_{t}L_{x}^{1}}
	+(1+\|\Ru_{q}\|_{\Lo^{7r}_{\om}C_{t}L_{x}^{1}}^{7})\ell^{-26}\sigma^{-1}
	%&\lesssim\varsigma_{q}(\lambda_{q}^{(4\a+12)}(8\cdot rL^{2}50^{q})^{50^{q}})+(\lambda_{q}^{(4\a+12)}(8\cdot 7rL^{2}50^{q})^{50^{q}})^{7}\ell^{-26}\lambda_{q+1}^{-2\varepsilon}\notag\\
	%\lesssim \lambda_{q}^{-19}+\lambda_{q+1}^{-\varepsilon}
	\lesssim
	\dqqq.
\end{align}

\subsubsection{Corrector error}
By (\ref{3-49}), \eqref{5-15},  (\ref{6-3}) and Lemma \ref{Lemma Estimates of perturbations 2},
\begin{align*}
	\|\Ru_{corr}\|_{L_{(\frac{\varsigma_{q}}{2}, T]}^{1}L_x^{1}}
	&\lesssim
	\|\omw_{q+1}^{(c)+(o)}\|_{L_t^{2}L_x^{2}}	(\|w_{q+1}\|_{L_t^{2}L_x^{2}}+\|\omw_{q+1}^{(p)}\|_{L_t^{2}L_x^{2}})
	%&\lesssim(1+\|\Ru_{q}\|_{C_{t}L_{x}^{1}}^{5})(\ell^{-7}\lambda_{q+1}^{-1}+\ell^{-20}\sigma^{-1})\notag\\
	%&\quad\times(1+\|\Ru_{q}\|_{C_{t}L_{x}^{1}}^{5})(\ell^{-7}+\ell^{-7}\lambda_{q+1}^{-1}+\ell^{-20}\sigma^{-1})\notag\\
	\lesssim (1+\|\Ru_{q}\|_{C_{t}L_{x}^{1}}^{10})\ell^{-27}\sigma^{-1}.
\end{align*}
Taking into account  (\ref{2-7}), (\ref{2-15}) and \eqref{5-1}  we get
\begin{align}\label{6-38}
	\|\Ru_{corr}\|_{\Lo^{r}_{\om}L_{(\frac{\varsigma_{q}}{2}, T]}^{1}L_{x}^{1}}
	&\lesssim
	(1+\|\Ru_{q}\|_{\Lo^{10r}_{\om}C_{t}L_{x}^{1}}^{10})\ell^{-27}\sigma^{-1}
	%&\lesssim\big(\lambda_{q}^{(4\a+12)}(8\cdot 10rL^{2}50^{q})^{50^{q}}\big)^{10}\ell^{-27}\laq^{-2\varepsilon}\notag\\
	%&\lesssim
	%\laq^{-\varepsilon}
	\lesssim
	\dqqq.
\end{align}

\subsubsection{ Commutator errors}\label{subsec 6.3.4}
The estimate of  $\|\Ru_{com1}\|_{\Lo^{r}_{\om}L_{(\frac{\varsigma_{q}}{2}, T]}^{1}L_{x}^{1}}$ is same as that in (\ref{4-39}). Thus, we have
\begin{align*}%\label{5.79}
	\|\Ru_{com1}\|_{L_{(\frac{\varsigma_{q}}{2}, T]}^{1}L_x^{1}}
	&\lesssim
	\ell^{\frac{1}{2}-\delta}	\big(\|v_{q}\|_{C_{t,x}^{1}}+	(1+\varsigma_{q}^{-(\frac{1}{2}-\delta)})M
	+\|Z\|_{C_{t}^{\frac{1}{2}-\delta} L_{x}^{2}}+\|Z\|_{C_{t} H_{x}^{1-\delta}}\big)
	(\|v_{q}\|_{C_{t,x}}+\|z\|_{C_{t} L_{x}^{2}}),
\end{align*}
and 
\begin{align}\label{6-39}
	\|\Ru_{com1}\|_{\Lo^{r}_{\om}L_{(\frac{\varsigma_{q}}{2}, T]}^{1}L_{x}^{1}}
	\lesssim
	%\lambda_{q}^{-9}
%	\lesssim
	\dqqq.		
\end{align}

Concerning the $L_{(\frac{\varsigma_{q}}{2}, T]}^{1}L_x^{1}$-norm of $ \Ru_{com2} $,
by \eqref{5-18}, (\ref{5-39}), (\ref{6-4}),  (\ref{6-31}) and mollification estimate,
\begin{align*}
	\|\Ru_{com2}\|_{L_{(\frac{\varsigma_{q}}{2}, T]}^{1}L_x^{1}}
	&\lesssim
	\big(\|v_{q}\|_{C_{t,x}}
	+(1+\|\Ru_{q}\|_{C_{t}L_{x}^{1}}^{5})\ell^{-20}\sigma^{-1}+\|z\|_{C_{(\frac{\varsigma_{q}}{2}-\ell, T]}L_{x}^{2}}\big)\notag\\
	&\quad\times \big(\ell^{\frac{1}{2}}(1+\varsigma_{q}^{-\frac{1}{2}})M+\la^{-15(1-\delta)}(\|Z\|_{C_{t}^{\frac{1}{2}-\delta} L_{x}^{2}}+\|Z\|_{C_{t} H_{x}^{1-\delta}})\big),
\end{align*}
which, via (\ref{stochastic evolution}), (\ref{2-7}), (\ref{2-14}), (\ref{2-15}) and (\ref{5-1}), yields that
\begin{align}\label{6-40}
	\|\Ru_{com2}\|_{\Lo^{r}_{\om}L_{(\frac{\varsigma_{q}}{2}, T]}^{1}L_{x}^{1}}
	&\lesssim
	\big(\|v_{q}\|_{\Lo^{2r}_{\om}C_{t,x}}+(1+\|\Ru_{q}\|_{\Lo^{10r}_{\om}C_{t}L_{x}^{1}}^{5})\ell^{-20}\sigma^{-1}+\|z\|_{\Lo^{2r}_{\om}C_{(\frac{\varsigma_{q}}{2}-\ell, T]} L_{x}^{2}}\big)\notag\\
	&\quad\times
	\big(\ell^{\frac{1}{2}}(1+\varsigma_{q}^{-\frac{1}{2}})M+\la^{-15(1-\delta)}(\|Z\|_{\Lo^{2r}_{\om}C_{t}^{\frac{1}{2}-\delta}L_{x}^{2}}+\|Z\|_{\Lo^{2r}_{\om}C_{t}H_{x}^{1-\delta}})\big)\notag\\
	%&\lesssim\( \la^{(2\a+5)}(8\cdot   20rL^{2}50^{q-1})^{10\cdot50^{q-1}}+(\lambda_{q}^{(4\a+12)}(8\cdot10rL^{2}50^{q})^{50^{q}})^{5}\ell^{-20}\laq^{-2\varepsilon}+(2r-1)^{\frac{1}{2}}L+M\)\notag\\
	%&\quad\times\big(\ell^{\frac{1}{2}}(1+\varsigma_{q}^{-\frac{1}{2}})M+\la^{-15(1-\delta)}(2r-1)^{\frac{1}{2}}L\big)\notag\\
	%&\lesssim
%	\lambda_{q}^{-4}
	&\lesssim
	\dqqq.
\end{align}

%Therefore,  combining (\ref{4-36}), (\ref{4-37}), (\ref{4-38}), (\ref{4-39}) and (\ref{4-40}),  we conclude that
Therefore,  combining (\ref{6-36}), and \eqref{6-37}-(\ref{6-40}) altogether,  we conclude that
\begin{equation}\label{6-41}
	\begin{aligned}
		\|\Ru_{q+1}\|_{\Lo^{r}_{\om}L_{(\frac{\varsigma_{q}}{2}, T]}^{1}L_{x}^{1}}\lesssim \dqqq.
\end{aligned}\end{equation}

\subsection{Verification of inductive estimate \eqref{2-16}: near the initial time} For $t\in[0,\frac{\varsigma_{q}}{2}]$,  since $\Theta_{q+1}(t)=0$ and  $w_{q+1}=0$,  we have the same estimate as that in \eqref{4-42-1} and \eqref{4-42-2} for $t\in[0,\frac{\varsigma_{q}}{2}]$. Thus, we have
\begin{equation}\label{6-42}
	\begin{aligned}
		\|\Ru_{q+1}\|_{\Lo^{r}_{\om}L_{t}^{1}L_{x}^{1}}\lesssim \dqqq.
\end{aligned}\end{equation}
As a consequence, we prove the decay estimate \eqref{2-16} for $\Ru_{q+1}$ at level $q+1$. 

Finally,  
we infer from Sections \ref{Sec-S2} and \ref{Sec-Reynolds 2} that the main iteration in Proposition \ref{Proposition Main iteration}, and so Theorems \ref{Thm-Nonuniq-Hyper} and \ref{Theorem Vanishing noise}, 
hold in the supercritical regime $\mathcal{S}_{2}$.

\medskip 
\noindent{\bf Acknowledgment.}
Z. Zeng is supported by Natural Science Foundation of Jiangsu Province 
 (No. SBK20240 43113). D. Zhang is partially supported by the NSFC grants (No. 12271352, 12322108) and
Shanghai Frontiers Science Center of Modern Analysis.  
The authors also thank 
the supports by 
Shanghai Frontiers Science Center of Modern Analysis.

 \end{document}